\documentclass[11 pt,reqno]{amsart}
\usepackage{amsmath,amsfonts,amssymb,mathrsfs,bm,graphicx,stmaryrd,bbm}
\usepackage{amsthm,mathtools}

\usepackage[usenames,dvipsnames]{color}
\usepackage{epstopdf}

\usepackage{longtable}
\usepackage{enumitem}
\usepackage{subcaption}
\DeclareCaptionLabelFormat{period}{#2.}
\usepackage{versions}

\usepackage[colorlinks=true,linkcolor=blue, bookmarksdepth=3]{hyperref}

\usepackage[letterpaper,hmargin=1.0in,vmargin=1.0in]{geometry}

\parskip 	\smallskipamount

\allowdisplaybreaks

\newcommand{\hair}{\ifmmode\mskip1mu\else\kern0.08em\fi}
\renewcommand{\P}{\mathbb{P}}

\newcommand{\Var}{\mathrm{Var}}
\newcommand{\E}{\mathbb{E}}

\renewcommand{\L}{\mathbb{L}}

\newcommand{\R}{\mathbb{R}}

\newcommand{\ce}{\mathcal{E}}
\newcommand{\N}{\mathbb{N}}
\newcommand{\Z}{\mathbb{Z}}

\newcommand{\Exp}{\mathrm{Exp}}

\newcommand{\intint}[1]{\llbracket 1,#1 \rrbracket}
\renewcommand{\epsilon}{\varepsilon}

\newtheorem{maintheorem}{Theorem}

\newtheorem{theorem}{Theorem}[section]
\newtheorem*{theorem*}{Theorem}

\newtheorem*{prop*}{Proposition}
\newtheorem{prop}[theorem]{Proposition}
\newtheorem*{corollary*}{Corollary}
\newtheorem{corollary}[theorem]{Corollary}
\newtheorem{lemma}[theorem]{Lemma}
\newtheorem{proposition}[theorem]{Proposition}

\theoremstyle{definition}
\newtheorem{defn}[theorem]{Definition}

\newtheorem{remark}[theorem]{Remark}

\newcommand{\melonweight}{X_{n}^k}
\newcommand{\weight}{\ell}
\DeclareMathOperator*{\tf}{TF}
\newcommand{\linelower}{\mathbb{L}_{\mathrm{low}}}
\newcommand{\lineupper}{\mathbb{L}_{\mathrm{up}}}
\newcommand{\melon}[1]{\gamma^{#1}_{n}}
\newcommand{\pos}[2]{p_{#1,k}^{(#2)}}

\newcommand{\pathweight}[3]{X_{n,k,#3}^{(#1,#2,s)}}
\newcommand{\pathweightnos}[3]{X_{n,k,#3}^{(#1,#2)}}

\newcommand{\sep}[1]{\mathrm{sep}_{k}^{(#1)}}
\newcommand{\paragram}[2]{F^{(#1, #2, s)}}

\newcommand{\Xmid}{X^{r,s, t}_{\mathrm{mid}}}
\newcommand{\lineleft}{\mathbb{L}_{\mathrm{left}}}
\newcommand{\lineright}{\mathbb{L}_{\mathrm{right}}}

\newcommand{\pleft}{p_{\mathrm{left}}}
\newcommand{\pright}{p_{\mathrm{right}}}
\newcommand{\ordmel}[2]{\smash{\gamma_{n, \shortrightarrow}^{#1, #2}}}
\newcommand{\tordmel}[2]{\gamma_{n, \shortrightarrow}^{#1, #2}}
\newcommand{\pordmel}[2]{\smash{\phi_{n, \shortrightarrow}^{#1, #2}}}
\newcommand{\philower}{\phi_{\mathrm{low}}}
\newcommand{\phiupper}{\phi_{\mathrm{up}}}
\newcommand{\mc}{\mathcal}

\newcommand{\WPF}[2]{\smash{\mathrm{WE}_{#2}^{#1}}}
\newcommand{\ar}{\to}

\title[The geodesic watermelon in last passage percolation]{Interlacing and scaling exponents for the geodesic watermelon in last passage percolation}
\author{Riddhipratim Basu}
\address{Riddhipratim Basu, International Centre for Theoretical Sciences, Tata Institute of Fundamental Research, Bangalore, India}
\email{rbasu@icts.res.in}
\author{Shirshendu Ganguly}
\address{Shirshendu Ganguly, Department of Statistics, U.C. Berkeley, Berkeley, CA, USA}
\email{sganguly@berkeley.edu}
\author{Alan Hammond}
\address{Alan Hammond, Departments of Mathematics and Statistics, U.C. Berkeley, Berkeley, CA, USA}
\email{alanmh@berkeley.edu}
\author{Milind Hegde}
\address{Milind Hegde, Department of Mathematics, U.C. Berkeley, Berkeley, CA, USA}
\email{milind.hegde@berkeley.edu}

\excludeversion{submitted-version}
\includeversion{arxiv-version}

\begin{document}

\begin{abstract}
In a discrete planar last passage percolation (LPP), 
random values are assigned independently to each vertex in $\Z^2$, and 
each finite upright path in $\Z^2$ is ascribed the weight given by the sum of values attached to the vertices of the path. 
The weight of a collection of disjoint paths is the sum of its members’ weights. The notion of a geodesic, namely a path of maximum weight between two vertices, has a natural generalization concerning several disjoint paths.
Indeed, a $k$-geodesic watermelon  in $[1,n]^2\cap \Z^2$ is a collection of $k$ disjoint upright paths contained in this square that has maximum weight among all such collections. While the weights of such collections are known to be important objects, the maximizing paths have remained largely unexplored beyond the $k=1$ case. 
For exactly solvable models, such as exponential and geometric LPP, it is well known 
that for $k=1$ the exponents that govern fluctuation in weight  and  transversal distance are $1/3$ and $2/3$; which is to say, the weight of the geodesic on the route $(1,1) \to (n,n)$ typically fluctuates  around  a dominant linear growth of the form $\mu n$ by the order of~$n^{1/3}$; and the maximum Euclidean distance of the geodesic from the diagonal typically has order $n^{2/3}$.
Assuming a strong but local form of convexity  and
one-point moderate deviation estimates for the geodesic weight profile---which are available in all known exactly solvable models---we establish that, typically, the $k$-geodesic watermelon's weight falls below $\mu n k$ by order $k^{5/3}n^{1/3}$, and its transversal fluctuation is of order $k^{1/3}n^{2/3}$. Our arguments crucially rely on, and develop, a remarkable deterministic interlacing property that the watermelons admit. Our methods also yield sharp rigidity estimates for naturally associated point processes. These bounds improve on estimates obtained by applying tools from the theory of determinantal point processes available in the integrable setting. 
\end{abstract}
\maketitle

\begin{figure}[h]
\includegraphics[width=0.43\linewidth]{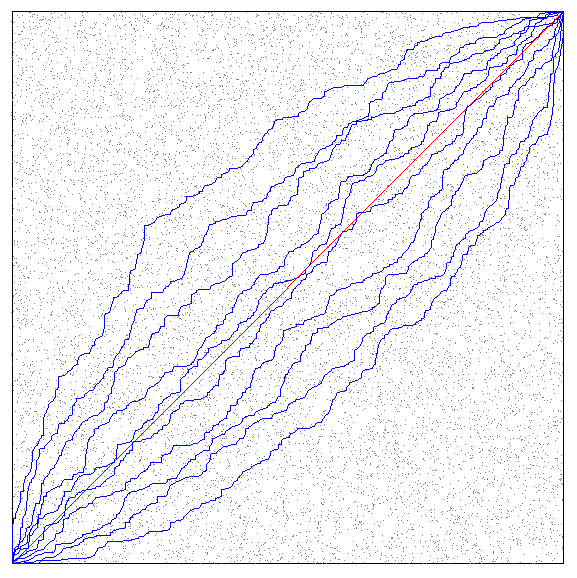}
\hspace{0.2cm}
\includegraphics[width=0.43\linewidth]{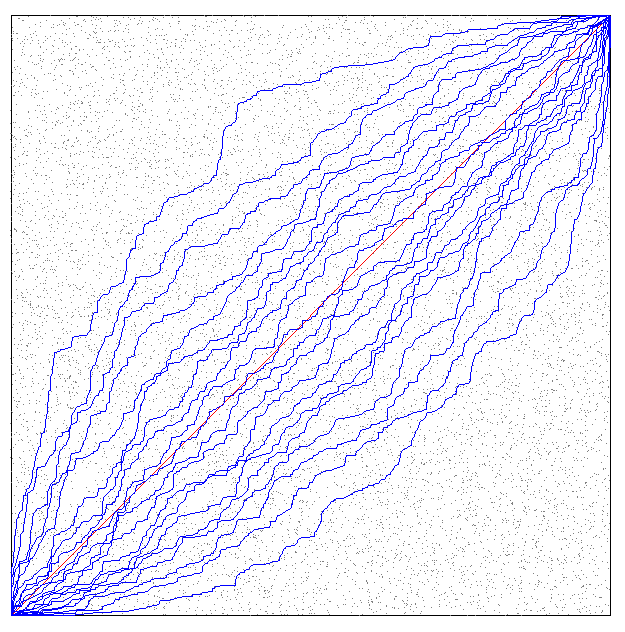}
\caption{Simulations of $k$-geodesic watermelons in Poissonian LPP; $k$ equals ten on the left and twenty on the right, with $n=200$ in both. Grey are the points of the underlying Poisson process of intensity one on $[0,n]^2$; blue are the watermelon's curves; and red is the diagonal $y=x$ around which the watermelon fluctuates on the scale $k^{1/3}n^{2/3}$.}%
\label{f.watermelon}
\end{figure}

\tableofcontents

\section{Introduction and main results}\label{intro12}

In a discrete planar last passage percolation (LPP) model, each vertex in $\Z^2$ is independently ascribed a non-negative value sampled from a given law $\nu$. An upright path in $\Z^2$ is a finite nearest-neighbour path on $\Z^2$ each of whose steps is upwards or to the right. To each such path a weight is ascribed by the LPP model, this being the sum of the values assigned to the vertices in $\Z^2$ visited by the path. For $n \in \N$, let $X_n$ denote the maximum weight assigned to an upright path that begins at the vertex $(1,1)$ and that ends at the vertex $(n,n)$. Any upright path with these endpoints whose weight equals $X_n$ is called a geodesic between $(1,1)$ and $(n,n)$, and $X_n$ is called the last passage time from $(1,1)$ to $(n,n)$; these two concepts make sense for  any pair of coordinate-wise ordered vertices.  Let the transversal fluctuation $\tf(n)$ denote the maximum anti-diagonal fluctuation of any geodesic between $(1,1)$ and $(n,n)$. That is, $\tf(n)$ equals the minimum of positive real $r$ such that every vertex of any geodesic between $(1,1)$ and $(n,n)$ has Euclidean distance from the diagonal $y = x$ at most $r$. Planar LPP models are paradigmatic examples of models predicted to exhibit features of the Kardar-Parisi-Zhang (KPZ) universality class; in particular, the characteristic KPZ exponents of one-third and two-thirds for the weight and transversal fluctuations of geodesics.

A few planar LPP models---for which the law $\nu$ has a special form---enjoy {\em integrable} properties that vastly facilitate their analysis.  Two such models  are geometric LPP with parameter $p \in (0,1)$, where $\nu(k) = p^{k-1} (1-p)$ for $k \in \N$; and exponential LPP, where $\nu$ is the exponential law. (By scaling properties of the exponential distribution, the rate of the exponential distribution is irrelevant, and we will consider exponentials of rate one.) For any non-negative $\nu$,  
the existence of the limiting growth rate $\mu$ for the weight maximum, specified as the almost sure limiting value of $n^{-1} X_n$ as $n \to \infty$, is a consequence of a straightforward superadditivity argument; the positivity of $\mu$ holds as soon as $\nu$ is not degenerate at $0$, while the finiteness is guaranteed under mild moment assumptions. For geometric and exponential LPP, $\mu$ can be explicitly evaluated. The integrable structure of these and a handful of other models has been crucial in proofs establishing the pair of KPZ scaling exponents.

Incidentally, the first model for which such a program was carried out was  Poissonian last passage percolation on the plane, where the statistic of interest $L_{n}$ is the maximum number of points on an upright path from $(0,0)$ to $(n,n)$ in a rate one Poisson process on $\R^2$; conditionally on the total number of points in $[0,n]^2$, this is Ulam's problem for the longest increasing subsequence in a uniform random permutation. In their seminal work~\cite{baik},  Baik, Deift, and Johansson showed that $n^{-1/3}(L_{n}-2n)$ converges weakly to the GUE Tracy-Widom distribution, namely the high-$n$ limiting law of a scaled version of the largest eigenvalue of an $n\times n$ random matrix picked according to the Gaussian unitary ensemble. Shortly thereafter, Johansson \cite{J00} used this longitudinal fluctuation result to establish the transversal fluctuation exponent of two-thirds in this model, showing that with high probability the smallest strip around the diagonal containing any geodesic has width $n^{2/3+o(1)}$. Similar transversal fluctuation results have been proved for last passage percolation on~$\Z^2$ with exponential and geometric passage times \cite{balazs2006}; and for the semi-discrete model of Brownian last passage percolation \cite{brownianLPPtransversal}.

The Baik-Deift-Johansson theorem was derived by noting that the random statistic $L_n$ has the distribution of  the length of the top row of a Young tableau picked according to the Poissonized Plancherel measure of appropriate parameters. The latter observation was obtained via the Robinson-Schensted-Knuth correspondence in \cite{schensted,knuth1970permutations}
and was first  exploited in \cite{logan26variational,vershik1977asymptotics}.
The correspondence extends to other rows of the tableau. Indeed, Greene \cite{greene} established that the sum of the lengths of the first $k$ rows of the random Young tableau picked from the same measure has the distribution of the maximum number of Poisson points on $k$ upright paths from $(0,0)$ to $(n,n)$ that are disjoint except at these shared endpoints. Baik, Deift and Johansson conjectured that the scaled lengths of the top $k$ rows converge jointly in distribution to the top $k$ points of a determinantal point process on $\R$
called the Airy point process.
The conjecture was proved soon after, independently, by Borodin, Okounkov and Olshanski \cite{borodin2000asymptotics}; Johansson \cite{johansson-toprows}; and Okounkov \cite{okounkov2000random}. By Greene's theorem, then, the highest scaled weight of a set of $k$ disjoint upright paths in Poissonian LPP converges in law to the sum of the top $k$ points of the Airy point process. Other integrable LPP models exhibit a similar correspondence; we discuss it in the case of exponential LPP in Section~\ref{s:special}. 

We have recounted these fragments of KPZ history to advocate the conceptual importance of systems of $k$ upright disjoint paths with given endpoints that maximize collective weight. We call them $k$-\emph{geodesic watermelons} and devote this article to a unified geometric treatment of them.

The parameter $k \in \N$ will be positive, with $k \leq n$. 
Define $X_n^k$ to be the maximum weight of any collection of $k$ disjoint upright paths contained inside the square with opposite corners $(1,1)$ and $(n,n)$; note that the collection attaining the maximum need not be unique, and that $X_n^1 = X_n$. The transversal fluctuation of any such collection of paths  is the maximum Euclidean distance between a vertex in~$\Z^2$ lying in the paths' union and the diagonal $y =x$. We specify $\overline\tf(n,k)$ to equal the maximum value of the transversal fluctuation over such sets of $k$ paths whose weight realises~$X_n^k$, and $\underline\tf(n,k)$ to be the minimum value of the transversal fluctuation over the same collection.  

Our main result establishes the values of the exponents that govern  weight and transversal fluctuations of $k$-geodesic watermelons in the context of geometric and exponential LPP. We note a simple heuristic to predict the exponents by considering the case that $k=n$, where the $k$-geodesic watermelon uses all the vertices in $\intint{n}^2$: then by the strong law of large numbers, $X_n^k$ will be of order $\E[\xi]n^2$, with $\xi$ distributed according to $\nu$, and $\underline\tf(n,k) = \overline\tf(n,k) = \Theta(n)$. It is reasonable to believe that $X_n^k$ to first-order should be linear in $nk$, just as $X_n$ is to first-order linear in $n$. Then if we assume that the fluctuations are governed by exponents $\alpha$ and $\beta$ as $X_{n}^k = \mu n k - \Theta(k^{\alpha}n^{1/3})$ and $\underline\tf(n,k) = \overline\tf(n,k) = \Theta(k^\beta n^{2/3})$, we obtain, from the $k=n$ case, the prediction that $\alpha = 5/3$ and $\beta = 1/3$. This is what our main result establishes with high probability for all $k$ up to a small constant times $n$.
\newpage

\begin{maintheorem}
\label{t.notwidenothin}
Consider geometric LPP of given parameter $p \in (0,1)$, or exponential LPP. 

\begin{enumerate} 
\item \label{t.weightfluc exp geo}
{\em Weight fluctuation:} There exist positive constants $C_1$, $C_2$, $C$, $c_1$ and $c$  such that
\begin{equation*}
 \P\left(X_{n}^k - \mu n k \notin - k^{5/3}n^{1/3} \cdot \big(C_1 ,C_2 \big) \right)\leq Ce^{-ck^2} 
\end{equation*}
 for $1\le k\leq c_1 n$,  
where we denote $a \cdot I = \big\{ ax: x \in I \big\}$ for $a \in \R$ and $I \subset \R$.
\item \label{tf} {\em Transversal fluctuation:} 
\begin{enumerate}
	\item \label{tf1} There exist positive constants $M$, $C$, $c$ and $c_1$  such that, for $1 \leq k \leq c_1n$,
	\begin{equation*}
	\P\left(\overline\tf(n,k) > Mk^{1/3}n^{2/3}\right) \leq Ce^{-ck}.
 	\end{equation*}
	
	\item \label{tf2} A matching lower bound holds: there exist positive constants $C$, $c$, $c_1$ and $\delta$ such that, for $1 \leq k \leq c_1 n$,
	\begin{equation*}
	\P\left(\underline\tf(n,k) < \delta k^{1/3}n^{2/3}\right) \leq Ce^{-ck^2}.
	\end{equation*}
\end{enumerate}
\end{enumerate}
\end{maintheorem}

Since the fundamental advances \cite{baik,J00}, much analysis of integrable LPP models has been based on exact formulas for the point-to-point last passage time and for finite-dimensional distributions of the passage time profile from a point to certain special lines \cite{borodin2008large,borodin2008large2}. These integrable techniques extend to the continuous scaling limits, such as the KPZ fixed point \cite{matetski2016kpz}. %
More recently, 
probabilistic and geometric technique in alliance with integrable input has been brought to bear on LPP problems.
 The Brownian Gibbs property is a resampling invariance enjoyed by the random ensemble of curves associated to Brownian LPP by the RSK correspondence.  
 It has been used in~\cite{brownianLPPtransversal} to analyse the weight of the $k$-geodesic watermelon  on the route from $(0,0)$ to $(n,n)$ in Brownian LPP, and in \cite{calvert2019brownian} to gain strong control on LPP weight profiles and the Airy$_2$ process. 
The resampling property is also a central tool in the recent advance of \cite{dauvergne2018directed}, which, with the aid of \cite{dauvergne2018basic}, constructs the full scaling limit of Brownian LPP. Similar Gibbs properties in other models have been explored in works such as \cite{corwin2018transversal,corwin2018kpz,aggarwal2019universality,wu2019tightness}.

The present paper falls within the scope of a separate program of probabilistic and geometric inquiry into KPZ, focused on exponential and Poissonian LPP.  By exploiting moderate deviation estimates from integrable probability and aspects of geodesic geometry, the slow bond problem was solved in the preprint~\cite{slow-bond}.  Developing this vein,  \cite{coalescence, hammond2018modulus, basu2018time} and \cite{zhang2019optimal} have offered information about coalescence structure of geodesics; their local fluctuations; and temporal correlation exponents. A combination of geometric and integrable methods, including the use of one-point estimates, has also been crucially used in \cite{ferrari2018universality,ferrari2019time} to establish universality of the GOE Tracy-Widom distribution in point-to-line LPP with general slope and time correlation exponents with generic initial conditions.

This paper pursues the preceding geometric and probabilistic program while adopting a novel geometric perspective: $k$-geodesic watermelons interlace, each with the next as the parameter $k \in \N$ rises; as we will explain in Section~\ref{s:interlace}, this property is a tool that governs our ideas and the proofs of our results. The technique is robust. Although our main theorem addresses geometric and exponential LPP,  its derivation makes very limited use of integrable inputs, holding sway under weak assumptions. Indeed, Theorem~\ref{t.notwidenothin} follows directly from a more general result, Theorem~\ref{t.notwidenotthingeneral}, valid under a rather natural set of assumptions that all known integrable models satisfy, which we state next. 

\subsection{Assumptions:}\label{assumptions} 
Here we state a set of assumptions and our main result in its general form Theorem~\ref{t.notwidenotthingeneral}. This form generalizes Theorem~\ref{t.notwidenothin}
because its hypotheses are the concerned assumptions, which we will show  in Appendix~\ref{app.exp and geo satisfy assumptions} to be satisfied by exponential and geometric LPP.

We recall first that $\nu$ is the distribution of the vertex weights and has support contained in $[0,\infty)$. %
 Consider next the limit shape defined by the map $[-1,1] \to \R: x\mapsto \lim_{r\to\infty} r^{-1}\E[X_{r(1-x), r(1+x)}]$, where $X_{r(1-x), r(1+x)}$ is the last passage value from $(1,1)$ to $(r(1-x), r(1+x))$. (Note then that $X_{r,r}$ coincides the existing usage $X_r$; we will write the latter in this special case.) Standard superadditivity arguments yield that the last limit in fact exists and that this map is concave. Recall also that $\mu = \lim_{r\to\infty} r^{-1}\E[X_{r,r}]$ is this map evaluated at zero. Superadditivity also yields that $r^{-1}\E[X_{r(1-x),r(1+x)}]$ for finite $r$ is at most the value of the limiting map at $x$. An important problem for general LPP models is to bound the non-random fluctuation given by the difference of these two quantities; for example, $\mu r -\E[X_r]$ when $x=0$, expected to be of fluctuation order $\sqrt{\Var(X_r)}=\Theta(r^{1/3})$. With this context, we state our assumptions.

\goodbreak
\begin{enumerate}

	\item \textbf{No atom at zero and limit shape existence:} 
	The distribution $\nu$ is such that $\nu( \{ 0 \} ) = 0$ and $\mu < \infty$. 
	\label{a.passage time continuity}
   \item \textbf{Strong concavity of limit shape and non-random fluctuations:} There exist positive constants $\rho$, $G_1, G_2$, $g_1$, and $g_2$ such that, for large enough $r$ and $x\in[-\rho, \rho]$, \label{a.limit shape assumption}
	\begin{equation*}%
	\E[X_{r(1-x),r(1+x)}] -\mu r\in [-G_1x^2r-g_1r^{1/3}, -G_2x^2r-g_2 r^{1/3}].
	\end{equation*}

	\item \textbf{Moderate and large deviation estimates, uniform in direction: } \label{a.one point assumption combined}
	
	\begin{enumerate}
		\item Fix any $\delta>0$, and let $|x|\in [0,1-\delta]$. Then, there exist positive finite constants $c =c(\delta)$, $\theta_0=\theta_0(\delta)$, and $r_0=r_0(\delta)$ such that, for $r>r_0$ and $\theta>\theta_0$,
		\begin{align*}
		&\displaystyle	\P\left(X_{r(1+x),r(1-x)} -\E[X_{r(1+x),r(1-x)}] > \theta r^{1/3}\right) \leq \exp(-c\min(\theta^{3/2}, \theta r^{1/3}))\quad \text{and} %
		\\
		&\displaystyle \P\left(X_{r(1+x),r(1-x)} -\E[X_{r(1+x),r(1-x)}] < -\theta r^{1/3}\right) \leq \exp(-c\theta^{3/2}).
		\end{align*}
		
		\label{a.one point assumption}

		\item \label{a.one point assumption convex}
		There exist convex functions $I_r: [0,\infty) \to \R$ for $r\in\N$ and a constant $c'$ such that $I_r(\theta) \geq c'\min(\theta^{3/2}, \theta r^{1/3})$ for all $\theta\geq 0$, such that
		$$\P\left(X_r -\E[X_r] > \theta r^{1/3}\right) \leq \exp(-I_r(\theta)).$$
	\end{enumerate}
	\end{enumerate}
	We will call these, naturally enough, Assumptions \ref{a.passage time continuity}, \ref{a.limit shape assumption}, and \ref{a.one point assumption combined}. They are expected to hold for a wide class of distributions $\nu$ and, in particular, are known to hold for the geometric and exponential cases. It is worth pointing out that a significant portion of our main result, Theorem~\ref{t.notwidenotthingeneral} ahead, holds without Assumption \ref{a.one point assumption convex}. %

	Assumption~\ref{a.limit shape assumption} encodes a non-trivial random fluctuation about a locally strongly concave limit shape. The assumption indicates that, even in the diagonal case, $X_r$ falls short in mean of the linear growth rate $\mu r$ by the order of typical fluctuation. This phenomenon is associated to the negativity of the mean of the GUE Tracy-Widom distribution. In regard to this assumption,
	we will say the endpoint of a path starting at $u\in \Z^2$ satisfies the ``$\rho$-condition" if it lies in the interval joining $u+\left((1-\rho)r, (1+\rho)r\right)$ and $u+\left((1+\rho)r, (1-\rho)r\right)$.

	It can be shown quite easily that Assumption~\ref{a.one point assumption} implies that $\nu$ has finite exponential moment, which in particular implies Assumption~\ref{a.passage time continuity}. Assumption~\ref{a.one point assumption} is itself expected to hold in a stronger form in general, with the lower tail bound with exponent 3 in place of the weaker $3/2$ as we have assumed.
	Finally, Assumption~\ref{a.one point assumption convex} is a slightly stronger version of the upper tail bound of \ref{a.one point assumption} when $x=0$. The existence of the convex functions $I_r$ follows from the superadditivity of the sequence $\{X_r\}_{r\in \N}$ and is not actually an assumption, but the lower bound on $I_r$ is. %

{We are ready to state Theorem~\ref{t.notwidenotthingeneral}. But first we remark that much recent progress in understanding the geometry of first passage percolation and other non-integrable models have been conditional results that hinged on similar assumptions on the limit shape and concentration estimates about the limit shape. This approach goes back to Newman's work in the 90s (see e.g.\ \cite{New95}) where geodesics and fluctuations in FPP were studied under curvature assumptions on the limit shape. More prominent recent examples include the work \cite{Cha13} of Chatterjee where the KPZ relation between the weight and transversal fluctuation exponents were proved assuming in a strong form the existence of these exponents; see also \cite{AD14}. The geometry of geodesics and bi-geodesics has been addressed under assumptions of strong convexity of the limit shape \cite{DH14} and moderate deviations around it \cite{Ale20}. Similar results have also been obtained in exactly solvable cases where the essential integrable ingredients used were estimates analogous to the ones in the above assumptions. It is thus of much interest to extract the minimal set of assumptions under which one can establish sharp geometric results for LPP models. 
\begin{maintheorem}
\label{t.notwidenotthingeneral}
Consider a last passage percolation model on $\Z^2$ that satisfies Assumptions \ref{a.passage time continuity}, \ref{a.limit shape assumption} and~\ref{a.one point assumption}.
\begin{enumerate}
\item \label{weight1'}{\em Weight fluctuation:} There exist positive  constants $C_1$, $C_2$, $C$, $c_1$ and $c$ such that, for $1\le k\leq c_1 n^{1/2}$,  
\begin{equation*}
\P\left(X_{n}^k - \mu nk\notin -k^{5/3}n^{1/3}\cdot(C_1, C_2) \right)\leq Ce^{-ck^2}.
\end{equation*}

\item {\em Transversal fluctuation:} \label{tf'}
\begin{enumerate}
	\item \label{tf1'} There exist positive  constants $M$, $C$, $c$ and $c_1$  such that, for $1 \leq k \leq c_1n^{1/2}$,
	\begin{equation*}
	\P\left(\overline\tf(n,k) > Mk^{1/3}n^{2/3}\right) \leq Ce^{-ck}.
	\end{equation*}

	\item \label{tf2'} There exist positive  constants $C$, $c$, $c_1$ and $\delta$ such that, for $1\le k\leq c_1 n^{1/2}$, 
	\begin{equation*}
	\P\left(\underline\tf(n,k) < \delta k^{1/3}n^{2/3}\right) \leq C e^{-ck^2}.
	\end{equation*}
\end{enumerate}
\end{enumerate}
If  Assumption \ref{a.one point assumption convex} also  holds, there exists $c_1>0$ such that the above hold for $k\leq c_1n$.

\end{maintheorem}

\begin{remark}
An aspect of Theorem~\ref{t.notwidenotthingeneral}(\ref{weight1'}), namely the bound $\P(X_{n}^k < \mu nk - C_1k^{5/3}n^{1/3})\leq e^{-ck^2}$, holds in the broader range $k\leq c_1n$ without Assumption~\ref{a.one point assumption convex}: see Theorem~\ref{t.flexible construction}.

A simple argument shows that the tail exponent for this probability is sharp; i.e.,
$$\P\left(X_{n}^k < \mu nk - C_1k^{5/3}n^{1/3}\right)\ge e^{-ck^2}$$
 for some constant $c>0.$
 Indeed,  the event that $X_{n} < \mu n - C_1k^{2/3} n^{1/3}$ implies that $X_{n}^k < \mu n k - C_1k(k^{2/3}) n^{1/3}$ for any $k$, and $n$ correspondingly large enough. The preceding display follows from  
 $$
 \P\left(X_{n} < \mu n - C_1k^{2/3}n^{1/3}\right)=e^{-\Theta(k^{2})} \, ,
 $$ 
 a bound that is known, for example, in Exponential LPP: see \cite[Theorem 2]{basu2019lower}. %
\end{remark}

\begin{remark} 
A weaker form of Assumption  \ref{a.one point assumption}, namely that for some $\alpha<1$, all $\delta>0$, $|x|\in[0,1-\delta]$, and $\theta>\theta_0 = \theta_0(\delta)$,
$$\P\left(|X_{r(1+x), r(1-x)} - \E[X_{r(1+x), r(1-x)}| > \theta r^{1/3}\right) \leq \exp(-c\theta^{\alpha}),$$
for some $c = c(\delta)>0$, is enough to imply a variant of Theorem~\ref{t.notwidenotthingeneral} where the tail probability exponents are suitable functions of $\alpha$; the $5/3$ and $1/3$ exponents for weight and transversal fluctuations are unchanged. It does not appear to us to be challenging to chase through our arguments to compute the forms of these upper bound tail exponents, although we do not do so. 
\end{remark}

\subsubsection{A non-determinantal setting: the point-to-line geodesic watermelon}
Though our main result Theorem \ref{t.notwidenotthingeneral} is stated for last passage percolation on $\Z^2$, our technique is robust, and an inviting prospect is to adapt our method to other integrable models such as the semi-discrete model of Brownian LPP or the continuum model of Poissonian LPP; the adaptations appear to be for the most part minor, but we have not pursued this direction carefully. However, all  four examples---exponential and  geometric LPP and these last two---are {\em determinantal} in the sense that there is an exact representation of the geodesic watermelon weight as the sum of the position of top $k$ particles in a determinantal point process with explicit, albeit complex, formulae available for the distributions. We shall provide a more detailed discussion regarding the determinantal process connections for exponential and geometric LPP in Section~\ref{s:special}. But to illustrate the power of our geometric methods we end by treating a particular ``non-determinantal" setting: where although there exist explicit formulae for the one point distribution (the geodesic weight), the weight of the geodesic watermelon is not known to admit any connection to a determinantal process for $k>1$.

Formally, consider point-to-line LPP with independent vertex weights such that Assumptions \ref{a.passage time continuity}, \ref{a.limit shape assumption}, and \ref{a.one point assumption combined} are satisfied. Let $\Upsilon_{n}^{k}$ denote a  collection of $k$ disjoint paths contained in the triangular region
$$\left\{(x,y)\in \Z^2: x\geq 1, y\geq 1, x+y\leq 2n\right\},$$
that maximizes the total weight among all such collection of $k$ disjoint paths. Let $Z_{n}^{k}$ denote the total weight of the paths in $\Upsilon_{n}^{k}$. Clearly $Z_{n}^{1}$ is the point-to-line last passage time from $(1,1)$ to the line $x+y=2n$. While it is known in exponential LPP that $Z_{n}^{1}$ has the same distribution as the largest eigenvalue of the Laguerre Orthogonal Ensemble (implicitly in the works \cite{baik2001symmetrized,baik2002painleve,baik2001algebraic} with an explicit statement in \cite{basu2019lower}), as far as we are aware there does not exist any representation for $Z_{n}^{k}$ for $k>1$ as a functional of a determinantal point process.

As before we specify the transversal fluctuation of a collection of paths to be the maximum Euclidean distance between a vertex in $\Z^2$ lying in the paths' union and the diagonal $y =x$. We specify $\overline\tf^*(n,k)$ to equal the maximum value of the transversal fluctuation over such sets of $k$ paths whose weight realizes $Z_n^k$, and $\underline\tf^*(n,k)$ to be the minimum value of the transversal fluctuation over the same collection of sets.

The following is the analogue of Theorem \ref{t.notwidenothin} in this setting.

\begin{maintheorem}
\label{t:p2lgeneral}
Consider a model, such as geometric or exponential LPP, that  satisfies Assumptions \ref{a.passage time continuity}, \ref{a.limit shape assumption} and \ref{a.one point assumption}. The statements in Theorem~\ref{t.notwidenotthingeneral} %
that invoke Assumption~\ref{a.one point assumption} hold after the replacements $X^k_n \to Z^k_n$, $\overline\tf(n,k) \to \overline\tf^*(n,k)$  and $\underline \tf(n,k) \to \underline\tf^*(n,k)$ are made. In particular, the results for point-to-line geodesic watermelons hold for $k\leq c_1n^{1/2}$ for an absolute constant $c_1>0$.
\end{maintheorem}

The next two sections develop an overview of the paper's concepts and results. 
The vital phenomenon of interlacing of  geodesic watermelons is surveyed in 
Section~\ref{s:interlace}, with some of its main consequences being indicated.  Section~\ref{iop} offers an outline to the paper's main proofs. We state some important technical results needed to obtain Theorem~\ref{t.notwidenotthingeneral}; outline how these results are proved; and explain how they are used, alongside interlacing, to prove Theorem~\ref{t.notwidenotthingeneral}.
Section~\ref{iop} ends with an indication of the structure of the later sections of the paper, which are devoted to giving the proofs.

\subsection*{Acknowledgements}
The authors thank Ivan Corwin for pointing them to references that the mean of the GUE Tracy-Widom distribution is negative. RB thanks Manjunath Krishnapur for useful discussions on determinantal point processes. MH thanks Satyaki Mukherjee for piquing his interest in the problem. RB is partially supported by a Ramanujan Fellowship (SB/S2/RJN-097/2017) from the Government of India, an ICTS-Simons Junior Faculty Fellowship, DAE project no. 12-R\&D-TFR-5.10-1100 via ICTS and Infosys Foundation via the Infosys-Chandrasekharan Virtual Centre for Random Geometry of TIFR. SG is partially
supported by NSF grant DMS-1855688, NSF CAREER Award DMS-1945172, and a Sloan Research Fellowship. AH is supported by the NSF through grant DMS-1855550 and by a Miller Professorship at U.C. Berkeley. MH acknowledges the generous support of the U.C. Berkeley Mathematics department via a summer grant and the Richman fellowship.

\section{Watermelon interlacing}
\label{s:interlace} 

In this section, we will specify notation and important definitions concerning 
 geodesic watermelons; state the important interlacing property that they enjoy; and state monotonicity and rigidity results for a natural point process associated to the weights of these watermelons. The results of this section do not rely on the assumptions introduced in Section~\ref{assumptions}, and, excepting Proposition~\ref{p.melon interlacing}, are deterministic.

\subsection{The geodesic watermelon} \label{s.geodesic watermelon subsection}
Let 
$\left\{\xi_v :  v\in \Z^2\right\}$
denote a field of values. For now, we take these values to be deterministic non-negative reals. 
For $a,b \in \Z$ with $a \leq b$, we denote the integer interval $\big\{ x\in \Z: a \leq x \leq b \big\}$ by $\llbracket a,b \rrbracket$. We let $\N = \{1,2, \ldots\}$ denote the natural numbers (without zero).

Let two elements $u = (u_1,u_2)$ and $v = (v_1,v_2)$ of $\Z^2$
be such that $u_1 \leq v_1$ and $u_2 \leq v_2$. 
  An upright path from $u$ to $v$ is a function $\gamma: \llbracket a,b \rrbracket \to \Z^2$, $\gamma_a = u$ and $\gamma_b = v$, with each increment of $\gamma$ equalling either $(0,1)$ or $(1,0)$; thus, $a,b \in \Z$ satisfy $b - a = (u_2 - u_1) + (v_2 - v_1)$. (We will sometimes omit `upright': every path is upright.)  The weight $\weight(\gamma)$ of $\gamma$ equals $\sum \big\{ \xi_v: v \in \gamma \big\}$, where we have abused notation, in a way that we often will, by mistaking $\gamma$ for its range. The weight of a disjoint collection of upright paths is the sum of the weights of the elements of the collection.
  
We will denote by $X_{u \to v}$ the maximum weight of all upright paths from $u$ to $v$. Recall that, for $r \in \N$, $X_r$ is a shorthand for $X_{r,r}$, itself shorthand for the last passage value $X_{(1,1) \to (r,r)}$ for the route from $(1,1)$ to $(r,r)$. 

Let $k \in \N$ be positive. We may wish to specify  the $k$-geodesic watermelon in the square $\intint{n}^2$ as a maximum weight collection of $k$ disjoint upright paths from $(1,1)$ to $(n,n)$; but, naturally, we cannot, because such paths meet when they begin and end.  The next definition succinctly deals with the need to unpick these points of contact.

\begin{defn}\label{d.geodesicwatermelon}
Let $n \in \N$ and $k \in \intint{n}$. A $k$-geodesic watermelon is a maximum weight collection of $k$ disjoint upright paths in $\intint{n}^2$.
\end{defn} 

The nomenclature of watermelons is not new. It was introduced in the physics literature to denote certain ensembles of non-intersecting curves, such as non-intersecting Brownian bridges, whose curves bear a faint likeness to the stripes on the surface of watermelon fruit.
These bridge systems arise, for example, when describing the weight of collections of disjoint paths in LPP models as the common endpoint of the collection varies.  The name `geodesic watermelon' distinguishes the denoted concept from the existing one, which we might call a \emph{weight} watermelon. We will not allude to the latter, beyond mentioning that   \cite{johnston2020scaling} treats weight watermelons related to the geometric RSK correspondence. We will sometimes write `$k$-melon' as shorthand for `$k$-geodesic watermelon'.

The two quantities measuring transversal fluctuation of the $k$-geodesic watermelon, $\overline\tf(n,k)$ and $\underline \tf(n,k)$,  were specified before Theorem~\ref{t.notwidenothin}. They respectively equal the maximum and minimum, over the collection of $k$-geodesic watermelons, of the maximum Euclidean distance to the diagonal $y=x$ of the vertices lying 
in some path belonging to that $k$-geodesic watermelon.

The set of vertices in a $k$-geodesic watermelon may not be unique; it is not in geometric LPP, for example. Further, for a given set of vertices which is a $k$-geodesic watermelon, the collection of constituent curves is not unique. It is easy to see that a sufficient condition for the geodesic watermelons to be almost surely unique is for $\nu$ to have no atoms, as, for example, in the case of exponential LPP.

The next result will help in specifying the curves of a watermelon that we may label and study. The result serves to capture the sense that the watermelon has maximum weight subject to its disentangling coincidences near $(1,1)$ and $(n,n)$. The proof will be given in Section~\ref{s:geometry}.

\begin{proposition}\label{p.starting and ending points}
Suppose that $\xi_v \geq 0$ for $v \in \intint{n}^2$.
For any given k-geodesic watermelon $\phi_n^k$, there exists a k-geodesic watermelon $\gamma_n^k$ such that the $i$\textsuperscript{th} element of $\gamma_n^k$ starts at $(1,k-i+1)$  and ends at $(n,n - i+1)$ for $i \in \intint{k}$, and the union of the vertices in the curves of $\gamma_n^k$ contains this union for $\phi_n^k$, and coincides with it if $\xi_v > 0$ for all $v\in\intint{n}^2$.
\end{proposition}

Theorem~\ref{t.notwidenotthingeneral} concerns the vertex sets in geodesic watermelons considered simultaneously, rather than any particular watermelon. For this reason, we do not attempt to single out a melon when there are several distinct watermelon vertex sets. A subtlety regarding non-uniqueness and interlacing will however be addressed in the upcoming Section~\ref{s.interlacing}.

For any given set of vertices which form a geodesic watermelon, Proposition~\ref{p.starting and ending points} permits us to label the curves of the watermelon, which will aid us in stating the upcoming interlacing result. We wish to do this in a left-to-right manner, and to make this precise we next introduce a partial order on paths.

\begin{defn}\label{d.time range}
The time-range $\mc{R}(\gamma)$ of an upright path $\gamma$ is the interval of integers $t$ such that $x+y = t$ for some $(x,y) \in \gamma$. 
For $t \in \mc{R}(\gamma)$, the point $(x,y)$ is unique, and we abuse notation by writing $\gamma(t) = x - y$. (In contrast, subscripts were used to express $\gamma$ as a function on an integer interval earlier in the section.)

Two upright paths $\gamma$ and $\phi$
satisfy $\gamma \preceq \phi$ if the set  $\mc{R}(\gamma) \cap \mc{R}(\phi)$ is non-empty and every element $t$ satisfies $\gamma(t) \leq \phi(t)$. If each inequality is strict, we write $\gamma \prec \phi$. Informally, we say that $\gamma$ is to the left of $\phi$.

 A vector of upright paths is ordered if its component increase under $\prec$, and weakly ordered if they do so under $\preceq$.  
 For a vector of upright paths $\gamma$, note that if its components $\gamma_i$ are disjoint and every pair has non-disjoint time ranges, then by planarity there is exactly one labelling of the paths which is ordered from left to right. In that case, we shall simply say the collection of paths is ordered. 
\end{defn}

With this definition, for any given $k$-geodesic watermelon $\Gamma_n^k$ identified by Proposition~\ref{p.starting and ending points}, we record its elements as the components of the vector  
\begin{equation}\label{watermelondef}
\Gamma^k_{n} = \left(\Gamma^{k,1}_{n,\shortrightarrow}, \ldots, \Gamma^{k,k}_{n,\shortrightarrow}\right) \, ,
\end{equation}
choosing the left-to-right order, so that $\Gamma_{n,\shortrightarrow}^{k,i} \prec \Gamma_{n,\shortrightarrow}^{k,i+1}$ for $i\in\intint{k-1}$; thus, $\Gamma_{n,\shortrightarrow}^{k,i}$ begins at $(1,k-i+1)$  and ends at $(n,n - i+1)$.

\subsection{Interlacing}\label{s.interlacing}

The basic property that drives our proofs is simple and intuitive. Consider exponential LPP, so that all $k$-geodesic watermelons are almost surely unique. The square $\intint{n}^2$ is partitioned into two regions, NW and SE, by the geodesic $\Gamma^1_n$, which lies in both. 
Then $\Gamma_{n,\shortrightarrow}^{2,1}$ lies in NW and  $\Gamma_{n,\shortrightarrow}^{2,2}$ in SE. Three regions are specified by the boundaries  $\Gamma_{n,
\shortrightarrow}^{2,1}$ and  $\Gamma_{n,\shortrightarrow}^{2,2}$ and the sides of the square, and these contain, one apiece, the three paths in the $3$-geodesic watermelon: see Figure~\ref{f.interlacing} ahead. The partial order introduced in Definition~\ref{d.time range} allows us to put interlacing on a precise footing.

\begin{figure}[h]
   \centering
        \includegraphics[width=.75\textwidth]{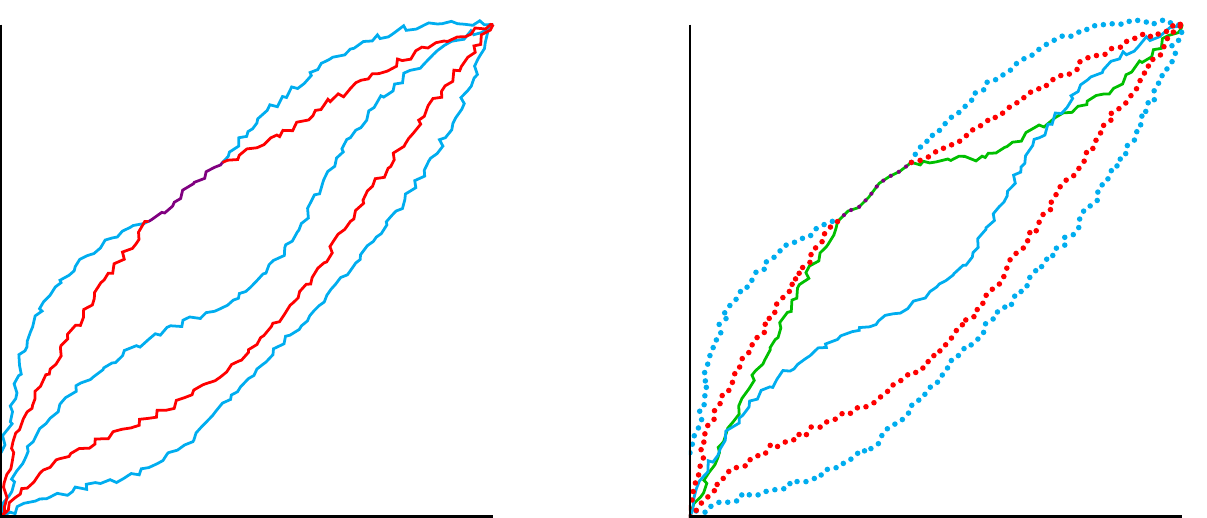}
 \caption{On the left is a depiction of the interlacing of $k$-geodesic watermelons for $k=2$ (red) and $k=3$ (blue). Each red curve remains within the region defined by adjacent blue curve boundaries. Observe that the red curves may overlap with the blue ones as shown between the top red and top blue curve; in fact, simulations suggest that a much more significant degree of overlap than depicted is typically present. Interlacing concerns consecutive values of $k$:  the $k$- and $(k+2)$-melons may cross, as the right figure illustrates for $k=1$. 
  The geodesic curve in green make uses of the overlapping purple portion of the $2$- and $3$-melons, while remaining in the envelope defined by the red curves. The middle blue curve cannot use the purple portion, and this causes the green curve to cross over the middle blue curve.}
 \label{f.interlacing}
\end{figure}

\begin{defn}
A $k$-vector $\gamma$ and a $(k+1)$-vector $\gamma'$ whose components are upright paths
 {\em interlace} when their components satisfy $\gamma'_i \preceq \gamma_i \preceq \gamma'_{i+1}$ for $i \in \intint{k}$.
\end{defn}

Figure \ref{f.interlacing} shows a pair of interlacing geodesic watermelons. We are ready to state our interlacing results, starting with the case where $\nu$ is continuous, so that Proposition~\ref{p.starting and ending points} gives a unique $k$-geodesic watermelon $\Gamma_n^k$ for every $k\in\intint{n}$.

\begin{proposition}[Geodesic watermelons interlace]\label{p.melon interlacing}
Suppose that the law $\nu$ is continuous.  For $k \in \intint{n-1}$, the geodesic watermelons
$\Gamma_{n}^{k}$ and $\Gamma_{n}^{k+1}$ interlace almost surely. 
\end{proposition}

When the distribution~$\nu$ has atoms, the set of vertices visited by a geodesic watermelon may be variable. Here is a deterministic interlacing result that is valid in this setting. 

\begin{proposition}\label{p.specified melon interlacing}
For $n \in \N$, suppose that the values $\xi_v$ are non-negative for $v \in \intint{n}^2$. 
Let $k \in \intint{n}$, and let $\phi_n^k$ be any $k$-geodesic watermelon with starting and ending points on the bottom left and top right sides of $\intint{n}^2$ as in Proposition~\ref{p.starting and ending points}. 
For $i \in \intint{n}$, we may find an $i$-geodesic watermelon 
$\gamma_n^i$ such that consecutive terms in the sequence $\big\{ \gamma_n^i: i \in \intint{n} \big\}$ interlace, with the union of the vertices in the curves in $\gamma_n^k$
coinciding with this union for $\phi_n^k$.
\end{proposition}

\subsection{Monotonicity of watermelon weight increments}\label{s.geometry.proof of ordering}

A natural point process $\{ Y_{n,i}: i \in \intint{n} \}$ associated to the sequence $\{X_n^k : k \in \intint{n} \}$ is obtained by stipulating that 
\begin{equation}\label{e:ensemble}
X_n^k \, = \, \sum_{i=1}^k Y_{n,i}
\end{equation}
for each $k \in \intint{n}$. 
This process records increments $Y_{n,i}=X_{n}^i-X_{n}^{i-1}$, $i \in \llbracket 2,n \rrbracket$,  in geodesic watermelon weight as the curve number rises, starting out at  $Y_{n,1}=X_{n}^1$, the geodesic weight for the route from $(1,1)$ to $(n,n)$.
 
We mention that, in exponential LPP, the process $\{ Y_{n,i}: i \in \intint{n} \}$ equals in law a list in decreasing order of the  eigenvalues of an $n\times n$-matrix picked randomly according to the  Laguerre Unitary Ensemble (LUE) \cite{adler-eig-perc-connection}. 
In particular, this random list decreases almost surely. Arguments of the flavour of those that establish interlacing yield this deterministic monotonicity of $\{ Y_{n,i}: i \in \intint{n} \}$ in a more general setting.

\begin{proposition}[Increment monotonicity]
\label{p:ordering}
For $n$ a positive integer, suppose that the values $\xi_v$ are non-negative for $v \in \intint{n}^2$.  Let $k \in \llbracket2,n\rrbracket$.
Then $Y_{n,k-1} \geq Y_{n,k}$.
Suppose further that the values $\big\{ \xi_v : v \in \intint{n}^2 \big\}$ are independently picked according to a continuous law $\nu$. 
Then the stated inequality is almost surely strict. 
\end{proposition}
An analogous deterministic statement for semi-discrete LPP is proved by algebraic methods in~\cite{dauvergne2018directed}.

\section{Technical ingredients and the key ideas in the proofs}\label{iop}

Here we describe some important technical tools, relying on the above geometric features, which form key ingredients for our proofs. We will also indicate roughly how  these tools are proved. This done, we will outline how the tools will aid in proving Theorem~\ref{t.notwidenotthingeneral}.

In this section, and indeed in the rest of this paper, Assumptions~\ref{a.passage time continuity}, \ref{a.limit shape assumption}, and \ref{a.one point assumption} will be in force in all statements. When Assumption~\ref{a.one point assumption convex} is used in place of \ref{a.one point assumption}, we will indicate this.

\subsection{The technical tools}

There are three major technical tools, each of which is a result of independent interest. 
The first two results give existence and non-existence of certain numbers of disjoint paths in $\intint{n}^2$ with high probability, while the third one gives a quantitative result on paths with high transversal fluctuation. 

\subsubsection{First tool: construction of disjoint paths of high weight}
The key starting point is an explicit construction of $m\in\intint{k}$ disjoint paths achieving, with high probability, a cumulative weight $\mu n m -O(mk^{2/3}n^{1/3})$. This allows by taking $m=k$ to prove the weight lower bound aspect of Theorem~\ref{t.notwidenotthingeneral}(\ref{weight1'}).

\begin{theorem}
\label{t.flexible construction}
There exist $c, c_1, C_1>0$, and $k_0\in \N$ such that, for $k_0\leq k \leq c_1 n$ and $m\in \intint{k}$, it is with probability at least $1-e^{-ckm}$ that there exists a set of $m$ disjoint paths $\big\{ \gamma_i : i \in \intint{m} \big\}$ in the square $\llbracket 1,n \rrbracket^2$, with $\gamma_i$ for each index $i$ being a path from $(1, k-i+1)$ to $(n,n-i+1)$ satisfying $\max_i \tf(\gamma_i) \leq 2mk^{-2/3}n^{2/3}$,  such that
$$ 
\sum_{i=1}^m \ell(\gamma_i) \, \geq \, \mu nm - C_1mk^{2/3}n^{1/3}.
$$
\end{theorem}

\begin{figure}[h]
   \centering
        \includegraphics[width=.4\textwidth]{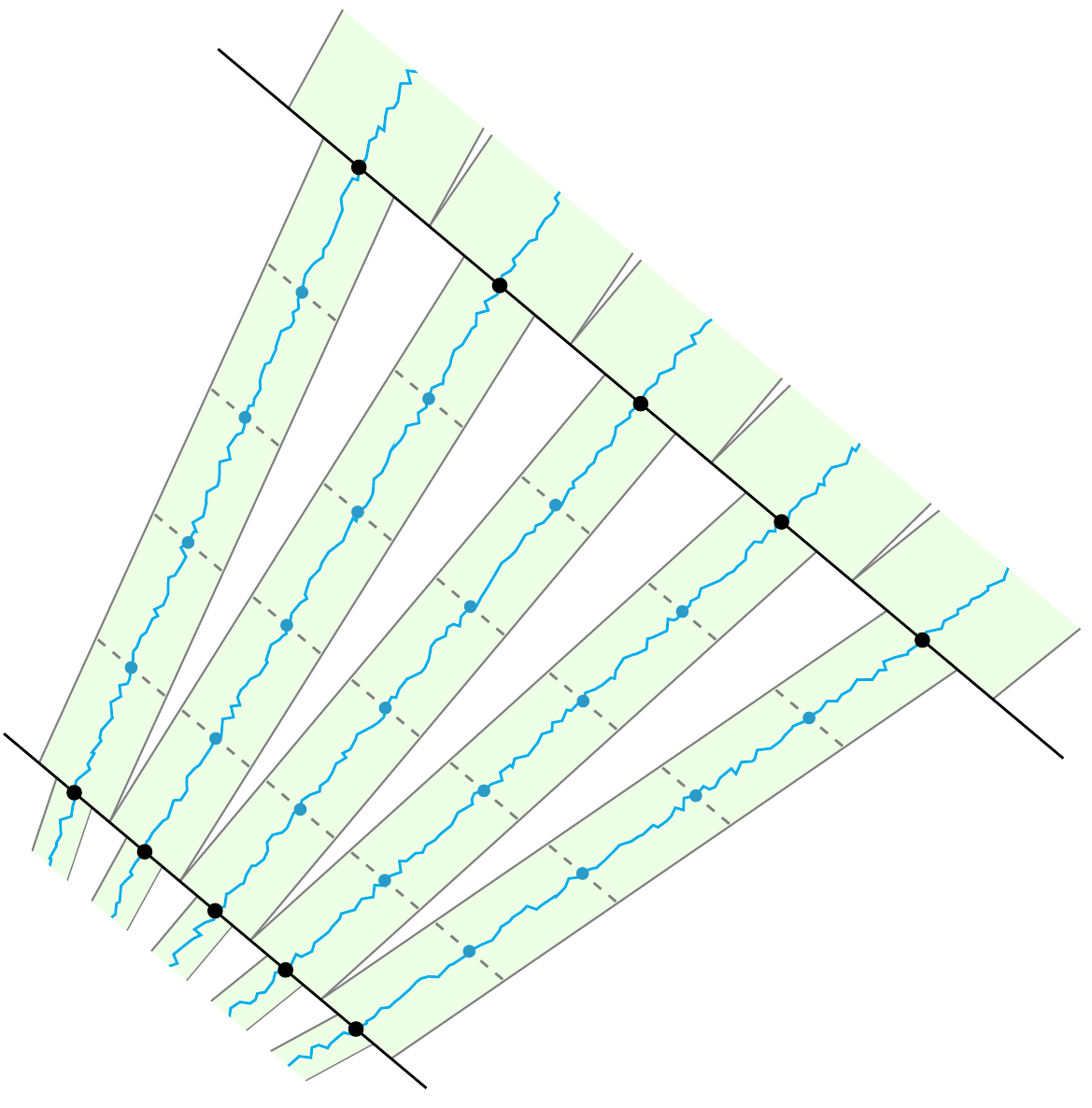}
 \caption{The dyadic construction of $k$ disjoint curves with $k=5$. Here we have focused on a single scale, while the preceding and succeeding scales are partially visible. Observe that in a single scale, each of the $k$ curves passes through $k$ consecutive parallelograms of a particular size (as determined by the scale) and aspect ratio. A more detailed depiction of this part of the construction is provided in Figure~\ref{f.construction segment 2}}.
 \label{f.construction mini}
\end{figure}

The construction, in  Section~\ref{s:construction}, is a multi-scale one where at each dyadic scale we consider geodesics constrained to lie in $k$ disjoint parallelograms with a carefully chosen aspect ratio: see Figure~\ref{f.construction mini}.  The paths $\gamma_i $ for $i=1,2,\ldots, k$ are now formed by concatenating these constrained geodesics at different scales. The weight lower bound in Theorem \ref{t.flexible construction} can then be obtained by analyzing the typical weight of the constrained geodesics, and a concentration estimate implies that the weight of a concatenation of independent geodesics is typically close to its mean. 

\subsubsection{Second tool: several disjoint paths constrained in a strip typically have low weight}
This tool asserts that the event that there exist $k$ curves, each meeting a certain weight lower bound and packed within a thin strip around the diagonal, has very low probability. 
It will serve to help prove Theorem~\ref{t.notwidenotthingeneral}(\ref{tf2'}),
the transversal fluctuation lower bound.
Letting 
\begin{equation}\label{parallelogram1}
U_{n,w}:= \big\{2\leq x+y\leq 2n, |x-y|\leq w\big\},
\end{equation}
 we show that if the width $w$ equals $k^{1/3-o(1)}n^{2/3}$, then it is unlikely that there are $k$ disjoint paths in $U_{n,w}$ each of which has weight at least $\mu n-\Theta (k^{2/3}n^{1/3})$.

\begin{theorem}
\label{t:disjoint}
Given $C_3>0$, there exist $\delta=\delta(C_3)>0$ and $c_1 = c_1(C_3)>0$ such that, for $k\in\N$ and $k\leq c_1n$, the following holds. Let $\ce=\ce(n,k,\delta,C_3)$ denote the event that there exist $k$ disjoint paths $\gamma_1,\gamma_2,\ldots, \gamma_{k}$ contained in $U_{n,\delta k^{1/3}n^{2/3}}$ such that $\ell(\gamma_i)\geq \mu n- C_3 k^{2/3}n^{1/3}$ for $i=1,2,\ldots, k$. Then there exists an absolute constant $c>0$ such that $\P(\ce)\leq e^{-ck^2}$.  
\end{theorem}

This result will be proved in  Section~\ref{s:bk}. How? If there are a large number of disjoint paths in a strip, then some of them must be restricted to some thin region. Our assumptions will imply that any path contained in a thin region has an exponentially-low-in-$k$ probability of having weight shortfall from $\mu n$ less than $C_3 k^{2/3}n^{1/3}$, and the van den Berg-Kesten (BK) inequality will complete the proof. Recall that the BK inequality bounds the probability of occurrence of a number of events \emph{disjointly} by the product of their individual probabilities.

A similar argument has appeared in \cite{BHS18} to show that the maximum number of disjoint geodesics across the shorter sides of an $n\times n^{2/3}$ rectangle is uniformly tight with good tails.  The argument in this paper is a refinement of this approach, requiring more careful entropy calculations to take care of the growing width of $U_{n,\delta k^{1/3}n^{2/3}}$ in $k$.

\subsubsection{Third tool: the weight shortfall of paths with high transversal fluctuations} 
This interesting result will help to  prove Theorem~\ref{t.notwidenotthingeneral}(\ref{tf1'}), the upper bound on geodesic watermelon transversal fluctuation.
The tool asserts that any path from $(1,1)$ to $(r,r)$ of transversal fluctuation greater than $sr^{2/3}$ typically suffers a weight loss of order $s^2r^{1/3}$. The result applies to any path and is not a bound on the probability of the maximal weight path having such a transversal fluctuation. The statement may be expected from known relations between transversal and longitudinal fluctuations (which is a consequence of the curvature assumption in Assumption~\ref{a.limit shape assumption} and one-point fluctuation information in Assumption~\ref{a.one point assumption}). While similar results with suboptimal exponents have appeared (see \cite{slow-bond, BGZ19}), we obtain the following sharp  result, adapting a multi-scale argument from \cite{slow-bond}. The proof appears in Section~\ref{s:upperexp}.
\begin{theorem}\label{t.tf}
Let $X^{r,s,t}$ be the maximum weight over all paths $\Gamma^{r,s,t}$ from the line segment joining $(-tr^{2/3}, tr^{2/3})$ and $(tr^{2/3}, -tr^{2/3})$ to the line segment joining 
$(r-tr^{2/3},r+tr^{2/3})$ and $(r+tr^{2/3},r-tr^{2/3})$
such that $\tf(\Gamma^{r,s,t}) > (s+t)r^{2/3}$, with $t\leq s$.
There exist absolute constants
$r_0$, $s_0$, $c>0$ and $c_2>0$
such that, for $s>s_0$ and $r>r_0$,
$$\P\left(X^{r,s,t} - \mu r > -c_2s^2r^{1/3}\right) < e^{-cs^3}.$$
\end{theorem}
We now show how the interlacing property and the above results will yield Theorem~\ref{t.notwidenotthingeneral}. 

\subsection{Proof outline of Theorem~\ref{t.notwidenotthingeneral}}
We already noted how Theorem~\ref{t.flexible construction} implies the weight lower bound of Theorem~\ref{t.notwidenotthingeneral}(\ref{weight1'}). This is a collective weight lower bound, but for the remaining parts of Theorem~\ref{t.notwidenotthingeneral} we will roughly need a lower bound on individual curve weights of order $\mu n - \Theta(k^{2/3}n^{1/3})$; note that this quantity is $k^{-1}$ times the weight lower bound $\mu nk-\Theta(k^{5/3}n^{1/3})$ of Theorem~\ref{t.notwidenotthingeneral}(\ref{weight1'}). This individual curve lower bound is provided in a slightly weaker (but sufficient) sense by an averaging argument, which is a simple but core ingredient of our proofs.

\subsubsection{The averaging argument} Let $X_n^{j,i}$ be the weight of the $i$\textsuperscript{th} heaviest curve of a given $j$-geodesic watermelon. The averaging argument yields that
there exists with high probability $j\in \llbracket k+1, 2k\rrbracket$ such that, for all $i\le j$,
\begin{equation}\label{avecons}
X_{n}^{j,i}\ge X_{n}^{j,j}\ge \mu n -\Theta(k^{2/3}n^{1/3}),
\end{equation}
and similarly for some $j\in \llbracket \lfloor k/2\rfloor, k\rrbracket$.
The argument is powered by a simple but very useful inequality: for $j\in \N$,
\begin{equation}\label{e.averaging inequality}
X_n^{j,j}\geq X_n^{j}-X_n^{j-1}.
\end{equation}
This is proved by considering the lowest weight curve and the remaining $j-1$ curves of the $j$-melon separately. The weight of the latter is of course at most the weight of the $(j-1)$-melon, and this gives the inequality. Notice that while $X_n^{j,j}$ is defined only once a $j$-geodesic watermelon has been specified, the lower bound in \eqref{e.averaging inequality} is independent of the choice of $j$-geodesic watermelon. This allows us to prove that \eqref{avecons} holds for some $j\in\llbracket k+1, 2k\rrbracket$ simultaneously for all choices of $j$-geodesic watermelon with high probability. 

 Deriving \eqref{avecons} from \eqref{e.averaging inequality}  is a matter of appealing to the lower bound on $X_n^{2k}$ from Theorem~\ref{t.notwidenotthingeneral}(\ref{weight1'}) and the following crude upper bound on $X_n^k$, which is proved using one-point information (Assumption~\ref{a.one point assumption}) and the BK inequality (see Section \ref{s:bk}); note the change in sign of the fluctuation term when compared with the weight upper bound of Theorem~\ref{t.notwidenotthingeneral}(\ref{weight1'}).

\begin{lemma}\label{l.first upper bound on melon weight}
 Under Assumption~\ref{a.one point assumption}, we may find $c>0$ such that, if $t>0$, there exists $k_0 = k_0(t)$ for which $n>n_0$ and $k_0<k\leq  t^{-3/4}n^{1/2}$ imply that
$$\P\left(\melonweight > \mu nk + tk^{5/3}n^{1/3}\right) \leq \exp\left(-ct^{3/2}k^2\right).$$
If instead Assumption~\ref{a.one point assumption convex} is available, then the upper bound on $k$ may be taken to be $\min(1,t^{-3/2})n$.
\end{lemma}

The relaxation on the condition on $k$ we obtain under Assumption~\ref{a.one point assumption convex} is important to prove Theorem~\ref{t.notwidenothin} with the full claimed range of $k$. Lemma~\ref{l.first upper bound on melon weight} is the source of all future statements which give different conditions $k$ under Assumption~\ref{a.one point assumption convex}. The proof of Lemma~\ref{l.first upper bound on melon weight} using Assumption \ref{a.one point assumption convex} is slightly different than with \ref{a.one point assumption}, and is handled separately in Appendix~\ref{app.proof of lemmas}, while the proof with \ref{a.one point assumption} appears in Section~\ref{s:bk}.

With the result of the averaging argument in hand, Theorem~\ref{t.notwidenotthingeneral}(\ref{tf}) is a straightforward consequence of interlacing and Theorems~\ref{t:disjoint} and \ref{t.tf}, as we briefly describe next.

\subsubsection{Bounding transversal fluctuations}
Equation~\eqref{avecons}, along with Theorem~\ref{t:disjoint}, implies the following transversal fluctuation lower bound for the random index $j$,
$$\P\left(\exists j\in \{\lfloor k/2\rfloor,\ldots, k\} \text{ such that } \underline\tf(n,j) > \delta k^{1/3}n^{2/3}\right) \geq 1-e^{-ck^2}.$$ 
This can then be upgraded to Theorem~\ref{t.notwidenotthingeneral}(\ref{tf2'}) by a simple application of interlacing. A similar argument with Theorem~\ref{t.tf} implies that
$$\P\left(\exists j\in \{k+1,\ldots, 2k\} \text{ such that } \overline\tf(n,j) < C k^{1/3}n^{2/3}\right) \geq 1-e^{-ck},$$ 
which when combined with interlacing gives Theorem~\ref{t.notwidenotthingeneral}(\ref{tf1'}).

What remains is the proof of the weight upper bound, Theorem~\ref{t.notwidenotthingeneral}(\ref{weight1'}), which will use essentially all of the gathered ingredients.

\subsubsection{Bounding above watermelon weight} 
The proof has a few parts. Consider, as we did above, a random $k/2\le j\le k$ that satisfies \eqref{avecons}, so that for all $1\le i\le j,$ $X_{n}^{j,i}$ is not too light (for any choice of $j$-geodesic watermelon). By an application of Theorem \ref{t:disjoint} with $k$ replaced by $k/4$, it follows that with  probability $1-e^{-ck^2}$, at least $k/4$ paths in all $j$-geodesic watermelons exit a strip of width $\Theta(k^{1/3}n^{2/3}).$ The interlacing property implies that the latter property holds for any fixed $k$-geodesic watermelon $\Gamma_{n}^{k}$ as well. 

We also know that the cumulative weight of the heaviest $7k/8$ paths of $\Gamma_n^k$ cannot be too high, by the crude estimate Lemma \ref{l.first upper bound on melon weight} (with $k$ replaced by $7k/8$). This implies that on the event that the $k$-geodesic watermelon's weight is high, the cumulative weight of the \emph{lightest} $k/8$ paths of $\Gamma_n^k$ is high. Hence the heaviest path among the latter has unusually high weight which, along with monotonicity, implies the same lower bound for all the heaviest $7k/8$ paths of $\Gamma_n^k$.  

Since we have already established that at least $k/4$ paths in the ensemble $\Gamma_{n}^{k}$ exit a strip around the diagonal of width $\Theta(k^{1/3}n^{2/3}),$ this implies that at least $k/8$ disjoint paths both exit a strip of width $\Theta(k^{1/3}n^{2/3})$ and each have an unusually high weight. The probability that a single curve has both these properties is at most $\exp(-ck)$ by Theorem~\ref{t.tf}, and so the probability that $O(k)$ many disjoint curves have these properties is $\exp(-ck^2)$ by the BK inequality. We note that, without interlacing, the above would give a single curve with unusually high weight and large transversal fluctuation, which would yield the weaker probability bound of $\exp(-ck)$.

\subsubsection{A few further comments in overview} The proof strategies make it apparent that the arguments have minimal dependence on the precise nature of the model and indeed straightforward modifications will yield Theorem 
\ref{t:p2lgeneral}.

The proofs critically rely on the averaging argument and watermelon interlacing. The latter is a central theme of the paper and is a consequence of non-local geometric aspects of the watermelons as will be clear from the proofs in Section~\ref{s:geometry}. The monotonicity result, Proposition~\ref{p:ordering}, will also be proved in Section~\ref{s:geometry}, by similar geometric considerations. The geometric point of view defines our approach, both in  technique and outcome. Interlacing is the principal means by which we have expressed this point of view, but in fact the stated monotonicity offers an alternative, but still geometric, route to proving Theorem~\ref{t.notwidenotthingeneral}. In Section~\ref{s:special}, we explore this approach and explain how it leads to certain sharp concentration bounds on each $Y_{n,k}$. We contrast this with weaker concentration bounds obtainable in the integrable setting of exponential LPP by determinantal point process techniques: see%
\begin{submitted-version}
Proposition~\ref{p:kthweight} and \eqref{e.eigenvalue concentration}.
\end{submitted-version}\begin{arxiv-version}
Propositions~\ref{p:kthweight} and \ref{p:kthweight2}.
\end{arxiv-version}%

\subsection{Basic tools}\label{s.basictools}

To implement the ideas of the preceding sections, we will require certain basic estimates on geometrically relevant quantities. These include, for example, tail estimates for last passage values where the endpoints are allowed to vary over intervals whose length is of order $n^{2/3}$, or where one endpoint is allowed to vary over the line $x+y=2n$; and tail and mean estimates for last passage values constrained to remain in certain parallelograms.  In this section we state the precise forms of these estimates that we will be using.

Similar estimates have appeared in the literature, for example in the preprint \cite{slow-bond}, but never under the general assumptions that we adopt. Our proofs follow similar strategies as in these previous works and are provided in Appendix~\ref{app.tail bounds} for completeness. The statements are nonetheless technical and the reader might choose to skip this section at a first reading and only refer back to it as needed later in the paper.

As we prepare to give the statements, we adopt a convention that will prove convenient at several moments in the article: we will refer to measurements taken along the anti-diagonal as ``width'', and measurements along the diagonal as ``height''. For example, the parallelogram defined by $\{(x,y) \in \Z^2 : |x-y|\leq 2\ell r^{2/3}, 2\leq x+y\leq 2r\}$ will be said to have width $2\ell r^{2/3}$ and height $r$. The implicit sense of the diagonal direction as vertical will find expression in usages such as `upper' and `lower' to refer to more or less advanced diagonal coordinates.

\subsubsection{Interval-to-interval estimates}
We start with estimates for the deviations of the weight when the endpoints are allowed to vary over intervals. For this we define some notation.

Extending the notation \eqref{parallelogram1}, let $U=U_{r,\ell r^{2/3}, zr^{2/3}}$ be a parallelogram of height~$r$ with one pair of sides parallel to the anti-diagonal, each of width $\ell r^{2/3}$,  the lower side centred at $(1,1)$ and the upper side at $(r-zr^{2/3},r+zr^{2/3})$ for a $z$ such that $|z|\leq r^{1/3}$. %
Let $A$ and $B$ be the lower and upper sides of $U$  and let $\mathcal S(U) = A\times B$. With $G_2$ and $\rho$ as in Assumption~\ref{a.limit shape assumption}, let 
\begin{align*}
\widetilde Z &:= \sup_{(u,v)\in \mathcal S(U)} \frac{X_{u \ar v} - \mu r}{r^{1/3}} +  \frac{G_2}{1+2\ell^{3/2}}\big(z\wedge \rho r^{1/3}\big)^2.
\end{align*}
Although we have yet to employ the notation $Z$ without a tilde, it will be used to denote a scaled form for weight of certain LPP paths. The tilde indicates a parabolic adjustment which compensates (and possibly over-compensates) the curvature penalty for off-diagonal movement. The parabolic form of this adjustment is indicated in Assumption~\ref{a.limit shape assumption}. This  form is guaranteed by Assumption~\ref{a.limit shape assumption} only up to anti-diagonal displacement of order $z=\rho r^{1/3}$; concavity nonetheless guarantees that at least this loss occurs even beyond this point, and this accounts for the $z\wedge \rho r^{1/3}$ term in the specification of $\widetilde{Z}$.

However, Assumption~\ref{a.one point assumption} only holds for paths whose endpoints are bounded away from the coordinate axes, and so for extreme anti-diagonal displacement we will need to work with a different object: define, for $\delta>0$,
$
\mathbb{L}_{\mathrm{ext}}(\delta) = \left\{(r+zr^{2/3},r-zr^{2/3}): |z|\geq (1-\delta)r^{1/3} \right\}$, which is an interval of points of extreme slope,
and let $\linelower = \{(x,y)\in \Z^2: x+y=2, |x-y|\leq 2r^{2/3}\}$. Finally define
$$Z^{\mathrm{ext}, \delta} = \sup_{u\in \linelower, v\in \mathbb{L}_{\mathrm{ext}}(\delta)} \frac{X_{u\ar v} - \mu r}{r^{1/3}},$$ the interval-to-interval passage time between $\linelower$ and $\mathbb{L}_{\mathrm{ext}}(\delta).$

\begin{proposition}\label{l.sup tail}
Suppose $\delta>0$ is such that $|z| < (1-\delta) r^{1/3}$. Then, under Assumptions  \ref{a.limit shape assumption} and \ref{a.one point assumption}, there exist  $\theta_0 = \theta_0(\ell, \delta)$, $c =c(\delta)>0$, and $r_0 = r_0(\ell,\delta)$ such that, for $r>r_0$ and $\theta > \theta_0$,
\begin{equation}\label{e.sup tail away from axes}
\P\left(\widetilde Z > \theta \ell^{1/2} \right) < \exp\left(-c\min(\theta^{3/2}, \theta r^{1/3})\right).
\end{equation}
Further, there exist $\delta>0$ and $c>0$ such that, for $\theta>\theta_0$ and $r>r_0$,
\begin{equation}\label{e.sup tail extreme}
\P\left(Z^{\mathrm{ext},\delta} > \theta\right) \leq \exp(-cr-c\min(\theta^{3/2}, \theta r^{1/3})).
\end{equation}
\end{proposition}

\subsubsection{Point-to-line estimates}

Here we bound the upper tail of point-to-line geodesic weights, using only the parabolic curvature of the limit shape (Assumption~\ref{a.limit shape assumption}) and the point-to-point upper tail (Assumption~\ref{a.one point assumption}).
We define $\lineleft = \lineleft (r,t) = \{(x,y): x+y=0, |y-x|\leq tr^{2/3}\}$.

\begin{proposition}\label{p.p2l general upper tail}
Let $X$ be the maximum weight of all paths $\Gamma$ from $\lineleft$ to the line $x+y=2r$ whose endpoint does not lie on the line segment connecting $(r-(s+t)r^{2/3}, r+(s+t)r^{2/3})$ and $(r+(s+t)r^{2/3}, r-(s+t)r^{2/3})$. Then under Assumptions~\ref{a.limit shape assumption} and~\ref{a.one point assumption} there exist constants $c_1>0, c>0$, $s_0$, $\theta_0$, and $r_0$ such that, for $r>r_0$ and either (i) $s=0$ and $\theta\geq \theta_0$, or (ii) $s>s_0$, $0\leq t \leq s^2$ and $\theta=0$, we have that
$$\P\left(X \geq \mu r +\theta r^{1/3} - c_1 s^2 r^{1/3}\right) \leq \exp\left(-c\big(\min(\theta^{3/2}, \theta r^{1/3}) + s^3\big)\right).$$
\end{proposition}

\subsubsection{Lower tail of constrained point-to-point}

Next we come to a crucial ingredient of Theorem~\ref{t.flexible construction}, an upper bound on the lower tail of the weight of the point-to-point geodesic constrained to not exit a given parallelogram, and a lower bound on the mean of the same. We obtain precisely the tail exponent of one that will be used in Theorem~\ref{t.flexible construction}; the optimal exponent, however, is expected to be three, matching the optimal exponent of the point-to-point weight lower tail. Even given the optimal lower tail exponent of three as input via Assumption~\ref{a.one point assumption}, the argument we give for Proposition~\ref{p.constrained lower tail} would only yield the exponent $3/2$ for the constrained lower tail.

With notation for parallelograms as above, let $U = U_{r,\ell r^{2/3}, zr^{2/3}}$. Set $u=(1,1)$ and $u^{\prime}=(r-(z+h)r^{2/3}, r+(z+h)r^{2/3})$ where $|h| \leq \ell/2$. Then define $X^U_{u \ar u'}$ to be the maximum weight over all paths from $u$ to $u'$ that are constrained not to exit $U$. Recall $G_1$ from Assumption~\ref{a.limit shape assumption}. We have the following lower tail and mean estimates for $X^U_{u \ar u'}$.

\begin{proposition}[Lower tail and mean of constrained point-to-point]\label{p.constrained lower tail}
Let $L_1,L_2>0$ and $K>0$ be fixed. Let $z$ and $\ell$ be such that $|z|\leq K$ and $L_1\leq \ell\leq L_2$. There exist positive constants $r_{0}(K, L_1, L_2)$ and $\theta_{0}(K, L_1, L_2)$, and an absolute constant $c>0$, such that, for $r>r_{0}$ and $\theta>\theta_{0}$, 
\begin{equation}\label{e.constrained lower tail}
\P\left(X_{u\ar u^{\prime}}^{U} \leq \mu r-\theta r^{1 / 3}\right) \leq \exp\left(-c\ell \theta\right).
\end{equation}
As a consequence, there exists $C = C(K, L_1, L_2)$ such that, for $r>r_0$,
\begin{equation}\label{e.constrained mean}
\E[X^U_{u \ar u'}] \geq \mu r - G_1z^2 r^{1/3} - Cr^{1/3}.
\end{equation}
\end{proposition}

\subsection{Organization}
In this section, we have signposted the location of several upcoming aspects. With some repetition, and for the sake of convenience, we now summarise the structure of the rest of the paper, which is principally devoted to proving  Theorem~\ref{t.notwidenotthingeneral}.
There are four elements: lower and upper bounds on the transversal fluctuation of geodesic watermelons; and upper and lower bounds on the watermelons' weight.
The respective arguments are offered in Sections  \ref{s:lowerexp},~\ref{s:upperexp},~\ref{s:upper} and~\ref{s:construction}.
The first of the four elements is aided by the `thin strip means low weight' Theorem~\ref{t:disjoint}. First, in Section~\ref{s:bk}, we will recall the BK inequality and use it to prove this theorem. 
The deterministic interlacing results assembled in  Section~\ref{s:interlace} are important inputs, and, in Section~\ref{s:geometry}, we prove them, so that the proof of Theorem~\ref{t.notwidenotthingeneral} is completed in this section. 
Section~\ref{s:special} compares concentration bounds for the weight increments $Y_{n,k}$ specified in~(\ref{e:ensemble}) obtained by our geometric methods and by determinantal point process techniques (in the case of exponential LPP)\begin{submitted-version}; a proof of the concentration bound via determinantal techniques is provided in the arXiv version of this article
\end{submitted-version}.
The proof of the point-to-line result, Theorem~\ref{t:p2lgeneral}, is provided in Section~\ref{s:p2l}.

There are three appendices: Appendix~\ref{app.exp and geo satisfy assumptions} is devoted to proving that exponential and geometric LPP satisfy the Assumptions; 
Appendix~\ref{app.tail bounds} to proving the three basic tools, Propositions~\ref{l.sup tail},~\ref{p.p2l general upper tail}
 and~\ref{p.constrained lower tail}, laid out in Section~\ref{s.basictools}; and Appendix~\ref{app.proof of lemmas} to proving Lemma~\ref{l.first upper bound on melon weight} under the stronger Assumption~\ref{a.one point assumption convex}. 
 The proofs in Appendix~\ref{app.tail bounds} have been deferred because they roughly mimic publicly available LPP arguments; we judge the flow of the arguments at large to be aided by  the few proof deferrals that we have mentioned. The appendices render the article independent of LPP inputs beyond fairly well-known assertions such as the expression of the cardinality of a determinantal point process as a sum of independent Bernoulli random variables, and an estimate on the mean of the eigenvalue count in a given interval for the Laguerre unitary ensemble. 

\section{Disjoint paths and the van den Berg-Kesten inequality}
\label{s:bk} 

In this section, we prove Theorem \ref{t:disjoint}, our assertion that several disjoint high-weight paths may not typically coexist in a narrow strip.   The proof is founded on the notion that  at least one among $k$ disjoint paths packed into the strip \smash{$U_{n,\delta k^{1/3}n^{2/3}}$} is forced to inhabit a thin region of width \smash{$\delta k^{-2/3}n^{2/3}$}. 
Here in outline are the principal steps that we will follow to realize the proof.
\begin{enumerate}[label=(\arabic*)]
	\item 
	We will argue that it is with probability decaying exponentially in $k$ that a path existing in a given thin region---of length $n$ and of width roughly~$k^{-2/3}n^{2/3}$---has weight larger than $\mu n - \Theta(k^{2/3}n^{1/3})$. %

	\item The van den Berg-Kesten (BK) inequality
	 bounds above the probability of events occurring in a certain {\em disjoint} fashion by the product of the probabilities of the events.
	 This inequality permits us to 
	 conclude that the probability that there exist $k$ disjoint paths whose weight satisfies the above lower bound, each contained in a specified thin region, is exponentially small in $k^2$.

	\item A union bound will then be taken indexed by a collection of $k$-tuples of thin regions whose elements capture all possible routes for the $k$ paths. The collection will be selected by means of a grid-based discretization; the cardinality of the collection will be small enough that the upper bound on probability for a given $k$-tuple will not be significantly undone by taking the union bound. 
\end{enumerate}

To implement this three-step plan, we need to make precise the notion of `thin region' in Step~$1$, and this in fact involves specifying the grid used in Step~$3$.
This we do in a first subsection. But first we remind the reader of our norm that statements by default suppose Assumptions~\ref{a.passage time continuity}, \ref{a.limit shape assumption} and \ref{a.one point assumption} and that we indicate when Assumption~\ref{a.one point assumption convex} is needed in place of \ref{a.one point assumption}.

\subsection{Specifying a grid that classifies paths running in a narrow strip}

 The grid will be called $G$. It will comprise anti-diagonal planar intervals whose elements track the progress of any path of transversal fluctuation at most $\delta k^{1/3}n^{2/3}$.  

Recall that width and height refer to  anti-diagonal and diagonal length. 
Beyond the height $n$, path number $k$ and strip thinness parameter $\delta > 0$, the grid will be constructed by the use of two positive parameters $\eta$ and $\epsilon$. 
We hope that the explanation we next offer alongside Figure~\ref{f.grid} will make their roles clear. 

The strip of width  $\delta k^{1/3}n^{2/3}$ and height $n$ will be divided into cells. Each cell will have width $\eta \epsilon n^{2/3}$ and height  $\epsilon^{3/2}n$. The width is the product of $\eta$ and the two-thirds' power of the height, so~$\eta$ is a measure of the thinness of the cell, judged on the KPZ scale; we will choose this parameter to be a small constant, independently of~$k$ and~$n$. We will choose $\epsilon$ to have order $k^{-2/3}$, so that a union of intervals at consecutive heights may play the role of a `thin region' indicated in the three-point summary.    
Each column of the grid contains $N$ cells,  where  $N = n/(\epsilon^{3/2}n) = \epsilon^{-3/2}$; and each row contains $M$ cells, where $M = 2\delta k^{1/3}n^{2/3}/(\eta\epsilon n^{2/3}) = 2\delta k^{1/3}\eta^{-1}\epsilon^{-1}$.

To begin the definition, let $\epsilon>0$ and $\eta>0$; we will specify the conditions on these values presently. 
We suppose $n \in \N$ to be large enough that $\epsilon \eta n^{2/3}\geq 1$, in order that cell width be at least the difference of $x$- or $y$-coordinates between consecutive elements in $\Z^2$ along an anti-diagonal line. 
Set $M,N \in \N$ according to  
\begin{equation}
N=\lfloor\epsilon^{-3/2}\rfloor\quad \text{and} \quad M = 2\cdot\lceil \delta k^{1/3}\eta^{-1}\epsilon^{-1}\rceil. \label{e.N and M value}
\end{equation}
 The grid $G$ is a set of planar anti-diagonal intervals $\L_{ij}$:
$$
G=\big\{\L_{ij}:  i \in \llbracket 0 , N \rrbracket , j \in \llbracket 0, M \rrbracket \big\} \, ,
$$
where 
$\L_{ij}$ connects the points
$$\left(v_i-h_j, v_i+h_j\right) \quad \text{and}\quad \left(v_i -h_{j+1}, v_i+h_{j+1}\right),$$ 
with $v_i = \lfloor i\epsilon^{3/2}n\rfloor$ and $h_j = \lfloor(\delta k^{1/3}-j\eta\epsilon)n^{2/3}\rfloor$ for $i\in\llbracket 0, N-1\rrbracket$; for $i=N$ we take $v_i = n$. Thus the last row of cells has greater height than the first $N-1$ rows', by a factor lying in $[1,2)$, but this will not affect our arguments.
Note that each element of the grid belongs to  the rectangle $\{|x-y| \leq 2\delta k^{1/3}n^{2/3}, 0\leq x+y\leq 2n\}$. The grid is partitioned into subsets at given height by setting  
$G_i = \big\{ \L_{ij} : j \in \llbracket 0, M \rrbracket \big\}$ for $i \in \llbracket 0, N \rrbracket$.
To avoid cumbersome expressions, we employ, in the remainder of this section, a notational abuse by which $(x,y)$ for $x,y \in \R$  will actually denote the rounded integer lattice point $(\lfloor x \rfloor, \lfloor y \rfloor)$.

\begin{figure}[h]
   \centering
        \includegraphics[width=.45\textwidth]{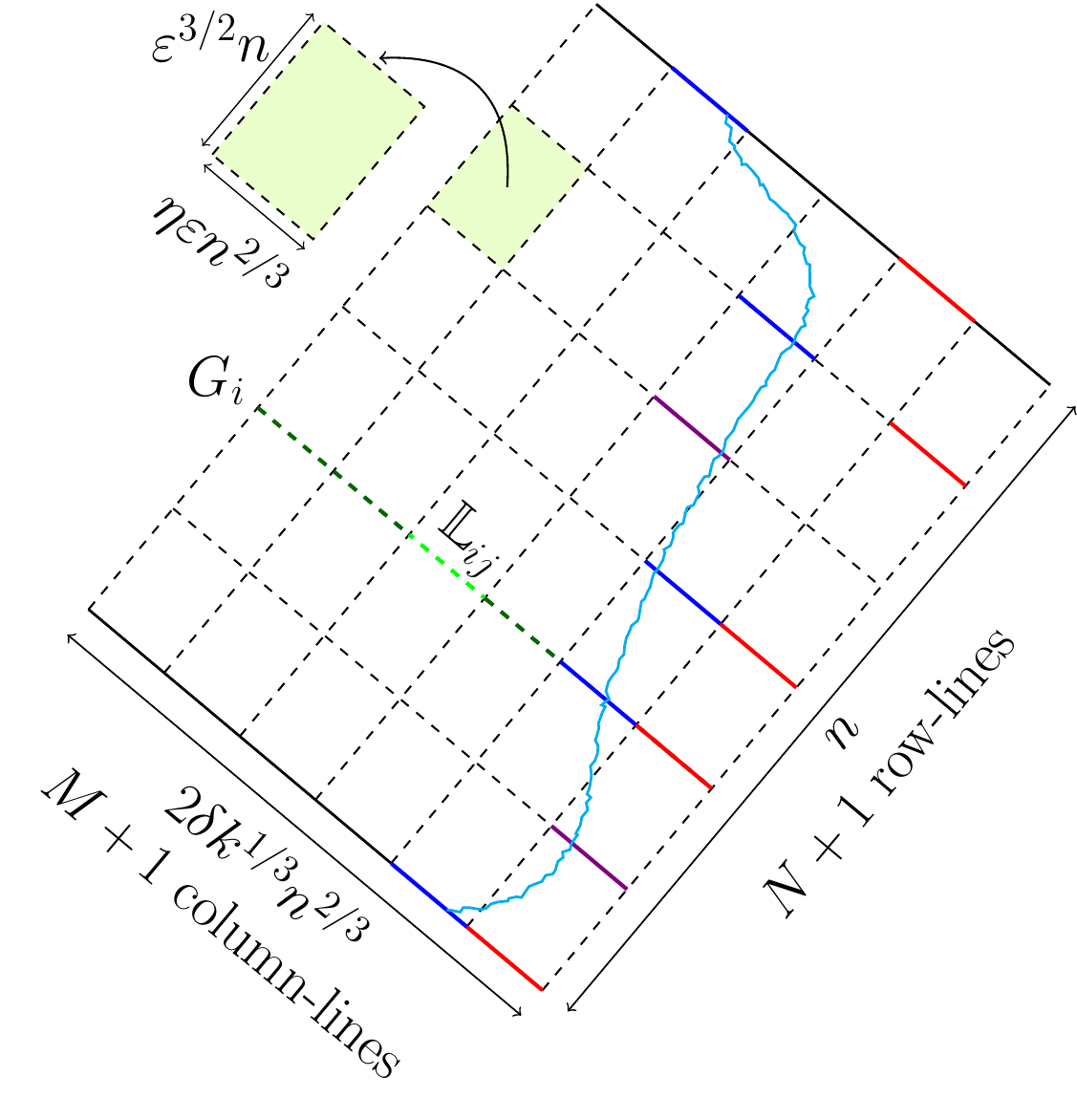}
 \caption{The grid $G$. There are $N+1$ row lines and $M+1$ column lines, and the width and height of a single cell are $\eta \epsilon n^{2/3}$ and $\epsilon^{3/2}n$ respectively. The width of the grid as a whole is $2\delta k^{1/3}n^{2/3}$. The role of the grid is to be a discretization of paths from the bottom side to the top. On the right, such a path is depicted in sky blue along with its corresponding discretization, which is the sequence of intervals, one on each grid line, which are coloured blue (or purple). Also depicted is a sequence of red intervals (when a red interval coincides with a blue interval, it is coloured purple). These red intervals satisfy the property that on each grid line, the red interval coincides with, or is to the right of, the blue interval on the same grid line. This is a property that is satisfied by the discretization of any path from bottom to top which is disjoint from and to the right of the sky blue path.}
 \label{f.grid}
\end{figure}

\subsection{Bounding the maximum weight of paths between grid elements at consecutive heights}

A {\em grid journey} is an $(N+1)$-vector $\big\{ L_i: i \in \llbracket{0,N} \rrbracket \big\}$ such that $L_i \in G_i$ for each index~$i$. A grid journey plays in the implementation the conceptual role of a thin region in the outline. 
Indeed, for any grid journey $\mathscr{L}$, we write $X_\mathscr{L}$ for the supremum of the weights of paths contained in $\intint{n}^2$
that intersect every one of the $N+1$ components of $\mathscr{L}$. 
In Step~$1$, we will bound the upper tail of~$X_\mathscr{L}$. To be ready to do so, we first gain control, in Lemma~\ref{l.sup mean}, on the mean of the weight of the heaviest path that begins in a given element of $G_i$
and that ends in a given element of $G_{i+1}$. Since the height discrepancy is $r = \varepsilon^{3/2}n$, Assumption~\ref{a.limit shape assumption} indicates that the weight of the geodesic running between {\em given} endpoints in these intervals has leading term $\mu r$, with a negative correction whose order is $r^{1/3}$. When the endpoints are varied by at most $\eta r^{2/3}$, a Gaussian-order perturbation is transmitted to the maximum weight. This change of order $\eta^{1/2} r^{1/3}$ is less than the order of the negative mean provided that $\eta$ is small enough (and it is this consideration that determines our upper bound $\eta_0$ on the parameter $\eta$). This is how we will prove Lemma~\ref{l.sup mean}.

For $\eta>0$, $r > 0$ and $|s|\leq r$, let $\linelower$ be the line segment joining $(-\eta r^{2/3}, \eta r^{2/3})$ and $(\eta r^{2/3}, -\eta r^{2/3})$ and let $\lineupper$ be the line segment joining $(r-s-\eta r^{2/3}, r+s+\eta r^{2/3})$ and $(r-s+\eta r^{2/3}, r+s-\eta r^{2/3})$. Define $Z$ by
$$
Z=\sup  \Big\{ r^{-1/3} \big( X_{u \ar v}-\mu r \big) : u\in \linelower,v\in \lineupper \Big\} \,,
$$
so that we seek control on the mean of $Z$. 

\begin{lemma}\label{l.sup mean}
Under Assumption~\ref{a.limit shape assumption}, there exist positive constants $c_1$, $r_0$ and $\eta_0$ such that, for $\eta < \eta_0$, $|s|\leq r$ and $r> r_0$,
$\E[Z] \leq -2c_1$. 
\end{lemma}

The basic strategy of the argument is to back away from $\linelower$ and $\lineupper$ on either side and consider a point-to-point weight; such ideas have appeared in the literature: see for instance \cite{slow-bond}.

\begin{proof}
Let $m_1$ and $m_2$ be the points on $\linelower$ and $\lineupper$ which are closest to the average of the ends of the respective intervals; we make this specification since these averages need not be lattice points. 

The $\ell^1$ distance of these two points is $2r$. 
	 Let $u^* \in \linelower$ and $v^* \in \lineupper$ be such that $X_{u^* \ar v^*}$ equals $\mu r + Zr^{1/3}$. For $\eta>0$, let $\phi_1 = m_1 - (\eta^{3/2}r,\eta^{3/2}r)$ and $\phi_2 = m_2+(\eta^{3/2}r,\eta^{3/2}r)$. Then we have
	$$X_{\phi_1 \ar u^*} + X_{u^* \ar v^*} + X_{v^* \ar \phi_2} \leq X_{\phi_1 \ar \phi_2}.$$
	Note that the $\ell^1$ distance between $\phi_1$ and $\phi_2$ lies in $(2+4\eta^{3/2})r + [-2,0]$. First we consider the case that this pair of points satisfies the $\rho$-condition in Assumption \ref{a.limit shape assumption}. We get that for large $r$, 
	$$\E[X_{\phi_1 \ar \phi_2}]  \leq \mu(1+2\eta^{3/2})r - g_2 \big((1+2\eta^{3/2})r-2\big)^{1/3} \leq \mu(1+2\eta^{3/2})r - g_2 r^{1/3}.$$
	 To get an upper bound on $\E[X_{u^* \ar v^*}]$, we need a lower bound on $\E[X_{\phi_1 \ar u^*}]$, which will also apply to $\E[X_{v^* \ar \phi_2}]$ by symmetry. Note that $u^*$ is independent of the environment outside $U$, so $\E[X_{\phi_1 \ar u^*}]\geq \inf_{u\in \linelower} \E[X_{\phi_1 \ar u}]$. We bound this below using Assumption \ref{a.limit shape assumption}, and note that that expression is minimized if $u$ is either endpoint of $\linelower$. So we take $u=(-\eta r^{2/3}, \eta r^{2/3})$. This gives
	\begin{align*}
	\E[X_{\phi_1 \ar u}] &\geq  \mu\eta^{3/2}r-\mu - G_1 \eta^{1/2} r^{1/3} - g_1\eta^{1/2}r^{1/3},
	\end{align*}
	for $r$ such that $\eta^{3/2}r > r_0$ for a large $r_0$.
	Thus we get an upper bound on $\E[X_{u^* \ar  v^*}]$:
	\begin{align*}
	\E[X_{u^* \ar v^*}] -\mu r \leq 2\mu-g_2 r^{1/3} + 2(G_1+g_1)\eta^{1/2}r^{1/3}
	&\leq -2c_1r^{1/3}
	\end{align*}
	for $c_1 = g_2/4$, $\eta$ small enough, and $r$ large enough. This completes the proof when the $\rho$-condition in Assumption \ref{a.limit shape assumption} is satisfied.

	If the $\rho$-condition is not satisfied, we have, using the monotonicity of the upper bound in Assumption \ref{a.limit shape assumption} for all $x$ and a simple calculation,
	$$\E[X_{u^* \ar v^*}] \leq \E[X_{\phi_1 \ar \phi_2}]  \leq (\mu + 2\mu\eta^{3/2} - \rho^2 G_2)r.$$
	In this case the proof is completed by noting that, for $\eta$ sufficiently small, the coefficient of $r$ is strictly smaller than $\mu$.
\end{proof}

The upper bound of $\eta_0$ from Lemma~\ref{l.sup mean} will be imposed on $\eta$, in order that the mean of $Z$ be negative. A further upper bound on $\eta$ will later be needed; namely, 
\begin{equation}\label{e.eta value}
\eta ec_1C_3^{-1} <1, 
\end{equation}
where $c_1$ is as in Lemma~\ref{l.sup mean} and $C_3$ is as in Theorem~\ref{t:disjoint}.  Let  $\eta < \eta_0$ satisfy~\eqref{e.eta value}.

The next three subsections in turn carry out the outlined three-step plan for proving Theorem~\ref{t:disjoint}.

\subsection{Step one: bounding above the weight of the heaviest path on a given grid route}
We rigorously take Step~$1$ by stating and proving the next result. The tail of $\exp(-\Theta(\epsilon^{-3/2}))$ we obtain here is crucial, as this will become $\exp(-\Theta(k))$ when we later set $\epsilon$ to be of order $k^{-2/3}$.

\begin{prop}\label{p.single curve weight}
Let $c_1$ be as in Lemma~\ref{l.sup mean} and let $\eta>0$ be fixed as after $\eqref{e.eta value}$. There exist positive $c_3$ and $C$ such that, when $\epsilon > 0$ and $n \in \N$ satisfy $n \geq C \epsilon^{-3/2}$,
$$
P\left(X_{\mathscr L} > \mu n-c_1\epsilon^{-1}n^{1/3}\right) \, \leq \, \exp\left\{ -c_3\epsilon^{-3/2}\right\} \, .
$$
\end{prop}

\begin{proof}
Set $Z_i=\sup_{u\in \overline L_{i},v\in L_{i+1}} \epsilon^{-1/2}n^{-1/3} \big( X_{u \ar v}-\mu \epsilon^{3/2}n \big)$ for $i \in \llbracket 0, N-1 \rrbracket$, where
$\overline L_i$ is the anti-diagonal interval which is the union of the two line segments obtained by displacing $L_i$ by $(1,0)$ and~$(0,1)$; this choice is made so that 
the $Z_i$ are determined by disjoint regions of the noise field and so are independent.
Then  $X_{\mathscr L}-\mu n \leq \epsilon^{1/2}n^{1/3} \sum_{i=0}^{N - 1}Z_i$.  
As noted, the collection $\big\{ Z_i: \llbracket 0, N-1 \rrbracket  \big\}$ is independent; and, by Lemma~\ref{l.sup mean}, the elements' means  lie uniformly to the left of zero. We will verify that these random variables have exponential upper tails and apply Bernstein's inequality to $\sum Z_i$ in order to obtain the proposition.

If we set $C>r_0$ in the statement that we are seeking to prove, then a choice of $r=\epsilon^{3/2}n$ in  Lemma \ref{l.sup mean}  will satisfy the hypothesis that $r \geq r_0$. 
Applying this lemma with $\eta$ as specified, we find that 
$$\E[Z_i] \leq -2c_1$$
 in view of translation invariance of the noise field $\big\{ \xi_v: v \in \Z^2 \big\}$.
	With $\eta$ now fixed, we claim that for $\theta>\theta_0$,
	$$\P(Z_i > \theta \eta^{1/2}) \leq \exp(-c\theta).$$
	This is because $Z_i\leq \tilde Z_i$ (where $\tilde Z_i$ is defined analogously to $\tilde Z$ in Proposition~\ref{l.sup tail}) in the case that $L_i$ and $L_{i+1}$ form opposite sides of a parallelogram whose midpoint-to-midpoint anti-diagonal displacement is at most $(1-\delta)\epsilon^{3/2}n$ for the $\delta>0$ fixed in the second part of Proposition~\ref{l.sup tail}; while $Z_i \leq Z_i^{\mathrm{ext},\delta}$ (again defined analogously to $Z^{\mathrm{ext},\delta}$ in Proposition~\ref{l.sup tail}) in the other case. Applying Proposition~\ref{l.sup tail} completes the claim.
Note that we have again used that $r=\epsilon^{3/2}n > r_0$ by setting $C>r_0$ in the statement of Proposition~\ref{p.single curve weight}.

The above displayed bound establishes that $Z_i$ has exponential upper tails and so satisfies the conditions of Bernstein's inequality (see equation 2.18 and Theorem 2.13 of \cite{wainwright-concentration}). Applying the latter and since  $N = \epsilon^{-3/2}$ (see \eqref{e.N and M value}), we get, for some $c_3>0$,
	\begin{align*}
	\P\left(X_{\mathscr L} > \mu n-c_1\epsilon^{-1}n^{1/3}\right) \leq \P\left(\sum_{i=0}^{N-1} Z_i > -c_1\epsilon^{-3/2}\right)
	&\leq \P\left(\sum_{i=0}^{N-1} (Z_i-\E[Z_i]) > c_1\epsilon^{-3/2}\right)\\
	&\leq \exp\left(-c_3 \epsilon^{-3/2}\right).\qedhere
	\end{align*}
	\end{proof}

\subsection{Step~$2$: the rarity of $k$ high weight paths in a narrow strip}

The BK inequality is the principal tool enabling the second step. The inequality provides an upper bound on the probability of a number of events occurring \emph{disjointly} in terms of the individual events' probabilities. To state it we need a precise definition of disjointly occurring events; this is taken from \cite{bk-inequality}, which also proves the inequality in the setting of infinite spaces that we require.

\begin{defn}
For $d, n\in \N$, let $A_i \subseteq \R^d$ be Borel measurable for $i \in \intint{n}$. For $\omega\in \R^d$ and $K\subseteq \intint{d}$, define the cylinder set $\mathrm{Cyl}(K,\omega) = \{\omega': \omega'_i = \omega_i, i\in K\}$. Also define, for $A\subseteq \R^d$,
$$[A]_K := \{\omega: \mathrm{Cyl}(K, \omega)\subseteq A\} \quad \text{and}\quad \bigbox_{i=1}^n A_i := \bigcup_{J_1, \ldots, J_n} \bigcap_{i=1}^n [A_i]_{J_i},$$
where the union is over disjoint subsets $J_{1}, \ldots,  J_n$ of $\{1, \ldots, d\}$.
\end{defn}

With the notation established, we may state the BK inequality.

\begin{proposition}[BK inequality, Theorem 7 of \cite{bk-inequality}]\label{p.bk}
Fix $n\in \N$ and let $A_i\subseteq \R^d$ be Borel measurable for $i\in\intint{n}$. Under any complete product probability measure $\nu$ on $\R^d$,
$$\nu\left(\bigbox_{i=1}^n A_i\right) \leq \prod_{i=1}^n \nu(A_i).$$
\end{proposition}

Note that we may take the $A_i$ to be the same event $A$ in this bound, in which case $\bigbox_{i=1}^n A_i$ is the event of the presence of $n$ disjoint instances of $A$. Many of our applications will take this form. The condition that $\nu$ is a complete probability measure is a technical one which does not affect our arguments, as we may assume that our vertex weight distribution $\nu$ is complete.

To state the formal realization of Step~2 of the outline, we define a \emph{$k$-disjoint grid journey} to be a collection of grid journeys $\{\mathscr L_{m} : m\in\intint{k}\}$ such that there exist $k$ disjoint curves $\gamma_1, \ldots, \gamma_k$, ordered from left to right, with $\gamma_m$ intersecting each component of $\mathscr L_m$ for each $m\in\intint{k}$. 

This can be expressed equivalently as the following constraint on the components of the $\mathscr L_m$, which we label as $(L_0^{m}, \ldots, L^m_N)$: if $L^m_i = \L_{i,j_{i,m}}$, then $j_{i,m} \geq j_{i,m-1}$ for each $i\in \llbracket 0,N\rrbracket$, $m\in \intint{k}$. This means that, for every grid line of $G$ of slope $-1$, the interval chosen for $\mathscr L_m$ coincides with or is to the right of that for $\mathscr L_{m-1}$ for every $m\in\intint{k}$. See the red and blue intervals depicted in Figure~\ref{f.grid}.

We set $\varepsilon>0$ for the rest of the proof of Theorem~\ref{t:disjoint} to be
\begin{equation}\label{e.epsilon value}
	\epsilon = \min(1,c_1C_3^{-1})\cdot k^{-2/3},
	\end{equation}
where $c_1$ is as in Proposition \ref{p.single curve weight}. Thus with both $\eta$ and $\epsilon$ fixed, the grid cells of $G$ have been fully specified.

With these definitions, we may state the next result, which obtains the $\exp(-ck^2)$ bound promised for Step~2 on the probability of there existing $k$ disjoint paths, each constrained to be in a thin region and each of which is not too light. 
	\begin{proposition}\label{p.weight of single path collection}
	Let $C_3>0$ be given and $\varepsilon$ as in \eqref{e.epsilon value}. There exist positive finite $c_3$ and $C$ such that, for any $n\in\N$ and $k\in\N$ satisfying $n> Ck$ and any $k$-disjoint grid journey $\big\{ \mathscr{L}_m : m \in \intint{k} \big\}$,
	$$
	\P\left( \bigbox_{m=1}^k \left\{X_{\mathscr L_m} > \mu n-C_3k^{2/3}n^{1/3}\right\}\right) \leq e^{-c_3k^2} \, .
	$$
	\end{proposition}

\begin{proof}%
 	In view of Proposition \ref{p.single curve weight} and the value of $\epsilon$, there exists  $C<\infty$ such that, for $m\in\intint{k}$ and $n>Ck$, 
	$$\P\left(X_{\mathscr L_m} > \mu n-c_1\epsilon^{-1}n^{1/3}\right) \leq e^{-c_3\epsilon^{-3/2}}.$$
By \eqref{e.epsilon value}, we may apply $\epsilon<c_1C_3^{-1}k^{-2/3}$ on the left and $\epsilon< k^{-2/3}$ on the right. This yields
	\begin{equation}\label{e.weight loss of collection}
		\P\left(X_{\mathscr L_m} > \mu n-C_3k^{2/3}n^{1/3}\right) \leq e^{-c_3k}.
	\end{equation}
	By the BK inequality Proposition~\ref{p.bk} and \eqref{e.weight loss of collection}, we obtain Proposition~\ref{p.weight of single path collection}. 
	\end{proof}

	\subsection{Step 3: Bounding the number of $k$-disjoint grid journeys}

	Following Step~3, we plan a union bound over the collection of $k$-disjoint grid journeys $\{\mathscr L_m : m\in\intint{k}\}$. For this, we need a bound on the cardinality of this collection, which we record next.

	\begin{lemma}\label{l.entropy bound}
	Let $\epsilon$ and $\eta$ be as set earlier. Then there exists $\delta = \delta(C_3, \eta)>0$ such that the number of $k$-disjoint grid journeys is bounded by $e^{\frac12 c_3k^2}$ for $k\geq 1$, with $c_3$ as in Proposition~\ref{p.weight of single path collection}.
	\end{lemma}

This lemma permits us to finish the proposed proof.

	\begin{proof}[Proof of Theorem~\ref{t:disjoint}]
	  We set $\delta$ to be as in Lemma~\ref{l.entropy bound} and $\varepsilon$ as in Proposition~\ref{p.weight of single path collection}. Let $\ce = \ce(n,k,\delta,C_3)$ be the event in the statement of Theorem~\ref{t:disjoint}. Note that
	  $$\ce \subseteq \bigcup \left(\bigbox_{m=1}^k \left\{X_{\mathscr L_m} > \mu n-C_3k^{2/3}n^{1/3}\right\}\right),$$
	  where the union is over the collection of $k$-disjoint grid journeys $\{\mathscr L_m : m\in \intint{k}\}$. Lemma~\ref{l.entropy bound} says that the cardinality of this set is $\exp(\frac{1}{2}c_3k^2)$, while Proposition~\ref{p.weight of single path collection} says that the probability of a single member of the union is at most $\exp(-c_3k^2)$. Taking a union bound thus yields
	\begin{align*}
	\P\left(\ce\right) < e^{-c_3k^2 + \frac12c_3k^{2}}= e^{-\frac12c_3k^2},
	\end{align*}
	 if $n\geq C\epsilon^{-3/2} = Ck$ and $k\in \N$. This proves Theorem~\ref{t:disjoint} with $c = \frac{1}{2}c_3$.
	\end{proof}
One task remains.

\begin{proof}[Proof of Lemma \ref{l.entropy bound}]
	We first observe a bound on the number of ways to select the collection of intervals from the $i$\textsuperscript{th} gridline $G_i$ that constitutes the union of the $i$\textsuperscript{th} component of $\mathscr L_m$ over $m\in\intint{k}$ for a $k$-disjoint grid journey $\{\mathscr L_m:m\in\intint{k}\}$. Raising this bound to the power $N$ will then yield a bound on the cardinality of the set of $k$-disjoint grid journeys.

	Assign to any collection of $k$ intervals from $G_i$ the $M$-vector whose $j$\textsuperscript{th} coordinate records the number of times $\L_{ij}$ was picked. The $M$-vector assigned has components which sum to $k$, and the cardinality of the set of such vectors is $\binom{k+M-1}{k}$. Not all such vectors can be achieved by the specified map, as we are ignoring the constraint imposed by the selection of intervals in $G_{i-1}$, and so this binomial coefficient is an upper bound. This then gives an upper bound of
	\begin{equation}\label{e.binomial coefficient}
	\left[\binom{k+M-1}{k}\right]^N
	\end{equation}
	on the number of $k$-disjoint grid journeys.
	Now to bound this quantity we recall that $M=2\lceil \delta k^{1/3}\eta^{-1}\epsilon^{-1}\rceil$. Given $\alpha>0$, we set $\delta$ to be
	$$\delta = 2^{-1}\eta (1+c_1^{-1}C_3)^{-1} \alpha.$$
	Then,  since $\epsilon^{-1} < (1+c_1^{-1}C_3)k^{2/3}$ from the statement of Proposition~\ref{p.weight of single path collection}, we get $M \leq \alpha k$ from \eqref{e.N and M value}. We  recall the well-known fact~\cite[eq. (4.7.4)]{ash-info-theory} that for $0<\beta<1/2$ and $n\geq 1$, $$\binom{n}{\beta n}\leq \exp\left(n H(\beta)\right), \text{ where }H(x) = -x\log x - (1-x)\log(1-x),$$ for $x\in[0,1]$, is the entropy function. We use this to bound \eqref{e.binomial coefficient}:
	\begin{align*}
	\left[\binom {k+M-1} {k}\right]^{N} \leq \left[\binom {(1+\alpha)k} {\alpha k}\right]^{N}
	&\leq \exp\left[(1+\alpha)kN\cdot H\left(\frac{\alpha}{1+\alpha}\right)\right].
	\end{align*}
	Using that $N=\epsilon^{-3/2}$ and $\epsilon = \min(1,c_1C_3^{-1})k^{-2/3}$, we set $\alpha$ small enough that the exponent is smaller than $\frac12 c_3 k^2$, since $H(x)\to 0$ as $x\to 0$. This bound holds for all $k\geq 1$, yielding
	\begin{align*}
	\left[\binom {k+M-1} {k}\right]^{N}&\leq \exp\left({\frac12 c_3 k^2}\right). \qedhere
	\end{align*} 
\end{proof}

\subsection{A crude upper bound of the watermelon weight: proving Lemma \ref{l.first upper bound on melon weight}}\label{proofofcrude}

Theorem~\ref{t.notwidenotthingeneral}(\ref{weight1'}) implies that $X_{n}^k - \mu n k$ is typically at most a negative quantity of order $k^{5/3}n^{1/3}$. We will first need a cruder upper bound,  Lemma \ref{l.first upper bound on melon weight}, where this order is asserted without any claim about the sign of the difference.
We begin by restating this application of the BK inequality.

\medskip 
\noindent  
\textbf{Lemma \ref{l.first upper bound on melon weight}.}
{\em Under Assumption~\ref{a.one point assumption}, we may find $c>0$ such that, if $t>0$, there exists $k_0 = k_0(t)$ for which $n>n_0$ and $k_0<k\leq  t^{-3/4}n^{1/2}$ imply that}
$$\P\left(\melonweight > \mu nk + tk^{5/3}n^{1/3}\right) \leq \exp\left(-c\min(t,t^{3/2})k^2\right).$$
{\em If instead Assumption~\ref{a.one point assumption convex} is available, then the upper bound on $k$ may be taken to be $n$.}

As mentioned in Section~\ref{iop}, we will prove here Lemma~\ref{l.first upper bound on melon weight} under only Assumption~\ref{a.one point assumption}, which limits the range of $k$ up to $O(n^{1/2})$. The stronger conclusion of the lemma under Assumption~\ref{a.one point assumption convex} has a slightly more complicated argument, and this is provided in Appendix~\ref{app.proof of lemmas}.

\begin{proof}[Proof of Lemma~\ref{l.first upper bound on melon weight} under Assumption~\ref{a.one point assumption}]
For any fixed $k$-geodesic watermelon $\mc W$, let $\smash{X_{n,\mc W}^{k,j}}$ be the weight of its $j$\textsuperscript{th} heaviest curve, imitating notation introduced in Section~\ref{s:interlace}. Define the event~$A_j$ by
\begin{align*}
A_j &= \left\{\exists \mc W: X^{k, i}_{n,\mc W} > \mu n + \frac14tj^{-2/3}k^{4/3}n^{1/3}\text{ for all } i\in\intint{j}\right\}\\
&=\left\{\exists \mc W: X^{k, j}_{n,\mc W} > \mu n + \frac14tj^{-2/3}k^{4/3}n^{1/3}\right\},
\end{align*}
where the second equality is because of the monotonicity relation $X_{n, \mc W}^{k, i} \geq X_{n,\mc W}^{k,i+1}$ for all $i \in \intint{k}$.
We start by claiming that, for $j\in \intint{k}$,
\begin{equation}
\P\left(A_j\right) \leq \exp\left(-ct^{3/2}k^2\right). \label{e.top j high weight}
\end{equation}
This is because the probability on the left-hand side is bounded by the probability that there exist $j$ disjoint curves, each with weight at least $\mu n + \frac14tj^{-2/3}k^{4/3}n^{1/3}$. This probability, by Assumption~\ref{a.one point assumption} and the BK inequality (Proposition~\ref{p.bk}), is bounded by
$$\left(\exp(-ct^{3/2}k^{2}j^{-1})\right)^{j} = \exp\left(-ct^{3/2}k^2\right) \, ,$$
for large enough $n$, if $tk^{2/3}$ is large enough. For this bound, we also need $\frac14tj^{-2/3}k^{4/3}n^{1/3} \leq n$, which is provided by the assumed upper bound on $k$. We now {\em  claim} that
$$\left\{\melonweight > \mu nk + tk^{5/3}n^{1/3}\right\} \subseteq \bigcup_{j=1}^k A_j \, .
$$
Indeed, suppose that $A_j^c$ occurs for each $j\in\intint{k}$.  For each $j\in\intint{k}$, by definition, on $A^c_j$,  $X_{n, \mc W}^{k, j} \leq \mu n+ \frac14tj^{-2/3}k^{4/3}n^{1/3}$ for all $k$-geodesic watermelons $\mc W$. This then gives, for any $\mc W$, that
\begin{align*}
\melonweight = \sum_{j=1}^k X_{n, \mc W}^{k,j} \leq \mu nk + \frac14tk^{4/3}n^{1/3}\sum_{j=1}^k j^{-2/3}
&\leq \mu nk + tk^{5/3}n^{1/3},
\end{align*}
which completes the proof of the claim. The derivation of Lemma~\ref{l.first upper bound on melon weight} is concluded by applying~\eqref{e.top j high weight}, using the union bound, and reducing $c$ in \eqref{e.top j high weight} to $c/2$ (which is possible for $k$ large enough depending on only $t$).
\end{proof}

\section{Not too thin: Bounding below the transversal fluctuation}
\label{s:lowerexp}
The aim of this section is to prove the lower bound on the transversal fluctuation exponent for the geodesic, i.e.,  Theorem~\ref{t.notwidenotthingeneral}(\ref{tf2'}). In view of Theorem \ref{t:disjoint},  this would be immediate if we could show that, with high probability, $X_{n}^{k,k} \geq \mu n - Ck^{2/3}n^{1/3}$ for some $C>0$. However, in the absence of such an estimate so far, we follow the strategy outlined in Section \ref{iop}, relying on an averaging argument and  geodesic watermelon interlacing. More precisely, the interlacing result Proposition \ref{p.melon interlacing} implies that it is sufficient to show the existence of some $j\in \intint{k}$ such that no $j$-geodesic watermelon is contained in the strip \smash{$U_{n,\delta k^{1/3}n^{2/3}}$}. To this end, we shall establish an averaged version of the lower bound of \smash{$X_{n}^{k,k}$}; i.e., with high probability, that $X_{n}^{j,j} \geq \mu n - Ck^{2/3}n^{1/3}$ for some $j\in \llbracket \lfloor \frac{k}{2} \rfloor , k\rrbracket$. This, together with the above observation and Theorem \ref{t:disjoint}, will complete the proof. 

We note in the above discussion that $X_n^{j,j}$ is only well-defined after specifying the $j$-geodesic watermelon, unlike $X_n^j$---and this we have not done in the case where $\nu$ has atoms. However, using a uniform lower bound on $X_n^{j,j}$ which holds over all possible choices of the curves of the $j$-geodesic watermelon, we will be able to show that the averaging statement  holds for $\underline X_n^{j,j}$, defined for $j\in\intint{n}$ as 
$$
\underline X_n^{j,j} \, = \, \min_{\mc W} X_{n, \mc W}^{j,j} \, ,
$$
where $X_{n,\mc W}^{j,j}$ is the weight of the lightest curve of the $j$-geodesic watermelon $\mc W$, and the minimum is over all $j$-geodesic watermelons.

We now state the averaging result.

\begin{lemma}\label{p.smallest curve weight lower bound}
For $C_1$ as in Theorem~\ref{t.notwidenotthingeneral}(\ref{weight1'}), and under Assumption~\ref{a.one point assumption} (respectively \ref{a.one point assumption convex}), there exist positive  constants $c_1,$ $c>0$, $C$ and $n_0$ such that, for $k\geq 1$ and $n\geq n_0$ for which $k\leq c_1n^{{1/2}}$ (resp. $k\leq c_1n$),
$$\P\left(\bigcup_{j =\lfloor \frac{k}{2} \rfloor}^{k} \left\{\underline X_{n}^{j,j} > \mu n-4C_1k^{2/3}n^{1/3}\right\}\right) \geq 1-Ce^{-ck^2}.$$
\end{lemma}

\begin{proof}
Recall from \eqref{e.averaging inequality} that, for every $j$, $\underline X_{n}^{j,j} \geq X_{n}^{j}-X_{n}^{j-1}$.
	Indeed, any $j$-geodesic watermelon is comprised of a lightest  curve and $j-1$ other curves that together weigh at most as much as the $(j-1)$-geodesic watermelon. See Figure~\ref{f.Xnk increment}.
We see then that, for any $k$,
	$$\P\left(\bigcup_{j=\lfloor k/2\rfloor}^{k} \left\{\underline X_{n}^{j,j} > \mu n-4C_1k^{2/3}n^{1/3}\right\}\right) \geq \P\left(\bigcup_{j=\lfloor k/2\rfloor}^{k} \left\{X_{n}^{j}-X_{n}^{j-1} > \mu n-4C_1k^{2/3}n^{1/3}\right\}\right).$$
\begin{figure}[h]
   \centering
        \includegraphics[width=.3\textwidth]{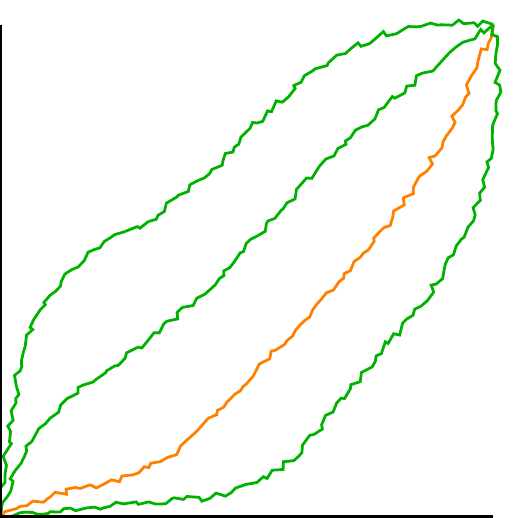}
 \caption{The basic argument for the inequality $X_n^{k,k} \geq X_n^k - X_n^{k-1}$. Here $k=4$. The four curves shown are the curves of the $k$-melon, and the orange curve is the curve of lowest weight among the four, and so has weight $X_n^{k,k}$. The remaining green curves form a collection of three disjoint curves, and hence have cumulative weight  bounded above by $X_n^{k-1}$. Thus $X_n^{k,k} + X_n^{k-1}\geq X_n^k$, and rearranging gives the inequality.}
 \label{f.Xnk increment}
\end{figure}
	We now observe that
	\begin{align*}
	\P\Bigg(\bigcap_{j=\lfloor k/2\rfloor}^{k} \Big\{X_{n}^{j}-X_{n}^{j-1} \leq \mu n-4C_1&k^{2/3}n^{1/3}\Big\}\Bigg)\\
	&\leq \P\left(X_{n}^{k}-X_{n}^{\lfloor k/2\rfloor-1} \leq \mu n(k+1 -\lfloor k/2\rfloor)-2C_1k^{5/3}n^{1/3}\right).
	\end{align*}
	From Lemma \ref{l.first upper bound on melon weight} with $t=2^{5/3}C_1$, with the valid $k$ range depending on whether Assumption~\ref{a.one point assumption convex} is in force in that statement, we have that
	$\P\left(X_{n}^{\lfloor k/2\rfloor-1} > \mu n(\lfloor k/2\rfloor-1) + C_1k^{5/3}n^{1/3}\right) \leq e^{-ck^2}.$
	Thus,
	\begin{align*}
	\P\Big(X_{n}^{k}-X_{n}^{\lfloor k/2\rfloor-1} \leq \mu n(k+1 -\lfloor k/2\rfloor)&-2C_1k^{5/3}n^{1/3}\Big)
	\leq \P\left(X_{n}^{k} \leq \mu nk- C_1k^{5/3}n^{1/3}\right)+e^{-ck^2}.
	\end{align*}
	The weight lower bound of Theorem~\ref{t.notwidenotthingeneral}(\ref{weight1'}) tells us that the first term is bounded by $e^{-ck^{2}}$. By the three preceding displays, we find that
	$$\P\left(\bigcup_{j=\lfloor k/2\rfloor}^{k} \left\{\underline X_{n}^{j,j} > \mu n-4C_1k^{2/3}n^{1/3}\right\}\right) \geq 1-2e^{-ck^{2}},$$
	so we are done.
\end{proof}

We are now ready to complete the proof of Theorem~\ref{t.notwidenotthingeneral}(\ref{tf2'}).

\begin{proof}[Proof of Theorem~\ref{t.notwidenotthingeneral}(\ref{tf2'})]
Under Assumption~\ref{a.one point assumption} (resp. \ref{a.one point assumption convex}) and $c_1$ as in Lemma~\ref{p.smallest curve weight lower bound}, fix $k$ sufficiently large and $n$ such that $k\leq c_1n^{1/2}$ (resp. $k\leq c_1n$). Let $C_1$ be as in Theorem~\ref{t.notwidenotthingeneral}(\ref{weight1'}). Let $\delta=\delta(2^{2/3}\cdot 4C_1)$ be as in Theorem \ref{t:disjoint} and let $\delta '= 2^{-1/3}\delta$. Let $A_{k}$ denote the event that there exist $\lfloor k/2\rfloor$ disjoint paths contained in $U_{n,\delta' k^{1/3}n^{2/3}}$, each of which has weight at least \smash{$\mu n- 4C_1k^{2/3}n^{1/3}$}. Further, let $B_{k}$ denote the event from Lemma \ref{p.smallest curve weight lower bound}: 
$$B_{k}:=\left\{\bigcup_{j =\lfloor \frac{k}{2} \rfloor}^{k} \left\{\underline X_{n}^{j,j} > \mu n-4C_1k^{2/3}n^{1/3}\right\}\right\}.$$
Clearly, by Theorem \ref{t:disjoint} and Lemma \ref{p.smallest curve weight lower bound}, we have $\P(A_{k}^c \cap B_{k})\geq 1-e^{-ck^2}$ for some $c>0$. 

Observe next that on $A_{k}^{c}\cap B_k$, there exists $j\in \{\lfloor\frac{k}{2}\rfloor, \ldots , k\}$ such that all the curves of all $j$-geodesic watermelons have weight at least $\mu n- 4C_1k^{2/3}n^{1/3}$, and hence some of them must exit $U_{n,\delta' k^{1/3}n^{2/3}}$. By the interlacing result Proposition \ref{p.melon interlacing}, the same is true for all $k$-geodesic watermelons. This completes the proof of Theorem~\ref{t.notwidenotthingeneral}(\ref{tf2'}) with $\delta$ there  replaced by $\delta'$.
\end{proof}

\section{Not too wide:  Bounding above the transversal fluctuation}
\label{s:upperexp}
Our objective in this section is to prove Theorem~\ref{t.notwidenotthingeneral}(\ref{tf1'}). We derive this result using Theorem~\ref{t.tf}, whose proof appears  at the end of the section. The basic idea is same as in Section \ref{s:lowerexp}. To show that the $k$-geodesic watermelon has transversal fluctuation of order $k^{1/3}n^{2/3}$ with large probability, we shall rely on the interlacing property and show that, with large probability, there exists $j\in \llbracket k,  2k \rrbracket$ such that the $j$-geodesic watermelon has transversal fluctuations at most of order $k^{1/3}n^{2/3}$. Admitting Theorem~\ref{t.tf} for now, we may prove Theorem \ref{t.notwidenotthingeneral}(\ref{tf1'}).

\begin{proof}[Proof of Theorem~\ref{t.notwidenotthingeneral}(\ref{tf1'})]
Let $C_1$ be as in the Lemma~\ref{p.smallest curve weight lower bound} and let 
$$B_k' = \bigcup_{j=k}^{2k}\left\{\underline X_n^{j,j} > \mu n - 7C_1 k^{2/3}n^{1/3}\right\};$$
note that $4\times 2^{2/3} < 7$, so that Lemma~\ref{p.smallest curve weight lower bound} with $2k$ in place of $k$ implies that $\P\left(B_k'\right) \geq 1-e^{-ck^2}$.

Let $A'_{k}=A'_{k}(C'')$ denote the event that there exists a path from $(1,1)$ to $(n,n)$ that exits $U_{n,C''k^{1/3}n^{2/3}}$ (recall the notation from \eqref{parallelogram1}) and has weight at least $\mu n - 7C_1k^{2/3}n^{1/3}$. Now choose $C''>0$ (possible by invoking Theorem \ref{t.tf} with $s$ a multiple of $k^{1/3}$) such that 
$$\P(A'_{k})\leq e^{-c''k}$$ 
for some $c''>0$, for all $k$ sufficiently large and all $n$ sufficiently large, depending on $k$. Clearly, it now suffices to show that, on $B'_k \cap (A'_{k})^c$, no $k$-geodesic watermelon exits \smash{$U_{n,C''k^{1/3}n^{2/3}}$}. By definition of $B'_{k}$, there exists $j\in \llbracket k,2k\rrbracket$ such that all paths of all $j$-geodesic watermelons have weight at least $\mu n-7C_1k^{2/3}n^{1/3}$, and $(A'_{k})^c$ ensures that all these paths are contained in $U_{n,C''k^{1/3}n^{2/3}}$. The proof is completed by interlacing,  invoking Proposition \ref{p.melon interlacing}, with the parameter $M$ in Theorem~\ref{t.notwidenotthingeneral}(\ref{tf1'}) being $C''$.
\end{proof}

Now we turn to the proof of Theorem \ref{t.tf}.  We will follow the argument that yields Theorem 11.1 of \cite{slow-bond}, but merely suppose  Assumptions~\ref{a.limit shape assumption} and~\ref{a.one point assumption}.
We first prove a companion result concerning paths that have large transversal fluctuation at the midpoint. The idea is that the weight of such a path is less than the sum of the weights of two point-to-line paths whose endpoints lie outside a central interval. We will make use of the upper tail estimate for such point-to-line weights, Proposition~\ref{p.p2l general upper tail}.

For an upright path $\Gamma$ from $(1,1)$ to $(r, r)$, and $x\in \intint{r}$, let $\Gamma(x)$ denote the unique integer such that $(x-\Gamma(x), x+ \Gamma (x))$ lies on $\Gamma$. For $t\geq1$, let
\begin{align*}
\lineleft (r,t) &= \{(x,y): x+y=2, |y-x|<tr^{2/3}\}\qquad\text{and}\\
\lineright (r,t) &= \{(x,y): x+y= 2r, |y-x|<tr^{2/3}\}
\end{align*}
be line segments of length $2tr^{2/3}$ of slope $-1$ through $(1,1)$ and $(r, r)$ respectively. For the sake of brevity, and in the hope that there is little scope for confusion, we shall omit the two arguments from now on. 

\begin{prop} \label{p.midpoint tf}
Let $\Xmid$ be the maximum weight of all paths $\Gamma$
from $\lineleft$ to $\lineright$
such that $|\Gamma(r/2)| > (s+t)r^{2/3}$.
Then there exist constants $r_0, s_0$, $c>0$ and $\tilde c_2 >0$
(all independent of $t$) 
such that for $s>s_0$, $r> r_0$ and $0<t\leq s^2$,
$$\P\left(\Xmid- \mu r > -\tilde c_2s^2r^{1/3}\right) < e^{-cs^3}.$$
\end{prop}

\begin{proof}
	Let $X_1$ be the maximum weight of all paths $\Gamma$ from $\lineleft$ to the line $x+y=r$ whose midpoint is outside the interval joining $(r/2-(s+t)r^{2/3}, r/2+(s+t)r^{2/3})$ and $(r/2+(s+t)r^{2/3}, r/2-(s+t)r^{2/3})$. Let $X_2$ be the same for paths starting on the line $x+y=r$ outside the mentioned interval and ending on $\lineright$.	
	Then
	$$\left\{\Xmid > \mu r - \tilde c_2 s^2 r^{1/3}\right\}\subseteq \bigcup_{i=1}^2 \left\{X_i > \mu \frac{r}{2} - 0.5\tilde c_2s^2r^{1/3}\right\}.$$
	
	From here, we apply Proposition \ref{p.p2l general upper tail} with $\theta = 0$, which yields
	\begin{align*}
	\P\Big(X_i > \mu\frac{r}{2}  -0.5\tilde c_2s^2r^{1/3}\Big) \leq  e^{-cs^3} 
	\end{align*}
 for $i \in \{1,2\}$;		a union bound now completes the proof of Proposition~\ref{p.midpoint tf}. 
	\end{proof}

\begin{figure}[h]
   \centering
        \includegraphics[width=.75\textwidth]{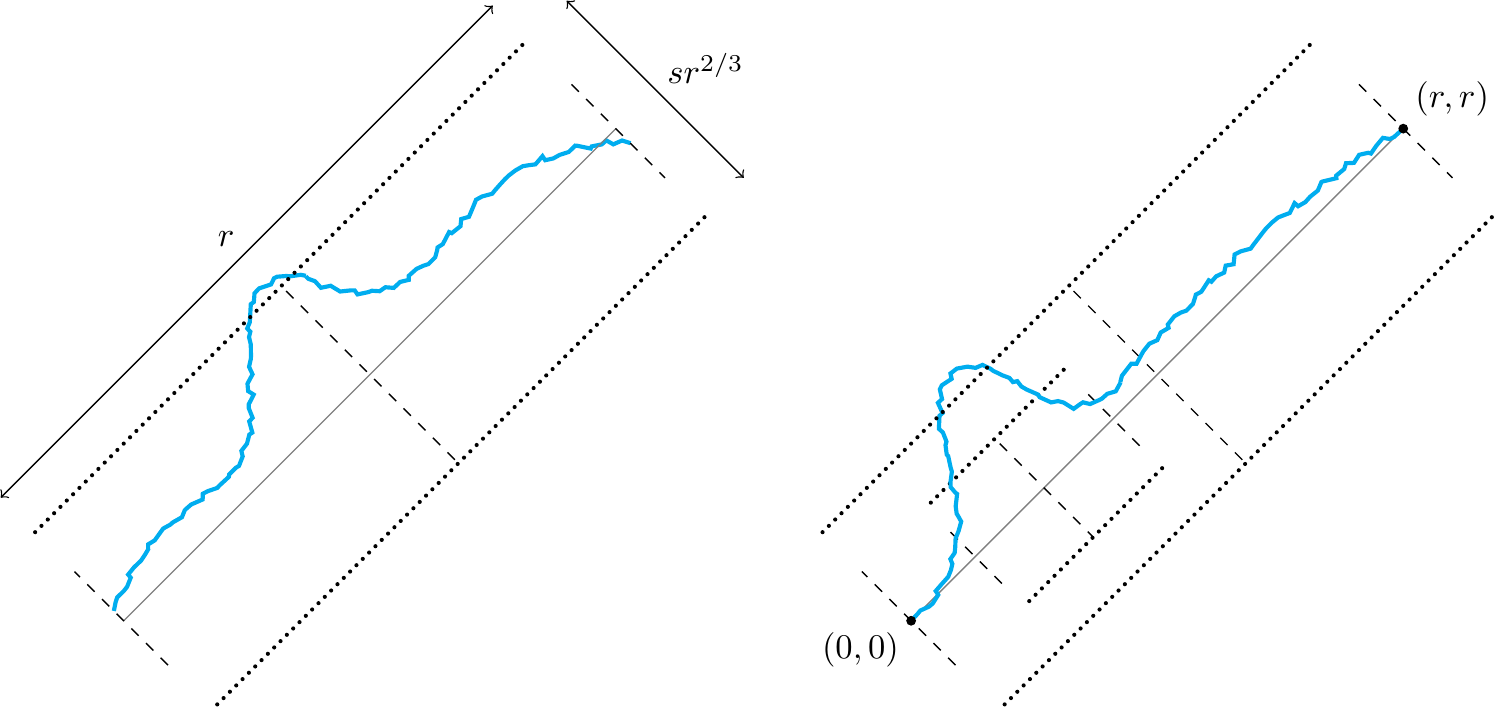}
 \caption{On the left is depicted the event whose probability is bounded in Proposition~\ref{p.midpoint tf}, namely that a path from $\lineleft$ to $\lineright$ has high transversal fluctuation at the midpoint, but does not suffer an appropriately high loss in weight. On the right, we see how this allows us to infer a similar bound for the event that the curve has high transversal fluctuation (not necessarily at the midpoint) and an unusually high weight. We argue that for any such curve there exists a pair of intervals on a dyadic scale such that the path passes through both, but avoids the longer interval at their midpoint; this is the collection of smaller three intervals shown in the right panel, which can be thought of as a smaller scale version of the left panel. The loss suffered by this portion of the curve cannot be compensated by the remaining portions.}
 \label{f.tf}
\end{figure}

We now use a multi-scale argument to extend the result at the midpoint to transversal fluctuations at any point. Roughly, we will dyadically place points on the diagonal $x=y$ between $(1,1)$ and $(r,r)$. Suppose that some path has transversal fluctuation greater than $(s+t)r^{2/3}$ at some point while satisfying the weight lower bound. Then a consecutive pair of points in the dyadic division exists  such that the intervening  path is not too light and has a midpoint whose  transversal fluctuation has order $(s+t)r^{2/3}$.  
The probability of this circumstance is bounded by Proposition~\ref{p.midpoint tf}. Bounds on path weights for each side of this dyadic interval will finish the proof.

We remind the reader from the statement of Theorem~\ref{t.tf} that $X^{r,s,t}$ refers to the maximum weight over all paths $\Gamma^{r,s,t}$ from $\lineleft$ to $\lineright$ with transversal fluctuation greater than $(s+t)r^{2/3}$ at some point. We now fix $\Gamma^{r,s,t}$ to be the leftmost such path (which is uniquely defined by the weight-maximization property of the path) whose weight is $X^{r,s,t}$.

\begin{proof}[Proof of Theorem \ref{t.tf}]
	We may assume that $s\leq r^{1/3}$, as otherwise the theorem is trivial.
	We may also simplify by assuming that $\Gamma^{r,s,t}(x) > (s+t)r^{2/3}$ for some $x\in[1,r]$ (the case where $\Gamma^{r,s,t}(x) < -(s+t)r^{2/3}$ is symmetric).
	Let $A$ be the event $\{X^{r,s,t} > \mu r -c_2s^2r^{1/3}\}$ for $c_2$ to be specified later.
	For $j\geq 1$, let $S_j$ be the dyadic points, i.e., 
	$$S_j = \left\{\ell 2^{-j}r: \ell  \in \llbracket 0, 2^j \rrbracket \right\}.$$
	For $j\geq 1$, let $T_j$ be the event defined by
	$$T_j = \left\{\Gamma^{r,s,t}(x) < (s_j+t)r^{2/3} \quad \forall x\in S_j\right\},$$
	where 
	\begin{equation}\label{e.s_j value}
	s_j = \frac{s}{M}\prod_{i=1}^{j-1}(1+2^{-i/3})
		\quad\text{and}\quad
	M = 2\cdot\prod_{i=1}^{\infty}(1+2^{-i/3}) < \infty,
	\end{equation}
	so that $s_j < \frac12s$ for all $j$. 
	We fix $j_0$ so that
	the separation between points in $S_{j_0}$ is $\frac{s}{10}r^{2/3}$, i.e., $j_0$ is such that 
	$2^{-j_0}r = \frac{s}{10}r^{2/3}$.
	The next lemma asserts that the defined dyadic breakup is fine enough to capture any path which has a high transversal fluctuation at some point.

	\begin{lemma}\label{l.empty intersection}
	We have that
	$\bigcap_{j=1}^{j_0} T_j = \emptyset.$
	\end{lemma}

	\begin{proof}
		By definition of $\Gamma^{r,s,t}$,
		there is a $x\in[1,r]$ such that $\Gamma^{r,s,t}(x) > (s+t)r^{2/3}$.
		Let $\bar x$ be the smallest point of $S_{j_0}$ bigger than $x$.
		Observe that since $\Gamma^{r,s,t}$ is an upright path, its $y$-coordinate $x+\Gamma^{r,s,t}(x)$ is increasing in $x$, so that
		$x+\Gamma^{r,s,t}(x) \leq \bar x+\Gamma^{r,s,t}(\bar x)$.
		From the definition of $S_{j_0}$, we also have that $\bar x \leq x+\frac{s}{10}r^{2/3}$. This implies
		$$\Gamma^{r,s,t}(x) \leq (\bar x-x) + \Gamma^{r,s,t}(\bar x) \leq \left(\frac{s}{10}+s_{j_0}+t\right)r^{2/3},$$
		the last inequality holding on $T_{j_0}$.
		However $s_{j_0} + \frac{s}{10}+t < \frac{s}{2} + \frac{s}{10}+t < s+t$, a contradiction.
	\end{proof}

	Let $T_0$ be the whole probability space. From Lemma \ref{l.empty intersection} we have
	$$\P(A) = \P\Bigg(\bigcup_{j=0}^{j_0} T_j^c \cap A\Bigg) \leq \sum_{j=1}^{j_0} \P\left(T_{j-1}\cap T_j^c \cap A\right).$$
	We set $r>r_0$ and $s>s_0$, where $r_0$ and $s_0$ are obtained from Proposition \ref{p.midpoint tf}. Then Proposition~\ref{p.midpoint tf} establishes that $\P(T_1^c\cap A) \leq e^{-cs^3}$, and so the proof of Theorem~\ref{t.tf} is completed by the next lemma.
\end{proof}

\begin{lemma}\label{l.needslabel} Let $r$ and $s$ be as before. There exists  $c>0$ such that, for $j\geq 2$,
$$\P\left(T_{j-1}\cap T_j^c \cap A\right) \leq 2^{-j}e^{-cs^3}.$$
\end{lemma}

\begin{proof}
	We split $T_j^c = \bigcup_{h=0}^{2^j} T_{j,h}^c$
	based on at which dyadic point $h2^{-j}r$ of $S_j$
	we have $\Gamma^{r,s,t}(h2^{-j}r) > (s_j+t)r^{2/3}$. 

	Let $\L_1$ be the line segment of length $2(s_{j-1}+t)r^{2/3}$ of slope $-1$ centred at $((h-1)2^{-j}r, (h-1)2^{-j}r)$, and let $\L_2$ be the same centred at $((h+1)2^{-j}r, (h+1)2^{-j}r)$.

	Now, on $T_{j-1}\cap T_{j,h}^c \cap A$, there exists a path $\Gamma_1$ from $\lineleft$ to $\L_1$, a path $\Gamma_2$ from $\L_1$ to $\L_2$, and a path $\Gamma_3$ from $\L_2$ to $\lineright$ with the following properties: (i) $\Gamma_2(h2^{-j}r) > (s_j+t)r^{2/3}$ and (ii)~$\weight(\Gamma_1) + \weight(\Gamma_2) + \weight(\Gamma_3)> \mu r-c_2s^2r^{1/3}.$

	We will first show that, due to condition (i), $\Gamma_2$ must suffer a large weight loss, via Proposition~\ref{p.midpoint tf}. Then we will show using Proposition~\ref{l.sup tail}  that $\Gamma_1$ and $\Gamma_3$ are unlikely to be able to make up this loss sufficiently well for (ii) to occur. The basic reason for the large weight loss of $\Gamma_2$ is that the loss from the transversal fluctuation is much amplified by its being defined on the scale $2^{-(j-1)}r$ instead of $r$.

	We set the parameters for the application of Proposition~\ref{l.sup tail}. In order to avoid confusion, the parameter values will be distinguished by tildes. So set $\tilde r = 2^{-(j-1)}r$. As the endpoint of $\Gamma_2$ can lie anywhere on an interval of length $2(s_{j-1}+t)r^{2/3}$, set $\tilde t$ such that $\tilde t(\tilde r)^{2/3} = (s_{j-1}+t) r^{2/3}$; i.e., $\tilde t = (s_{j-1}+t)2^{2(j-1)/3}$. The path $\Gamma_2$ at its midpoint it must be at least $(s_j+t)r^{2/3}$ away from the diagonal, so the minimum transversal fluctuation the path undergoes is $((s_j+t)-(s_{j-1}+t))r^{2/3} = s_{j-1}2^{-(j-1)/3}r^{2/3}$ (using \eqref{e.s_j value}). Accordingly we set $\tilde s = s_{j-1}2^{(j-1)/3}$, so that $\tilde s(\tilde r)^{2/3} = (s_j-s_{j-1})r^{2/3}$. 

	To apply Proposition~\ref{p.midpoint tf} with these parameters, we require \smash{$\tilde t\leq \big(\tilde s\big)^2$}. Since the hypotheses of Theorem~\ref{t.tf} include that $t\leq s$, and since $s_j\in[s/M, s/2]$ for all $j$, the requirement is implied if 
	$$\frac{3s}{2}\cdot2^{2(j-1)/3}\leq \frac{s^2}{M^2}\cdot2^{2(j-1)/3},$$
	which clearly holds for all large enough $s$. So, making use of distributional translational invariance of the environment and applying Proposition~\ref{p.midpoint tf}, we obtain
	\begin{align*}
	\P\left(\weight(\Gamma_2) > \mu \tilde r - \tilde c_2(\tilde s)^2(\tilde r)^{1/3}\right) \leq \exp\left(-c(\tilde s)^3\right) \, ; %
	\end{align*}
	i.e.,
	\begin{gather*}
	\P\left(\weight(\Gamma_2) > \mu \tilde r - \tilde c_2s_{j-1}^2 2^{(j-1)/3} r^{1/3}\right) \leq \exp\left(-c2^{(j-1)}s^3\right) \, .
	\end{gather*}
	This yields
	\begin{align*}
	\P\left(\sum_{i=1}^3 \weight(\Gamma_i)> \mu r-c_2s^2r^{2/3}\right) &\leq \P\left(\weight(\Gamma_2) > \mu \tilde r - \tilde c_2 2^{(j-1)/3}s_{j-1}^2r^{1/3}\right)\\
	&\quad+ \P\left(\weight(\Gamma_1) + \weight(\Gamma_3) > \mu (r-\tilde r) + (\tilde c_2 2^{(j-1)/3}s_{j-1}^2 - c_2s^2)r^{1/3}\right) \, .
	\end{align*}
	We must address the second term. Set $c_2 = \tilde c_2/2M^2$, so that $\tilde c_2 s_{j-1}^2/2 > c_2s^2$.
	After dividing $\L_1$ and $\L_2$ into segments of length $r^{2/3}$, applying Proposition \ref{l.sup tail}, and taking a union bound over the segments, we get that the second probability is bounded by $\exp\left(-c2^{(j-1)/2} s^3\right)$.

	For large enough $s$, these two quantities are bounded by $4^{-j}e^{-cs^3}$, and taking a union bound over the $2^j$ values of $h$, we obtain Lemma~\ref{l.needslabel}.
\end{proof}

Theorem~\ref{t.tf} readily implies that the transversal fluctuation of the geodesic is of order $n^{2/3}$, with the optimal tail exponent of three.

\begin{corollary}
There exist constants $s_0$ and $n_0$ such that for $s>s_0$ and $n>n_0$,
$$\P\left(\tf(\Gamma_{n}) > sn^{2/3}\right) \leq e^{-cs^3}.$$
\end{corollary}

\begin{proof}
We have
\begin{align*}
\P\left(\tf(\Gamma_{n}) > sn^{2/3}\right) &\leq \P\left(\tf(\Gamma_{n}) > sn^{2/3}, X_{n} > \mu n - c_2 s^2 n^{1/3}\right) + \P\left(X_{n} \leq \mu n - c_2 s^2 n^{1/3}\right)\\
&\leq e^{-cs^3} + e^{-cs^3},
\end{align*}
the last inequality from Theorem \ref{t.tf} for the first term and Assumption~\ref{a.one point assumption} for the second.
\end{proof}

We end this section by recording a version of Theorem~\ref{t.tf} for point-to-line paths which will be used in Section~\ref{s:p2l}.

\begin{theorem}\label{t.point to line weight loss}
Let $\Gamma^{r,s}_\mathrm{line}$ be the maximum weight path among all paths $\Gamma$ from $(1,1)$ to the line $x+y=2r$ with $\tf(\Gamma) > sr^{2/3}$, and let $X^{r,s}_{\mathrm{line}}$ be its weight. Under Assumptions~\ref{a.limit shape assumption} and \ref{a.one point assumption}, there exist absolute constants $r_0$, $s_0$, $c>0$ and $c_3>0$ such that, for $s_0<s<r^{1/3}$ and $r>r_0$,
$$\P\left(X^{r,s}_{\mathrm{line}} -\mu r > -c_3 s^2 r^{1/3}\right) \leq e^{-cs^3}.$$
\end{theorem}

\newcommand{\gammaline}{\Gamma^{r,s}_{\mathrm{line}}}

\begin{proof}
	Let $x$ be such that $\gammaline$ has endpoint $(r-x, r+x)$. Let $A$ denote the event that $|x|< sr^{2/3}/2$. %
	On the event $A$, we have the claimed bound by applying Theorem \ref{t.tf} with $t=s/2$. 
	On~$A^c$, the claim arises by applying Proposition \ref{p.p2l general upper tail} with $t=0$, $\theta=0$, and $s/2$ in place of $s$ in its statement.
\end{proof}

\section{Not too heavy: Bounding above the watermelon weight}
\label{s:upper}
Our aim in this section is to complete the proof of the weight upper bound of Theorem~\ref{t.notwidenotthingeneral}(\ref{weight1'}) and to provide the proof of Proposition~\ref{p:kthweight}. Recall that for the former our rough aim is to prove that there exist $C_2>0$ and $c>0$ such that, for appropriate ranges of $k$ and $n$,
\begin{equation}\label{e.weight upper bound}
\P\left(X_n^k > \mu nk - C_2k^{5/3}n^{1/3}\right) \leq e^{-ck^2}.
\end{equation}

We start by proving that with high probability at least $k/4$ curves of the $k$-geodesic watermelon must exit a strip $U = U_{n,\frac12\delta k^{1/3}n^{2/3}}$ around the diagonal of width $\frac12 \delta k^{1/3}n^{2/3}$, with $\delta$ as in Theorem~\ref{t:disjoint}. We will resort to proving an averaged version, as in Lemma~\ref{p.smallest curve weight lower bound}; i.e., we will show that there exists  $j\in \llbracket{ \lfloor k/2\rfloor, k\rrbracket}$ such that $k/4$ curves of the $j$-melon exit $U$. This will suffice by the interlacing guaranteed by Proposition~\ref{p.melon interlacing}.

For $j\in\intint{n}$ and $\delta>0$, define $\underline E_n^j(\delta)$ to be the minimum of the number of curves of a $j$-geodesic watermelon $\Gamma_n^j$ which \emph{exit} $U$, minimized over all $j$-geodesic watermelons $\Gamma_n^j$.

\begin{lemma}\label{l.no interior packing}

Let $C_1>0$ be as given in Theorem~\ref{t.notwidenotthingeneral}(\ref{weight1'}), and let $\delta = \delta(4C_1)$ be as in Theorem~\ref{t:disjoint}. There exists $c_1>0$ such that, under Assumption~\ref{a.one point assumption} (resp. \ref{a.one point assumption convex}), for $k \leq c_1 n^{1/2}$ (resp. $k\leq c_1n$),
$$\P\left(\bigcup_{j=\lfloor k/2\rfloor}^{k} \left\{\underline E_n^j(\delta) > \frac{k}{4}\right\}\right) \geq 1-e^{-ck^2}.$$
\end{lemma}

\begin{proof}
	Recall from Section~\ref{s:lowerexp} the definition of $\underline X_n^{j,j}$ as the minimum weight of the lightest curve of a $j$-geodesic watermelon $\mc W$ over all such $\mc W$. We will  bound below
	$$\P\left(\bigcup_{j=\lfloor k/2\rfloor}^k\left\{\underline X_{n}^{j,j}>\mu n- 4C_1k^{2/3}n^{1/3}\right\}\cap\left\{\underline E_n^j(\delta) > \frac{k}{4}\right\}\right),$$
	which clearly suffices.
	We have
	\begin{align*}
	&\bigcap_{j=\lfloor k/2\rfloor}^{k} \left(\left\{\underline X_{n}^{j,j}\leq \mu n- 4C_1k^{2/3}n^{1/3}\right\} \cup \left\{\underline E_n^j(\delta)) \leq \frac{k}{4}\right\}\right)\\
	&\subseteq \left(\bigcap_{j=\lfloor k/2\rfloor}^k \left\{\underline X_{n}^{j,j}\leq \mu n- 4C_1k^{2/3}n^{1/3}\right\}\right)\cup\left(\bigcup_{j=\lfloor k/2\rfloor}^k \left\{\underline E_n^j(\delta) \leq \frac{k}{4}, \underline X_n^{j,j}> \mu n - 4C_1 k^{2/3}n^{1/3}\right\}\right).
	\end{align*}
	Focus on the right-hand side of this display. Lemma \ref{p.smallest curve weight lower bound} says that the probability of the first term in parentheses is bounded by $e^{-ck^2}$. For the second term, note that, on each inner event,  there exist $3k/4$ curves contained in \smash{$U_{n,\frac{1}{2}\delta k^{1/3}n^{2/3}}$} whose weight is at least \smash{$\mu n - 4C_1k^{2/3}n^{1/3}$}. Theorem~\ref{t:disjoint} then implies that the probability of the inner event is bounded by $e^{-ck^2}$ for all $j$; and a union bound, accompanied by a reduction in the value of $c$, completes the proof.
\end{proof}

\begin{proof}[Proof of the weight upper bound in Theorem~\ref{t.notwidenotthingeneral}(\ref{weight1'})]
	Let $B_1$ denote the event that  for no $j \in \llbracket \lfloor k/2\rfloor, k\rrbracket$ are there at least $k/4$ curves in every $j$-geodesic watermelon that exit the strip of width $\frac{1}{2}\delta k^{1/3}n^{2/3}$. By Lemma~\ref{l.no interior packing},
	\begin{equation}\label{e.B_1 bound}
	\P(B_1) \leq e^{-ck^2}.
	\end{equation}
	By interlacing Proposition~\ref{p.specified melon interlacing}, we see that when  $B_1^c$ occurs, at least $k/4$ of the curves of every $k$-geodesic watermelon must exit the strip of width $\frac12\delta k^{1/3}$: see Figure~\ref{f.no interior packing}.

	\begin{figure}[h]
   \centering
        \includegraphics[width=.35\textwidth]{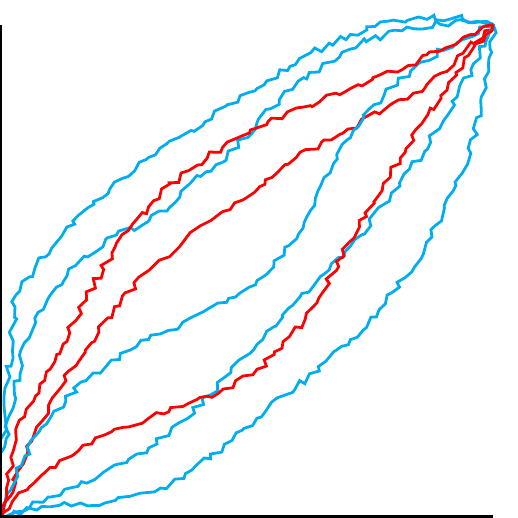}
 \caption{An illustration of how curves of $\Gamma_n^j$ exiting a strip force the same for curves of $\Gamma_n^k$ for $j=3$ (red) and $k=5$ (blue), where the two ensembles are elements of an interlacing family of gedoesic watermelons. The underlying observation is that interlacing allows one to match each curve of $\Gamma_n^j$ with one curve of $\Gamma_n^k$ which has larger transversal fluctuation. Perhaps contrary to intuition, interlacing does \emph{not} imply that a single curve of \smash{$\Gamma_n^{j}$} exiting a strip forces at least $k-j$ curves of $\Gamma_n^k$ to exit the same strip. This can also be seen in the figure, as there is only a single curve of $\Gamma_n^k$ which has a larger transversal fluctuation on the left than $\Gamma_n^j$ even though $k-j = 2$. }
 \label{f.no interior packing}
\end{figure}

Let $B_2$ denote the event that there exists an upright path $\Gamma$ from $(1,1)$ to $(n,n)$
for which  $\ell(\Gamma)>\mu n-c'k^{2/3}n^{1/3}$ and $\tf(\Gamma)>\frac12\delta k^{1/3}n^{2/3}$.
	By Theorem \ref{t.tf}, there exists $c' > 0$ such that
		\begin{equation}\label{e.B_2 bound}
	\P(B_2)\leq e^{-ck}.
	\end{equation}
Set $A=\left\{X_{n}^k > \mu nk - \frac{1}{16}c'k^{5/3}n^{1/3}\right\}$ and 
	$$B_3 = \left\{X_{n}^{\lfloor 7k/8\rfloor} > \mu n\cdot\lfloor 7k/8\rfloor + \frac{1}{16}c' k^{5/3}n^{1/3}\right\},$$
	so that, by Lemma \ref{l.first upper bound on melon weight} under Assumption~\ref{a.one point assumption} (resp. \ref{a.one point assumption convex}), for $k\leq c_1n^{1/2}$ (resp. $k\leq c_1n$),
	\begin{equation}\label{e.B_3 bound}
	\P(B_3) \leq e^{-ck^2}.
	\end{equation} 
	We now fix some $k$-geodesic watermelon $\Gamma_n^k$, and let $X_n^{k, j}$ be the weight of the $j$\textsuperscript{th} heaviest curve of $\Gamma_n^k$ for each $j\in\intint{k}$. The weight of the $\lfloor 7k/8\rfloor$ heaviest curves of $\Gamma_n^k$ must be at most the weight of the $\lfloor 7k/8\rfloor$-geodesic watermelon. Thus, when $A\cap B_3^c$ occurs,
	\begin{gather*}
	X_{n}^{\lfloor 7k/8\rfloor} + \left(X_{n}^{k,\lfloor 7k/8\rfloor+1}+\ldots+X_{n}^{k,k}\right) \geq X_{n}^{k} > \mu nk -\frac{1}{16}c'k^{5/3}n^{1/3}\\
	\implies X_{n}^{k,\lfloor 7k/8\rfloor+1}+\ldots+X_{n}^{k,k} > \frac18 \mu nk - \frac{1}{8}c' k^{5/3}n^{1/3}.
	\end{gather*}
	By the ordering of $\{X_{n}^{k,i}:i\in\intint{k}\}$, this then implies that $X_{n}^{k,\lfloor 7k/8\rfloor+1}> \mu n - c'k^{2/3}n^{1/3}$. Again by the ordering, this bound applies to $X_{n}^{k,1},\ldots,X_{n}^{k,\lfloor 7k/8\rfloor}$ as well. This means we have $\lfloor7k/8\rfloor$ disjoint curves $\Gamma$, each with $\ell(\Gamma) > \mu n - c'k^{2/3}n^{1/3}$.

	Thus, on $A\cap B_1^c\cap B_3^c$ and by the pigeonhole principle, we must have at least $\lfloor k/8\rfloor$ disjoint curves~$\Gamma$, each satisfying $\tf(\Gamma)>\frac12\delta k^{1/3}n^{2/3}$ and $\weight(\Gamma)>\mu n - c'k^{2/3}n^{1/3}$.
	By the BK inequality and~\eqref{e.B_2 bound}, the probability of this occurrence is seen to be at most
	$\exp\left(-ck\cdot k\right) = \exp\left(-ck^{2}\right)$.
	Noting the bounds \eqref{e.B_1 bound} and \eqref{e.B_3 bound} on $\P(B_1)$ and $\P(B_3)$, we complete the proof of the weight upper bound of Theorem~\ref{t.notwidenotthingeneral}(\ref{weight1'}), by taking $C_2$ in its statement to be $c'$.
\end{proof}

\section{Not too light: Bounding below the watermelon weight}
\label{s:construction}
In this section, we construct collections of disjoint paths in the square $\intint{n}^2$ that achieve a certain weight with high probability and so prove Theorem~\ref{t.flexible construction}. This construction is one of the principal new tools developed in this paper.

Recall that to prove Theorem~\ref{t.flexible construction}, for some $c>0$, $c_1>0$, $k\leq c_1n$, and $m\leq k$, we must construct $m$ disjoint paths $\gamma_1, \ldots, \gamma_m$ such that $\sum_{i=1}^m \ell(\gamma_i) \geq \mu n m - C_1 mk^{2/3}n^{1/3}$ and $\max_i\tf(\gamma_i) \leq 2mk^{-2/3}n^{1/3}$, with probability at least $1-\exp(-cmk)$.

\subsection{The construction in outline}

The construction leading to Theorem~\ref{t.flexible construction} may be explained in light of Theorem~\ref{t.notwidenothin}(\ref{tf}) on 
 the width of the $k$-geodesic watermelon having order $k^{1/3}n^{2/3}$, even if the latter assertion is a consequence of the theorem rather than a means for deriving it.  
Indeed, that $k$ watermelon curves coexist in a strip of width $k^{1/3}n^{2/3}$ suggests that, at least around the mid-height $n/2$, adjacent curves will separated on the order of $k^{-2/3}n^{2/3}$. We will demand this separation for the $m$ curves in our construction. The curves will begin near $(1,1)$ and end near $(n,n)$ at unit-order distance, so we must guide them apart to become separated during their mid-lives. 

We will index the life of paths in the square $\intint{n}^2$ according to distance along the diagonal interval, indexed so that $[a,b]$ refers to the region between the lines $x+y = 2a$ and $x+y=2b$. The diagonal interval $[1, n]$ that indexes the whole life of paths in the construction will be divided into five consecutive intervals called {\em phases} that carry the names take-off, climb, cruise, descent and landing. By the start of the middle, cruise, phase, the sought separation has been obtained, and it will be maintained there. This separation is gained during take-off and climb, and it is undone in a symmetric way during descent and landing.

Take-off is a short but intense phase that takes the curves at unit-order separation on the tarmac to a consecutive separation of order $k^{-1/3}n^{1/3}$
in a duration (or height) of order $k^{2/3}n^{1/3}$. Climb is a longer and gentler phase, of duration roughly $n/3$, in which separation expands dyadically until it reaches the scale $k^{-2/3}n^{2/3}$. Cruise is a stable phase of rough duration $n/3$.

The shortfall in weight of the $m$ paths in Theorem~\ref{t.flexible construction} relative to the linear term $\mu n m$ has order $mk^{2/3}n^{1/3}$.
The weight shortfall in each phase is the difference in total weight contributed by the curve fragments in the phase and the linear term given by the product of $\mu m$ and the duration of the phase.
The weight shortfall will be shown to have order $mk^{2/3}n^{1/3}$ for each of the five phases.%

Take-off is a phase where gaining separation is the only aim. No attempt is made to ensure that the constructed curves have weight, and the trivial lower bound of zero on weight is applied.
The weight shortfall is thus at most $\mu m \cdot \Theta(1) k^{2/3} n^{1/3}$.

The climb phase is where a rapid increase in separation is obtained, via a doubling of separation across dyadic scales. A careful calculation is needed to bound above the shortfall by  $mk^{2/3}n^{1/3}$ for the climb (and for the descent) phase.  The calculation will estimate the loss incurred across the dyadic scales using the parabolic loss in weight of Assumption~\ref{a.limit shape assumption}. This assumption plays a crucial role due to the rapid increase in separation of the curves, which causes a corresponding increase in anti-diagonal separation of the paths being considered.
The climb and descent phases are at the heart of the construction.

Cruise must maintain consecutive separation of order $k^{-2/3}n^{2/3}$ for a duration of roughly $n/3$. We construct paths that  travel through an order of $k$ consecutive boxes of height $k^{-1}n$ and width $k^{-2/3}n^{2/3}$. As the KPZ one-third exponent for energy predicts, each passage of a path across a box incurs weight shortfall of order $k^{-1/3}n^{1/3}$. Cruise weight shortfall is thus of order $m\cdot k\cdot k^{-1/3}n^{1/3} = mk^{2/3}n^{1/3}$.

Now we turn to giving the details of each phase. As in Section~\ref{s:bk}, we adopt a rounding convention for coordinates which are not integers, but, in contrast, we will be ignoring the resulting terms of $\pm 1$. More precisely, all expressions for coordinates of points should be rounded down, but  extra terms of $\pm 1$ which thus arise in non-coordinate quantities, such as expected values of weights, will be absorbed into constants without explicit mention.

\subsection{The construction in detail}

\subsubsection{Take-off}
In this phase, the $m$ curves will travel from $(1,1), \ldots,$ $(1,m)$ to the line $x+y = 2k^{2/3}n^{1/3}$.
Since we will make no non-trivial claim about the weight of take-off curves, we may choose these $m$ curves to be any disjoint upright paths, with the $i^\text{th}$  starting at $(1, i)$ and ending at
\begin{equation}\label{e.first segment end pos}
\left(k^{2/3}n^{1/3} - \left(\frac{m}{2}-i\right)k^{-1/3}n^{1/3},\  k^{2/3}n^{1/3} + \left(\frac{m}{2}-i\right)k^{-1/3}n^{1/3}\right).
\end{equation}
The next statement suffices to show that these $m$ disjoint paths exist; we omit the straightforward proof.

\begin{lemma}
Suppose given $m$ starting points $\{(1,i) : i\in\intint{m}\}$; and $m$ ending points $\{(x_i, y_i) : i\in\intint{m}\}$ on the line $x+y = r$, for some $r\geq m+1$, with $1\leq x_1 < x_2 < ... < x_m\leq m$. Then there exist $m$ disjoint paths such that the $i$\textsuperscript{th} connects $(1,i)$ to $(x_{m-i+1}, y_{m-i+1})$. 
\end{lemma}
The hypothesis of this lemma that $2k^{2/3}n^{1/3} \geq m + 1$ is satisfied when $n\geq 1$ because $m\leq k\leq n$.

\begin{figure}[h]

    \begin{subfigure}[t]{0.57\textwidth}
		\centering
		\includegraphics[width=\textwidth]{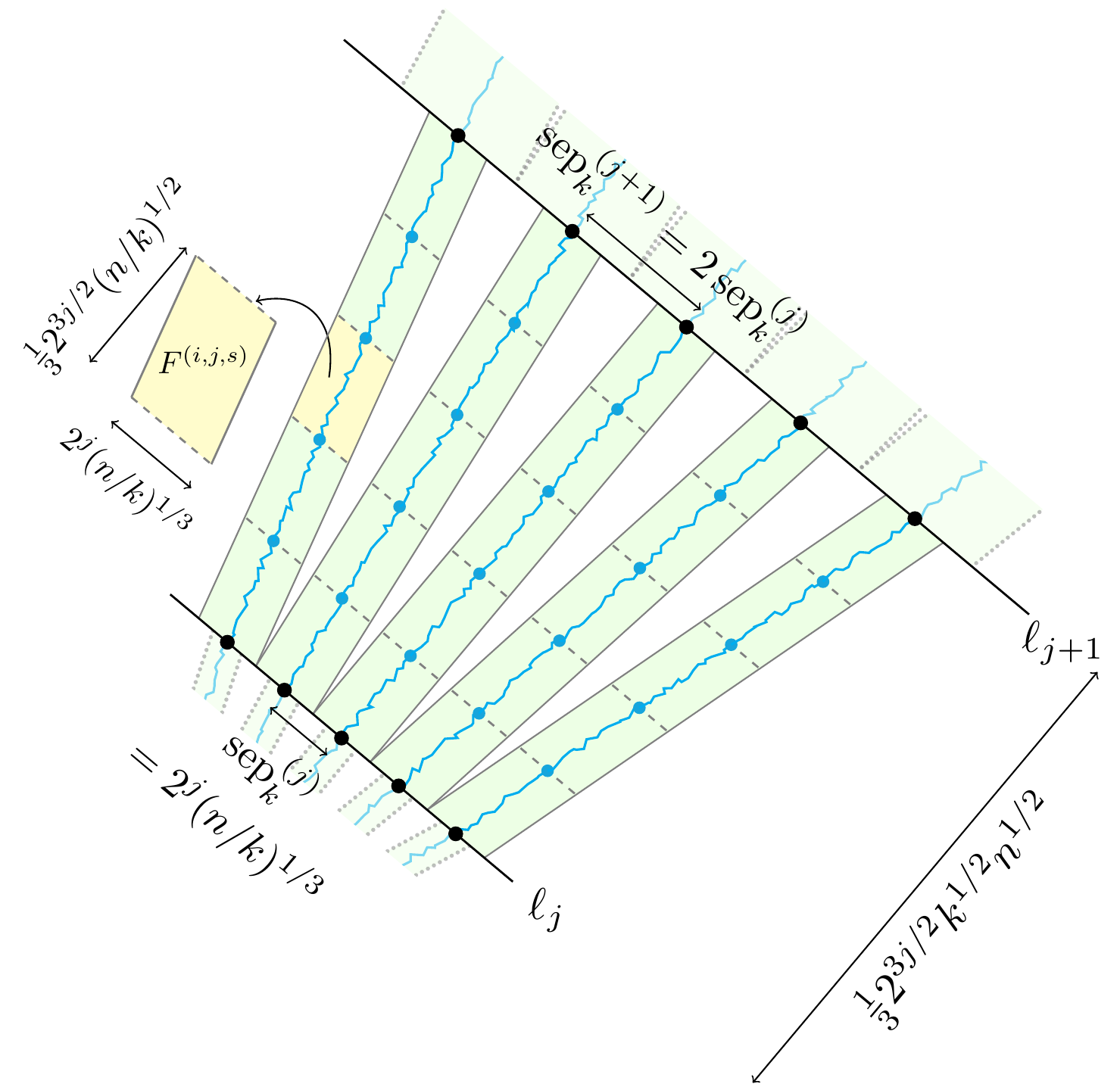}
		\caption{One segment of climb, the second phase}
		\label{f.construction segment 2}
	\end{subfigure} %
	\hspace{0.8cm}%
	\begin{subfigure}[t]{0.37\textwidth}
		\centering
		\raisebox{10mm}{\includegraphics[width=\textwidth]{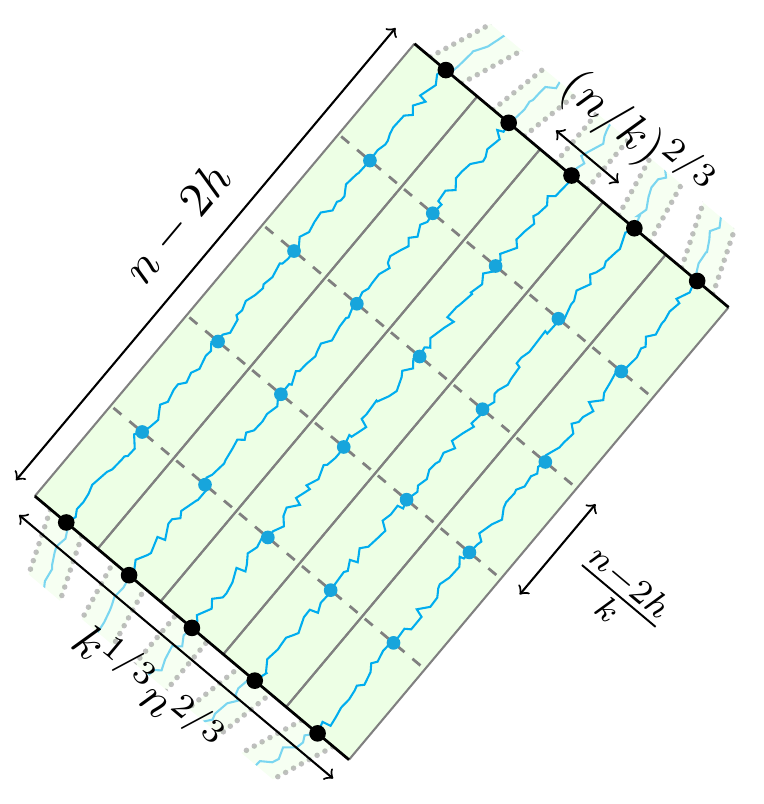}}
		\caption{Cruise, the third phase}
		\label{f.construction segment 3}
	\end{subfigure}
 \caption{In Panel A is a depiction of climb, between levels $j$ and $j+1$. Here $m=k=5$ and the highlighted flight corridor (small parallelogram) is $F^{(i,j,s)}$ for $i=1$ (first curve) and $s=3$ (third sublevel). Observe that any large green parallelogram with opposite sides on $\smash{\ell_{j}}$ and $\smash{\ell_{j+1}}$ does \emph{not} have an on-scale aspect ratio as $\smash{(\sep{j})^{3/2} = 2^{3j/2}k^{-1/2}n^{1/2}}$ is a factor of $k$ short of matching the separation along the diagonal between the two lines. However, after dividing the large parallelogram into $k$ smaller parallelograms $F^{(i,j,s)}$, the aspect ratio becomes on-scale as the small parallelogram's height is exactly $\frac{1}{3}2^{3j/2} k^{-1/2} n^{1/2}$. Also depicted in lower opacity is how the construction continues on a larger and smaller scale in the succeeding and preceding levels. In Panel B is depicted the simpler cruise phase, and how the second phase connects to it on either side. Here $h$ is the distance along the diagonal occupied by the first and second phases on the lower side.}
 \label{f.construction}
\end{figure}

\subsubsection{Climb}
This phase concerns the construction of curves as they pass through diagonal coordinates between $k^{2/3}n^{1/3}$ and
a value $h$ that we will specify in \eqref{e.height defn}.
In Lemma~\ref{l.height bound}, we will learn that $cn \leq h \leq n/2$; the climb phase thus has duration $\Theta(n)$, since $k$ is supposed to be at most a small constant multiple of $n$.

During climb, the order of separation rises from $k^{-1/3}n^{1/3}$ to  $k^{-2/3}n^{2/3}$. Separation will double during each of several segments into which the phase will be divided. The number $N$ of segments is chosen to satisfy 
\begin{equation}
2^{N} \in  k^{-1/3}n^{1/3} \cdot [1,2) \, . \label{e.number of levels}
\end{equation}

 A depiction of one of the segments is shown in Figure~\ref{f.construction segment 2}.

In order that the constructed curves remain disjoint and incur a modest weight shortfall, 
we will insist that they pass through a system of disjoint parallelograms whose geometry respects KPZ scaling:
the width of each parallelogram, and
the anti-diagonal offset between its lower and upper sides, will have the order of the two-thirds power of the parallelogram's height.

Segments $[2\ell_{j-1},2\ell_j]$ will be indexed by $j \in \intint{N}$, with  
\begin{equation}\label{e.ell defn}
\ell_j := \ell_{j-1} +  3^{-1} 2^{3(j-1)/2}k^{1/2}n^{1/2},
\end{equation}
and $\ell_0 = k^{2/3}n^{1/3}$. Climb begins where takeoff ends, at the diagonal coordinate $2\ell_0$: see~\eqref{e.first segment end pos}.

By level~$j$, we mean $\big\{ (x,y) \in \Z^2: x + y = 2\ell_j \big\}$. 
The $i^\text{th}$ curve will intersect level~$j$ at a unique point
$( \ell_j - \pos{i}{j}, \ell_j + \pos{i}{j})$,
where $\pos{i}{j}$  is inductively defined by 
\begin{equation}\label{e.pos defn}
\pos{i}{j} := \pos{i}{j-1} + \left(\frac m2 -i\right)2^{j-1}k^{-1/3}n^{1/3}
\end{equation}
from initial data $\pos{i}{0} := (m/2-i) k^{-1/3}n^{1/3}$ that is chosen consistently with \eqref{e.first segment end pos}.

We indicated that separation would double during each segment in climb from an initial value of  $k^{-1/3}n^{1/3}$.
This is what our definition ensures:  
the separation $\sep{j}$ between positions on the $j^\text{th}$ level is given by
\begin{equation}\label{e.sep defn}
\sep{j} := \pos{i}{j} - \pos{i-1}{j} = 2^j k^{-1/3}n^{1/3},
\end{equation}
where the latter equality is due to \eqref{e.pos defn}; and $\sep{0} = \pos{i}{0} - \pos{i-1}{0} = k^{-1/3}n^{1/3}$.

For indices $i$ 
that differ from the midpoint value $m/2$ by a unit order,
the curve of index $i$ has an anti-diagonal displacement of order $2^{j-1} k^{-1/3}n^{1/3}$ in its traversal between levels $j-1$ and $j$; see Figure~\ref{f.on scale parallelogram} for a depiction of anti-diagonal displacement. The  $(3/2)^{\text{th}}$ power  $2^{3(j-1)/2}k^{-1/2}n^{1/2}$ 
would seem a natural candidate for the value of $\ell_j$. Our definition in~\eqref{e.ell defn} includes a further factor of $k$, reflecting the greater separation available for curves at the edge, for which the index $i$ is close to zero or $m$. This dilation by $k$ influences the form of construction in the climb phase, and we will discuss it further soon. 

\begin{figure}[h]
   \centering
        \includegraphics[width=.75\textwidth]{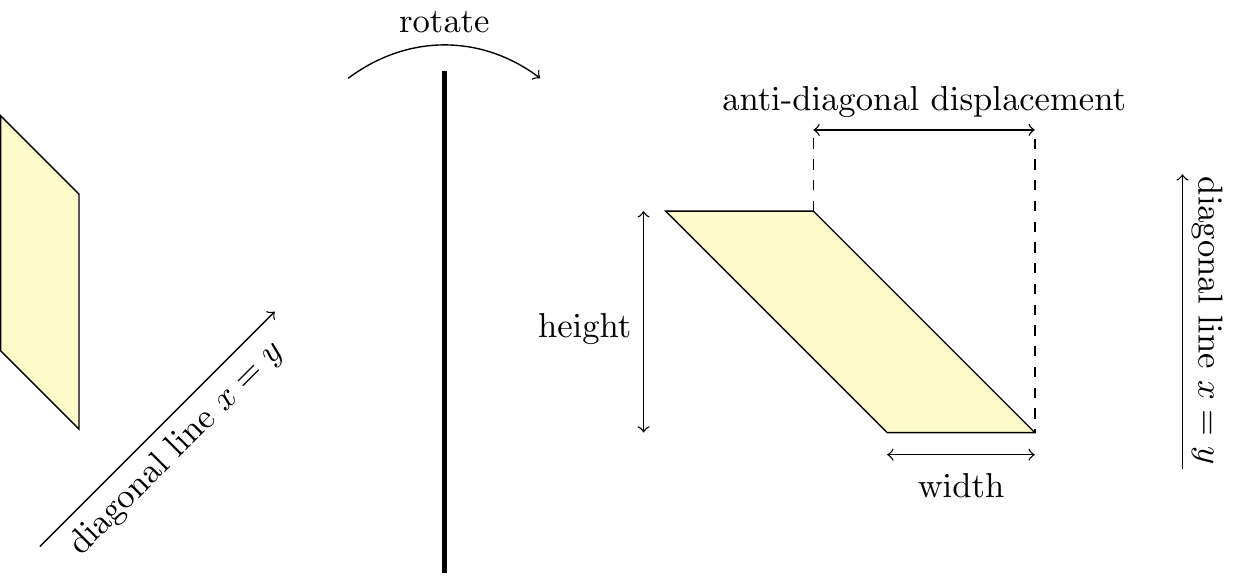}
 \caption{In order that the weight loss of the heaviest path constrained to lie in a parallelogram be on-scale, we need the anti-diagonal displacement and width of the parallelogram to both be of order $(\text{height})^{2/3}$. On the left is the parallelogram in the original coordinates, which has been rotated on the right for clarity so that the diagonal $x=y$ is visually vertical.}
 \label{f.on scale parallelogram}
\end{figure}

The height $h$ that marks the end of the climb phase may now be set: 
\begin{equation}
h := \ell_N.  \label{e.height defn}
\end{equation}

\subsubsection*{Flight corridors.} We have indicated that the curves under construction will be forced to pass through certain disjoint parallelograms. The latter regions will be called {\em flight corridors}. We may consider a flight corridor delimited by levels $j$ and $j+1$ of width $\sep{j}$. The width $2^j k^{-1/3}n^{1/3}$ and  height  $\smash{\frac{1}{3}2^{3j/2}k^{1/2}n^{1/2}}$ would, however, violate the relation $\mathrm{width} \approx \smash{\mathrm{height}^{2/3}}$ due to a mismatch in the exponent of $k$. The culprit is the extra factor of $k$ in the definition \eqref{e.ell defn}.

 The desired aspect ratio and anti-diagonal displacement conditions for all the curves will be obtained by use of a system of $k$ consecutive flight corridors for each curve. 
 Let $s \in \intint{k}$. Consider the planar line segment that runs from
\begin{gather}
\left(1-\frac{s-1}{k}\right)\left(\ell_{j}-\pos{i}{j},\, \ell_{j}+\pos{i}{j}\right) + \frac{s-1}{k}\left(\ell_{j+1}-\pos{i}{j+1},\, \ell_{j+1}+\pos{i}{j+1}\right) + (1,0)\nonumber\\
\text{to}\label{e.segment 2 parallelogram center coordinates}\\
\left(1-\frac{s}{k}\right)\left(\ell_{j}-\pos{i}{j},\, \ell_{j}+\pos{i}{j}\right) + \frac{s}{k}\left(\ell_{j+1}-\pos{i}{j+1},\, \ell_{j+1}+\pos{i}{j+1}\right) \, . \nonumber
\end{gather}
The $s$\textsuperscript{th} flight corridor $\paragram{i}{j}$ between levels $j$ and $j+1$ for the $i$\textsuperscript{th} curve
consists of those $(x,y) \in \Z^2$ that are displaced from some element in the just indicated line segment
by a vector $(t,-t)$ for some~$t$ with absolute value at most $2^{-1}\sep{j} - 1$. Thus the flight corridors associated to the $i$\textsuperscript{th} curve for different values of $i$ are disjoint.
 See Figure~\ref{f.construction}. The addition of $(1,0)$ in the first line of \eqref{e.segment 2 parallelogram center coordinates} is to ensure that the consecutive flight corridors of the same curve are disjoint.
 The $s^\text{th}$ subpath of the $i^\text{th}$ curve from level $j$ to $j+1$ 
 is defined to be a path of maximum weight between the two planar points just displayed that remains in the flight corridor  $\paragram{i}{j}$.
 Its weight will be denoted by
$\pathweight{i}{j}{2}$. Note that the $s$\textsuperscript{th} subpath's ending point in $\paragram{i}{j}$ and the $(s+1)$\textsuperscript{th} subpath's starting point in \smash{$F^{(i,j,s+1)}$} are adjacent to each other, which allows us to concatenate these curves.

The height of $\paragram{i}{j}$ is $k^{-1} \big( \ell_{j+1} - \ell_j \big) = 3^{-1} 2^{3j/2}k^{-1/2}n^{1/2}$ and its width is $\sep{j} = 2^{j}k^{-1/3}n^{1/3}$; so the flight corridor satisfies  the relation $\mathrm{height} = \frac{1}{3}\cdot\mathrm{width}^{3/2}$.

Anti-diagonal displacement at opposite ends of the flight corridor is highest for curves $1$ and $m$. For these paths, this displacement between levels $j$ to $j+1$ is 
$$
k^{-1}(\pos{i}{j+1}-\pos{i}{j}) = \frac{1}{2}2^jmk^{-4/3}n^{1/3}
$$
in view of \eqref{e.pos defn}. As the height gain in $F^{(i,j,s)}$ is $3^{-1}2^{3j/2}k^{-1/2}n^{1/2}$, and $m\leq k$, the anti-diagonal displacement for these corridors is at most $C\cdot(\text{height})^{2/3}$.

\subsubsection{Cruise}

Curves enter cruise at separation $k^{-2/3}n^{2/3}$, and must be maintained there for a duration of order $n$: see Figure~\ref{f.construction segment 3}. More precisely, and recalling that $h=\ell_{N}$,   
the $i^\text{th}$ curve enters cruise  at 
$$\left(h-\left(\frac{m}{2}-i\right)k^{-2/3}n^{2/3}, h+\left(\frac{m}{2}-i\right)k^{-2/3}n^{2/3}\right).$$
We set $\ell'_0 = h$, and $\ell'_j = \ell'_{j-1} + \frac{n-2h}{k} = h+(n-2h)\cdot\frac{j}{k}$
for $j \in \intint{k}$.
We will demand that the $i^\text{th}$ curve visits each point
\begin{equation}
\left(\ell'_{j}-\left(\frac{m}{2}-i\right)k^{-2/3}n^{2/3}, \ell'_{j}+\left(\frac{m}{2}-i\right)k^{-2/3}n^{2/3}\right). \label{e.3rd segment positions}
\end{equation}
Between consecutive points, this curve will be be constrained to remain in a flight corridor that comprises those vertices in $\Z^2$ that are displaced from the planar line segment that interpolates the pair of points by a vector $(t,-t)$ for some $t$ with absolute value at most $2^{-1}k^{-2/3}n^{2/3}$, which will be made to be disjoint by a unit displacement of the bottom side, as in \eqref{e.segment 2 parallelogram center coordinates}.
Let $\smash{\pathweightnos{i}{j}{3}}$ be the weight of the $i^\text{th}$ curve in the $j$\textsuperscript{th} corridor; namely, the maximum weight of paths that interpolate the endpoint pair and remain in the corridor.

Momentarily, we will verify that $h = \Theta(n)$. The new flight corridor thus satisfies the relation 
$$\mathrm{width} \in \left[C^{-1}\cdot\mathrm{height}^{2/3}, C\cdot \mathrm{height}^{2/3}\right]$$
 for a positive constant $C$. The corridor has no anti-diagonal displacement.

\begin{lemma}\label{l.height bound}
There exist positive $c_1$ and $c$ such that if $k<c_1 n$, then $h = \ell_{N}$ satisfies $cn\leq h\leq n/2$ for large enough $n$.
\end{lemma}
\begin{proof}
	This inference follows by noting that%
	\begin{align}\label{e.second segment height}
	h = \ell_{0}+\sum_{j=0}^{N-1} 3^{-1} 2^{3j/2}k^{1/2}n^{1/2} 
	&= \ell_{0}+\frac{1}{3(2\sqrt 2-1)}\left(2^{3N/2}-1\right)k^{1/2}n^{1/2}\\
	&= k^{2/3}n^{1/3}+\frac{1}{3(2^{3/2}-1)}\big( n - (kn)^{1/2} \big) \, .\nonumber\qedhere
	\end{align}
\end{proof}

\subsubsection{Descent and landing} These two phases are specified symmetrically with climb and take-off.

\subsection{Bounding below expected weight in the construction}\label{s.construction expected weight}

We will argue, in two steps, that the constructed curves attain the sought weight $\mu n m - \Theta(mk^{2/3}n^{1/3})$ with high probability. In this section, we will show that, under Assumptions~\ref{a.limit shape assumption}  and~\ref{a.one point assumption}, the expected total curve weight is at least $\mu n m - Cmk^{2/3}n^{1/3}$ for a constant $C>0$. In the next, the desired bound on 
curve weight will be obtained by showing that this weight concentrates around its mean due to the independence of contributions from the various flight corridors.

We will control the contribution from a given flight corridor using the lower bound on the expectation of the constrained point-to-point weight from Proposition~\ref{p.constrained lower tail}. That proposition will permit us to derive the next result, the conclusion of the present section.

\begin{proposition}\label{p.expected melon weight lower bound}
Let $X$ denote the sum of the weights of the $m$ curves in our construction. There exist $n_0 \in \N$, $C > 0$ and $c_1>0$ such that, for $n>n_0$, $k<c_1 n$ and $m\leq k$,
$\E[X] \geq \mu nm - Cmk^{2/3}n^{1/3}$.
\end{proposition}

\begin{proof}
The values of $C$ and $c$  may change from line to line but they do not depend on $n$, $m$, or $k$.

	We start with the weight of a curve fragment between two levels during climb. That is, we find a lower bound on $\sum_{s=1}^k \E[\pathweight{i}{j}{2}]$, where recall that $\pathweight{i}{j}{2}$ the weight of the heaviest path that travels between the points~\eqref{e.segment 2 parallelogram center coordinates}  without exiting the flight corridor~$F^{(i,j,s)}$. 

In specifying the climb phase, we noted that the condition in Proposition~\ref{p.constrained lower tail} that the anti-diagonal displacement of the flight corridor  is of order $r^{2/3}$ holds. Thus, using Proposition~\ref{p.constrained lower tail} with the settings $r=k^{-1}(\ell_{j+1}- \ell_j)$ and $zr^{2/3}=k^{-1}(\pos{i}{j+1}-\pos{i}{j})$, we obtain 
	\begin{align}\label{e.first expression of EXijs}
	\E[\pathweight{i}{j}{2}] \geq \frac{\mu}{k}\left(\ell_{j+1}-\ell_{j}\right) - C\frac{\left(\pos{i}{j+1}-\pos{i}{j}\right)^2}{k\left(\ell_{j+1}-\ell_j\right)} - \frac{c}{k^{1/3}}\left(\ell_{j+1}-\ell_j\right)^{1/3}.
	\end{align}
	We simplify the first two terms of \eqref{e.first expression of EXijs} using the expressions from \eqref{e.ell defn} and \eqref{e.pos defn} to obtain
	\begin{align*}
	\frac{\mu}{3k} 2^{3j/2}k^{1/2}n^{1/2} - C&\frac{\frac{1}{k^2}\left(\frac m2 -i\right)^2 2^{2j}k^{-2/3}n^{2/3}}{\frac{1}{3k} 2^{3j/2}k^{1/2}n^{1/2}}\\
	&= \frac{\mu}3 2^{3j/2}k^{-1/2}n^{1/2} - C\left(\frac m2 -i\right)^2 2^{j/2}k^{-13/6}n^{1/6}.
	\end{align*}
	For the last term of \eqref{e.first expression of EXijs} we have, from \eqref{e.ell defn},
	\begin{align*}
	\frac{c}{k^{1/3}}\left(\ell_2^{(j+1)}-\ell_2^{(j)}\right)^{1/3} 
	&= \frac{c}{k^{1/3}}\cdot2^{j/2}k^{1/6}n^{1/6}
	= c\cdot 2^{j/2}k^{-1/6}n^{1/6},
	\end{align*}
	where we have absorbed the constant factor of $3^{-1/3}$ into the value of $c$.
	Combining the three preceding displays and summing over $s$ and $i$ gives a lower bound on the expected weight of all the curves between levels $j$  and $j+1$:
	\begin{align*}
	\sum_{i=1}^m\sum_{s=1}^k \E[\pathweight{i}{j}{2}]&\geq \frac{\mu}{3}\cdot 2^{3j/2}mk^{1/2}n^{1/2} - C\cdot2^{j/2}m^3k^{-7/6}n^{1/6} - c2^{j/2}mk^{5/6}n^{1/6}\\
	&\geq \frac{\mu}{3}\cdot 2^{3j/2}mk^{1/2}n^{1/2} - C\cdot2^{j/2}mk^{5/6}n^{1/6};
	\end{align*}
	here we used that $m^3 \leq mk^2$ to bound the second term in the latter inequality.
	Summing from $j=0$ to $N-1$ gives a lower bound for the total weight of the climb phase of
	\begin{align}
	\sum_{i=1}^m\sum_{j=0}^{N-1}\sum_{s=1}^k \E[\pathweight{i}{j}{2}] 
	&\geq \frac{\mu m}{3}\cdot \sum_{j=0}^{N-1}2^{3j/2}k^{1/2}n^{1/2} - C\cdot 2^{N/2}mk^{5/6}n^{1/6}\nonumber\\
	&= \mu(h-k^{2/3}n^{1/3})\cdot m - Cmk^{4/6}n^{1/3}
	= \mu hm - Cmk^{2/3}n^{1/3}, \label{e.complete lower bound of second segment}
	\end{align}
	where $2^{N/2} = k^{-1/6}n^{1/6}$ from \eqref{e.number of levels} and the expression of $h$ from \eqref{e.second segment height} were used.

 As for cruise, an easy computation from Proposition~\ref{p.constrained lower tail} yields that
	\begin{equation}\label{e.lower bound for third segment}
	\E[\pathweightnos{i}{j}{3}] \geq \frac{\mu(n-2h)}{k} - ck^{-1/3}(n-2h)^{1/3}\geq \frac{\mu(n-2h)}{k} - ck^{-1/3}n^{1/3}.
	\end{equation}
We write $\overline X^{(i,j,s)}_{n,k,2}$ for the descent counterpart of the climb flight corridor weight maximum. The bound on mean valid for climb holds equally for descent.
 Combining climb, descent and cruise weight bounds with the zero bound for take-off and landing, the  total weight $X$ in our construction is seen to satisfy 
	\begin{equation}\label{e.construction weight}
	X \geq \sum_{i,j,s}\left(\pathweight{i}{j}{2}+ \overline X^{(i,j,s)}_{n,k,2}\right) + \sum_{i,j}\pathweightnos{i}{j}{3} \, . 
	\end{equation}
 Taking expectation, and applying \eqref{e.complete lower bound of second segment} and \eqref{e.lower bound for third segment}, we obtain
	$$\E\left[\sum_{i,j,s}\left(\pathweight{i}{j}{2}+ \overline X^{(i,j,s)}_{n,k,2}\right) + \sum_{i,j}\pathweightnos{i}{j}{3}\right] \geq \mu nm - Cmk^{2/3}n^{1/3} \, ,$$
	the conclusion of Proposition~\ref{p.expected melon weight lower bound}. 
\end{proof}

\subsection{One-sided concentration of the construction weight}\label{s.construction concentration}
We now know that the expected weight of the construction is correct. To prove Theorem~\ref{t.flexible construction},  we argue that the weight is unlikely to fall much below its mean. We will stochastically dominate the summands in the expression for $X$ from \eqref{e.construction weight} with independent exponential random variables of varying parameters, and use the following concentration result from \cite{janson2018tail} for sums of such variables. We denote by $\Exp(\lambda)$ the exponential distribution with rate $\lambda$.

\begin{proposition}[Theorem 5.1 (i) of \cite{janson2018tail}]\label{p.concentration of exponentials}
Let $W=\sum_{i=1}^{n} W_{i}$ where $W_{i} \sim \Exp\left(a_{i}\right)$ are independent. Define
\begin{align*}
\nu :=\E W=\sum_{i=1}^{n} \E W_{i}=\sum_{i=1}^{n} \frac{1}{a_{i}},\,\, \,\,
a_{*}  :=\min _{i} a_{i}.
\end{align*}
Then for  $\lambda\geq 1$,
$$
\P\left(W \geq \lambda \nu\right) \leq \lambda^{-1} e^{-a_{*} \nu(\lambda-1-\log \lambda)}.
$$
\end{proposition}

\begin{proof}[Proof of Theorem~\ref{t.flexible construction}]
	Let $K\geq 0$ and $0< L_1\leq L_2$. As before, let $U = U_{r,\ell r^{2/3}, zr^{2/3}}$ be a parallelogram with width $\ell r^{2/3}$, height $r$, and anti-diagonal displacement $zr^{2/3}$, where $|z|\leq K$ and $L_1\leq \ell \leq L_2$. Let $X_r^{U}$ denote the weight of the heaviest midpoint-to-midpoint path that lies in~$U$. By Proposition \ref{p.constrained lower tail}, and $\E[X^U_{r}] \leq \E[X_r] \leq \mu r$ (which is due to Assumption~\ref{a.limit shape assumption}), there exists $c_4=c_4(L_1,L_2,K)>0$ such that, for $r>r_0 = r_0(L_1, L_2, K)$ and $\theta>\theta_0 = \theta_0(L_1, L_2, K)$,
	$$\P\left(X^U_r -\E[X^U_r] < -\theta r^{1/3}\right) \leq e^{-c_4\theta} \, ;
	$$
	or
	\begin{equation}\label{e.stochastic domin of constrained weight}
	\P\left(X^U_r - \E[X^U_r] < -\theta\right) \leq \exp(-c_4\theta/r^{1/3}) \, ,
	\end{equation}
	the latter for $\theta$ larger than $\theta_0 r^{1/3}$ and $r>r_0$. Equation~\eqref{e.stochastic domin of constrained weight} says that we have the stochastic domination
	$$-\left(X^U_r-\E[X^U_r]\right) \leq_{\mathrm{sd}} \Exp(c_4/r^{1/3}) + \theta_0 r^{1/3},$$
	where, in an abuse of notation, $\Exp(\lambda)$ denotes a random variable with the exponential distribution of rate $\lambda$, and $X \leq_{\mathrm{sd}} Y$ denotes that the distribution of $X$ is stochastically dominated by that of~$Y$. 
	Thus, in our construction, we have a coupling
	\begin{align}
	\begin{split}\label{e.stochastic dom by exp}
	-\left(\pathweight{i}{j}{2}-\E[\pathweight{i}{j}{2}]\right) &\leq W_2^{(i,j,s)}+ \theta_0 2^{j/2}k^{-1/6}n^{1/6},\\
	-\left(\overline X^{(i,j,s)}_{n,k,2} -\E[\overline X^{(i,j,s)}_{n,k,2}]\right) &\leq \overline W_2^{(i,j,s)}+\theta_0 2^{j/2}k^{-1/6}n^{1/6}, \quad \text{and}\\
	-\left(\pathweightnos{i}{j}{3}-\E[\pathweightnos{i}{j}{3}]\right) &\leq W_3^{(i,j)} + \theta_0 k^{-1/3}n^{1/3}.
	\end{split}
	\end{align}
	Here the random variables on the right-hand side are independent and distributed as
	\begin{equation}\label{e.exponentials}
	\begin{split}
	W_2^{(i,j,s)} &\sim \Exp\left(c_42^{-j/2}k^{1/6}n^{-1/6}\right),\\
		\overline W_2^{(i,j,s)} &\sim \Exp\left(c_42^{-j/2}k^{1/6}n^{-1/6}\right),\quad\text{and}\\
		W_3^{(i,j)} &\sim \Exp\left(c_4k^{1/3}n^{-1/3}\right).
	\end{split}
	\end{equation}
 By our construction, $K$, $L_1$, and $L_2$ in Proposition~\ref{p.constrained lower tail} may be chosen independently of $i$ and~$j$.
Thus, whatever the dependence of $c_4$ and $\theta_0$ 
on $i$ and $j$, they are uniformly bounded away from $0$ and $\infty$ respectively, and so may be assumed to be constant.
We set 
	$$W :=\sum_{i,j,s}\left(W_{2}^{(i,j,s)}+\overline W_{2}^{(i,j,s)}\right) + \sum_{i,j}W_{3}^{(i,j)}.$$
	We wish to use Proposition~\ref{p.concentration of exponentials} on the sum $W$  of independent exponential random variables, and so we must estimate Proposition~\ref{p.concentration of exponentials}'s parameters. We have
	\begin{align*}
	\E\left[\overline W_{2}^{(i,j,s)}\right] = \E\left[W_{2}^{(i,j,s)}\right] &= c_4^{-1}2^{j/2}k^{-1/6}n^{1/6} \qquad \text{and} \qquad	\E\left[W_{3}^{(i,j)}\right] = c_4^{-1}k^{-1/3}n^{1/3}.
	\end{align*}
	Summing the expressions of the last display over the indices and using that $2^{N/2} = k^{-1/6}n^{1/6}$ from~\eqref{e.number of levels} gives that the total mean $\nu$ satisfies
	\begin{equation} \label{e.nu value}
	\nu := \E[W] \leq 8c_4^{-1}mk^{2/3}n^{1/3}.
	\end{equation}
	Noting that the coefficients of $\theta_0$ in \eqref{e.stochastic dom by exp} are the same as the mean of the corresponding exponential random variable up to a factor of $c_4$, similarly summing these coefficients shows the stochastic domination
	$$-(X-\E[X]) \leq W + 8\theta_0mk^{2/3}n^{1/3}.$$
	Using that $2^N = k^{-1/3}n^{1/3}$ from \eqref{e.number of levels}, we see that the minimum $a^*$ of the rates of the exponential random variables defined in \eqref{e.exponentials} is given by
	\begin{equation}\label{e.a* value}
	a^* = \min(c_42^{-N/2}k^{1/6}n^{-1/6}, c_4k^{1/3}n^{-1/3}) = c_4k^{1/3}n^{-1/3}.
	\end{equation}
	 Let $S = \max(c_4^{-1},\theta_0)$. By Proposition~\ref{p.expected melon weight lower bound}, with this result contributing the value of $C$, and the stochastic domination~\eqref{e.stochastic dom by exp},
	\begin{align*}
	\P\Big(X < \mu nm - (C+20S)mk^{2/3}n^{1/3}\Big) &= \P\left(-(X-\E[X]) > 20Smk^{2/3}n^{1/3}\right)\\
	&\leq \P\left(W \geq 20Smk^{2/3}n^{1/3} - 8\theta_0mk^{2/3}n^{1/3}\right).
	\end{align*}
	Since $20S-8\theta_0\geq 12S \geq 12c_4^{-1}$, the latter quantity is bounded using Proposition~\ref{p.concentration of exponentials}, \eqref{e.nu value}, and \eqref{e.a* value}:
	\begin{align*}
	\P\left(W > 12c_4^{-1}mk^{2/3}n^{1/3}\right)	\leq \P\left(W > \frac{12}{8} \nu\right)
	&\leq \frac23 \exp\left(-a^*\nu\left(\frac32-1-\log\frac32\right)\right)\\
	&\leq \exp\left(-cmk\right).
	\end{align*}
	Numerical evaluation of the exponent shows that we may take $c=3/4$.

In view of the paragraphs after \eqref{e.segment 2 parallelogram center coordinates} and \eqref{e.3rd segment positions},  flight corridors during climb, descent and cruise lie within the strip around the diagonal of width $2mk^{-2/3}n^{2/3}$; thus, the transversal fluctuations of the constructed curves in these phases also satisfy this bound.
By setting $C_1 = C+20S$, we complete the proof of Theorem~\ref{t.flexible construction}.
\end{proof}

\begin{remark}
The argument just given shows that, for positive constants $c$ and $x_0$ and  $x\geq x_0$,
$$
\P\Big(X < \mu nm - (C+x)mk^{2/3}n^{1/3}\Big) \leq \exp\left(-cmkx\right) \, .
$$

\end{remark}

\section{Combinatorics of watermelon geometry}
\label{s:geometry}

Here we  prove three geometric results that we have invoked several times; namely, Propositions \ref{p.starting and ending points}, \ref{p.melon interlacing}, and \ref{p:ordering}. The arguments are geometric and combinatorial, rather than probabilistic, and, we believe, of independent interest.  Proposition \ref{p.starting and ending points} asserts that a $k$-geodesic watermelon exists whose curves begin and end at certain prescribed points on the boundary of the square $\intint{n}^2$. 

We begin by defining a partial order on paths which we will need in the proof of Proposition~\ref{p.starting and ending points}. Having a partial order will allow us to consider maximal elements, which will have nice geometric properties.

Given two disjoint upright paths $\gamma$ and $\gamma'$ in $\intint{n}^2$, we say $\gamma'$ is vertically above $\gamma$, denoted $\gamma \preceq_{\mathrm{v}} \gamma'$, if
$$\exists x \in \intint{n} \text{ such that } \max\{y : (x,y) \in \gamma\} \leq \min\{y: (x,y) \in \gamma'\},$$
where $\min \emptyset = -\infty$ and $\max\emptyset = \infty$. Because we want this to be a partial order, we additionally impose reflexivity, i.e., $\gamma \preceq_{\mathrm{v}} \gamma$ for all upright paths $\gamma$. That this partial order is anti-symmetric, i.e., $\gamma\neq \gamma'$ implies that not both $\gamma \preceq_{\mathrm{v}} \gamma'$ and $\gamma' \preceq_{\mathrm{v}} \gamma$ hold, follows from the disjointness and upright nature of the two curves. In fact, two disjoint curves are comparable if and only if their projections on to the $x$-axis are not disjoint; we will refer to these projections as the \emph{$x$-projections} of the curves.

\begin{proof}[Proof of Proposition \ref{p.starting and ending points}]
The proof is an induction on $k$. The inductive hypothesis asserts that, for any $m$ and $n$, and given a collection of $k\leq \min(m,n)$ disjoint upright paths $\{\phi_i : i\in\intint{k}\}$ in $\intint{m}\times \intint{n}$, there exist $k$ disjoint upright paths $\gamma_1,\ldots,\gamma_k$ such that
\begin{enumerate}
\item $\cup_{j=1}^{k}\phi_j\subseteq \cup_{j=1}^{k}\gamma_j$;  and
\item for  $i \in \intint{k}$, 
$\gamma_i$ starts at $(1,k-i+1)$ and ends at  $(m,n-i+1)$.
\end{enumerate}

Taking $m=n$ and $\phi_1, \ldots, \phi_k$ to be the collection of curves in the given $k$-geodesic watermelon $\phi_n^k$, it is easy to see that this statement implies Proposition \ref{p.starting and ending points}: the set of $k$ disjoint paths $\{\gamma_i : i \in \intint{k} \}$ is a $k$-geodesic watermelon, because its weight is at least that of $\phi_n^k$, in view of  $\xi_v \geq 0$ for $v \in \Z^2$. 

Now we prove the inductive hypothesis. The base case $k=1$ is easy to verify. Assuming this statement at index $k-1$, we derive it at index $k$. Relabel the given curves if required so that $\phi_1$ is the topmost, and, if there is a tie, the leftmost. More precisely, consider the set of maximal elements of $\{\phi_i : i\in\intint{k}\}$ under the $\preceq_{\mathrm{v}}$ order; as they are disjoint and mutually incomparable, the maximal elements must have disjoint $x$-projections, and we label the maximal element with leftmost $x$-projection as $\phi_1$. Now remove the top row and the left column; which is to say, consider the restriction of $\phi_2,\ldots,\phi_{k}$ to $\llbracket2,m\rrbracket\times\intint{n-1}$. This is a collection of $k-1$ disjoint upright paths because $\phi_j, j\in\llbracket2,k\rrbracket$, are disjoint and upright.
Applying the assumed form of the inductive hypothesis to this collection, we obtain paths $\gamma_2',\ldots,\gamma_k'$, with $\gamma_i'$ from $(2,k-i+1)$ to $(m,n-i+1)$, such that $\cup_{j=2}^k \gamma_j'$ contains $\cup_{j=2}^k\phi_j \cap (\llbracket2,m\rrbracket\times \intint{n-1})$
 for $i \in \llbracket 2,k \rrbracket$. 
We select $\gamma_2', \ldots, \gamma_k'$ to be \emph{maximally to the right}; i.e., such that there is no collection of disjoint curves $\tilde\gamma_2, \ldots, \tilde\gamma_k$ that satisfy the three displayed conditions  with $\gamma_i' \preceq \tilde\gamma_i$ for all $i \in \llbracket 2,  k \rrbracket$ and $\gamma_j' \neq \tilde\gamma_j$ for some $j \in \llbracket 2 , k \rrbracket$. Finally we add the vertex $(1,k-i+1)$ to $\gamma'_i$ and call the resulting path $\gamma_i$, for each $i\in \llbracket2,k\rrbracket$.

It remains to construct $\gamma_1$. Starting from $(1,k)$, $\gamma_1$ first follows the vertical segment to $(1,y_0)$, where $y_0$ is the $y$-coordinate of the top vertex of $\cup_{j=2}^k\phi_j$ on the $x=1$ line; if no such vertex exists, we set $y_0=k$. We now follow the vertices of $\phi_1$ which are not elements of $\cup_{j=2}^k \gamma_j$ till we reach the last such vertex $(x_1,y_1)$. This is possible because Lemma~\ref{l.initial part of phi_1} ahead asserts that the leftmost vertex of $\phi_1\setminus\cup_{j=2}^k \gamma_j$ is reachable from $(1,y_0)$ and because Lemma~\ref{l.maximally to the right}, also ahead, says that no vertex of $\phi_1\setminus\cup_{j=2}^k\gamma_j$ is to the right of $\gamma_2$, which implies that the vertices of $\phi_1\setminus\cup_{j=2}^k \gamma_j$ form a union of path fragments which may be joined without intersecting $\gamma_2$. Now $\gamma_1$ follows the vertical segment from $(x_1,y_1)$ to $(x_1,n)$, if $y_1 <n$, and then horizontally to $(m,n)$.

     The displayed condition (2) is thus satisfied by $\{ \gamma_i: i \in \intint{k} \}$. It also immediately follows from the construction that the vertices of $\phi_1$ and the restriction of $\cup_{j=2}^k\phi_j$ to $\intint{m}\times\intint{n-1}$ are contained in $\cup_{j=1}^k \gamma_j$; the only vertices left to be accounted for are the vertices of $\cup_{j=2}^k\phi_j$ on the line $y=n$. But since $\phi_1$ is the leftmost maximal element, all these vertices are included in the final horizontal segment of $\gamma_1$. Indeed, if not, there would be a vertex of $\phi_{j_0}$ on the line $y=n$ with $x$-coordinate strictly smaller than $x_1$ for some $j_0\in\llbracket2,k\rrbracket$. Then either $\phi_1\preceq_{\mathrm{v}} \phi_{j_0}$ or, if these paths are not comparable, $\phi_1$ is not the leftmost maximal element; either way, a contradiction.
     Thus $\{\gamma_j : j \in \intint{k} \}$ satisfies the inductive hypothesis at index $k$, as we sought to show.
\end{proof}

\begin{lemma}\label{l.maximally to the right}
Any vertex of $\phi_1$ that lies to the right of $\gamma_2$ belongs to $\cup_{i=2}^k\gamma_i$.
\end{lemma}

\begin{proof} 
In search of a contradiction, suppose the contrary. 
Vertices not lying in $\cup_{i=2}^k\gamma_i$ will be called uncovered. Let $j$ be the smallest index such that an uncovered vertex $(x',y')$ of  $\phi_1$ lies between $\gamma_j$ and $\gamma_{j+1}$.  

We claim the following: No vertex of $\cup_{j=2}^k\phi_j$ lies in the quadrant whose south-east corner is $(x',y')$, i.e., $\{(x,y): x\leq x', y\geq y'\}$. To prove this, suppose to the contrary that there is a vertex of $\phi_{j_0}$ contained in this quadrant for some $j_0\in\llbracket2,k\rrbracket$. Consider the maximal element $\bar\phi$ of some chain starting from $\phi_{j_0}$. The union of the $x$-projections of the elements of this chain forms an interval. Let $I$ be the $x$-projection of $\bar\phi$, which is an interval. We have two cases: (i) $\max I < x'$ and (ii) $\max I \geq x'$. Case (i) contradicts that $\phi_1$ is the leftmost maximal element. In case (ii), the union of the $x$-projections of the elements of the chain contains $x'$ and hence, by the disjointness of $\{\phi_i : i\in\intint{k}\}$, there is an element of the chain which is larger in the $\preceq_{\mathrm{v}}$ order than $\phi_1$, which is again a contradiction.

Consider the alternative path $\gamma_j'$ formed by concatenating four curves: the portion of $\gamma_j$ up to and including the rightmost vertex on the line $y=y'$; the horizontal segment from this point to $(x',y')$; the vertical segment from $(x',y')$ to the lowest vertex of $\gamma_j$ on the line $x=x'$; and the portion of $\gamma_j$ from this latter point onwards. Note that replacing $\gamma_j$ by $\gamma_j'$ maintains the displayed conditions (1) and (2) in the proof of Proposition~\ref{p.starting and ending points} due to the above claim which ensures that no vertex in $\cup_{j=2}^k\phi_j \cap (\llbracket2,m\rrbracket\times \intint{n-1})$ becomes uncovered. However, the set $\{\gamma_i: i \in \llbracket 2,k \rrbracket\}$ was constructed to be maximally to the right, in contradiction to the last inference.
\end{proof}

\begin{lemma}\label{l.initial part of phi_1}
The leftmost vertex of $\phi_1\setminus\cup_{j=2}^k\gamma_j$ is reachable from $(1,y_0)$, as defined in the proof of Proposition \ref{p.starting and ending points}, by an upright path.
\end{lemma}

\begin{proof}
Let the leftmost lowermost vertex of $\phi_1\setminus\cup_{j=2}^k\gamma_j$ be $(x',y')$. We first note that Lemma~\ref{l.maximally to the right} implies that $y'\geq k$. So we only need to consider the case of the definition of $y_0$ where there is at least one vertex of $\cup_{j=2}^k\phi_j$ on the $x=1$ line.

We must prove that $y'\geq y_0$, as obviously $x'\geq 1$. Without loss of generality, assume that $(x_0, y_0) \in \phi_2$. Consider the maximal elements of all chains that $\phi_2$ is a member of, and consider the leftmost such maximal element $\bar\phi$. We claim that $\bar\phi = \phi_1$, which clearly implies that $y'\geq y_0$.

To prove this claim, note that the union of the $x$-projections of the elements of this chain form a single interval; call it $I$. Recall how we labelled $\phi_1$ and note that it precludes $\phi_1$'s $x$-projection from being disjoint from $I$, and so there must be an element of the chain, say $\phi_{j_0}$, whose $x$-projection is not disjoint from that of $\phi_1$. Now since $\phi_{j_0}$ and $\phi_1$ are disjoint, and by the labelling of $\phi_1$, we must have $\phi_{j_0} \preceq_{\mathrm{v}} \phi_1$. But this implies that $\bar\phi = \phi_1$, completing the proof.
\end{proof}

\subsection{Proof of melon interlacing, Proposition~\ref{p.melon interlacing}}
With Proposition~\ref{p.starting and ending points}, which allows us to pick the curves of geodesic watermelons to begin and end at the bottom-left and top-right corners of $\intint{n}^2$, in hand, we now move to showing Proposition~\ref{p.melon interlacing}, which says that we may also pick the curves such that they interlace. Instead of assuming that $\nu$ is continuous as in that proposition, we will prove the following deterministic statement, whose hypotheses hold almost surely when $\nu$ has no atoms. For a subset $A\subseteq \intint{n}^2$ of vertices, we first define its weight by
$$\ell(A) := \sum_{v\in A} \xi_v.$$

\begin{proposition}\label{p.stronger interlacing}
Suppose the vertex weights $\{\xi_v : v\in\intint{n}^2\}$ are non-negative and such that no two distinct subsets $A,B\subseteq \intint{n}^2$ have equal weight. Then the consecutive terms in the unique sequence of geodesic watermelons $\{\gamma_n^{j}: j\in\intint{n}\}$ produced by Proposition~\ref{p.starting and ending points} interlace.
\end{proposition}

The sequence of geodesic watermelons produced by Proposition~\ref{p.starting and ending points} is unique because the set of vertices of the $k$-geodesic watermelon is unique for every $k\in\intint{n}$ under the hypotheses of Proposition~\ref{p.stronger interlacing}.

\begin{proof}[Proof of Proposition~\ref{p.melon interlacing}]
This is immediate since when $\nu$ has no atoms and has support contained in $[0,\infty)$, the hypotheses of Proposition~\ref{p.stronger interlacing} are almost surely satisfied.
\end{proof}

After proving Proposition~\ref{p.stronger interlacing}, we will upgrade it to Proposition~\ref{p.specified melon interlacing}, which handles the case that $\nu$ has atoms, by a perturbation argument. We next introduce some notation to facilitate the proof of Proposition~\ref{p.stronger interlacing}.

Recall the partial order $\preceq$ introduced in Definition~\ref{d.time range}, where informally $\gamma\preceq \gamma'$ means $\gamma$ is to the left of $\gamma'$. Then as in \eqref{watermelondef}, for any $k$, let $\pordmel{k}{i}$ be the curves of $\phi_{n}^k$ produced by Proposition~\ref{p.starting and ending points}, ordered from left to right, i.e., $\pordmel{k}{1} \prec \pordmel{k}{2} \prec\ldots \prec \pordmel{k}{k}.$
For a set of vertices $A$, we call $\{\gamma_1,\ldots,\gamma_m\}$ a \emph{path fragment decomposition} of $A$ if each $\gamma_i$ is a connected component of $A$ (under the connectivity structure of $\Z^2$) as well as an upright path and $A = \cup_{i=1}^m \gamma_i$; note that not all sets have a path fragment decomposition. We then call the $\gamma_i$ \emph{path fragments}.

For $j\in \intint{k-1}$, let $R^k_j$ be the region between $\pordmel{k}{j}$ and $\pordmel{k}{j+1}$, defined by
$$R^k_j = \left\{(x,y) \in \intint{n}^2 : \pordmel{k}{j}\preceq \gamma_{(x,y)} \preceq \pordmel{k}{j+1}\right\},$$
where $\smash{\gamma_{(x,y)}}$ is the singleton path at $(x,y)$. Recall that an order relation between two paths is well-defined only if their time ranges are not disjoint. Since the elements of $\{\pordmel{k}{j} : j\in\intint{k}\}$ do not share the same time range, for certain $(x,y)$, it may be that only one of the inequalities $\pordmel{k}{j}\preceq \gamma_{(x,y)}$ and $\gamma_{(x,y)} \preceq \pordmel{k}{j+1}$ is well-defined (regardless of whether it is true or false). In these cases, the definition of $R_{j}^k$ above will only require that inequality for which the corresponding time-ranges are not disjoint to be true, and will disregard the remaining inequality which is not well-defined.
We also define $\smash{R^k_0 = \{(x,y) \in \intint{n}^2 : \gamma_{(x,y)} \preceq \pordmel{k}{1}\}}$ and $\smash{R_k^k = \{(x,y) \in \intint{n}^2 : \pordmel{k}{k}\preceq \gamma_{(x,y)}\}}$. We note that, by the ordering of $\{\pordmel{k}{j}:j\in\intint{k}\}$,
\begin{equation}\label{e.disjoint region condition}
R_{j_1}^k\cap R_{j_2}^k = \emptyset \quad\text{if}\quad |j_1-j_2|\geq 2.
\end{equation}
We also note that $\phi_{n}^k$ interlacing with $\phi_n^{k-1}$ is equivalent to $\smash{\pordmel{k}{j} \subseteq R_{j-1}^{k-1}}$ for each $j\in\intint{k}$.
We say that $\pordmel{k}{j}$ is in its \emph{home region} if $\pordmel{k}{j}\subseteq R^{k-1}_{j-1}$. We define the \emph{wilderness region} $W_{j}^{k}$ (for $\pordmel{k}{j}$) to be the complement of $\smash{R^{k-1}_{j-1}}$, i.e.,
$$W_j^k := \intint{n}^2 \setminus R_{j-1}^{k-1}.$$
The idea of the proof of Proposition~\ref{p.stronger interlacing} is the following. Consider the case that $\phi_n^k$ and $\phi_n^{k-1}$ do not interlace, and look at the excursions of $\pordmel{k}{j}$ into the respective wilderness regions $\smash{W_j^k}$ for each $j\in\intint{k}$. For each excursion, we identify a candidate path along the home region border, i.e., a subpath of $\pordmel{k-1}{j-1}$ or $\pordmel{k-1}{j}$, which has the same endpoints as the excursion. We then simultaneously swap all the excursions with the candidate paths. This results in two new collections of $k$ and $k-1$ disjoint paths (which do not necessarily interlace) which cover the same set of vertices with the same multiplicities. The weight maximality property of $\phi_n^k$ and $\phi_n^{k-1}$ will then imply that the two new collections of paths are also $(k-1)$- and $k$-geodesic watermelons, contradicting the uniqueness of the watermelons guaranteed by the hypothesis. See Figure~\ref{f.interlacing proof}.

To implement this plan, we define for each $j\in\intint{k}$ the set $\WPF{k}{j}$ of \emph{wilderness excursions} to be the path fragment decomposition of $\pordmel{k}{j}\cap W_j^k$. The first step is to identify the border candidate paths---subpaths of the paths of $\smash{\phi_n^{k-1}}$---which will be swapped with the excursions into the wilderness---into elements of $\WPF{k}{j}$.

\begin{figure}[h]
   \centering
        \includegraphics[width=\textwidth]{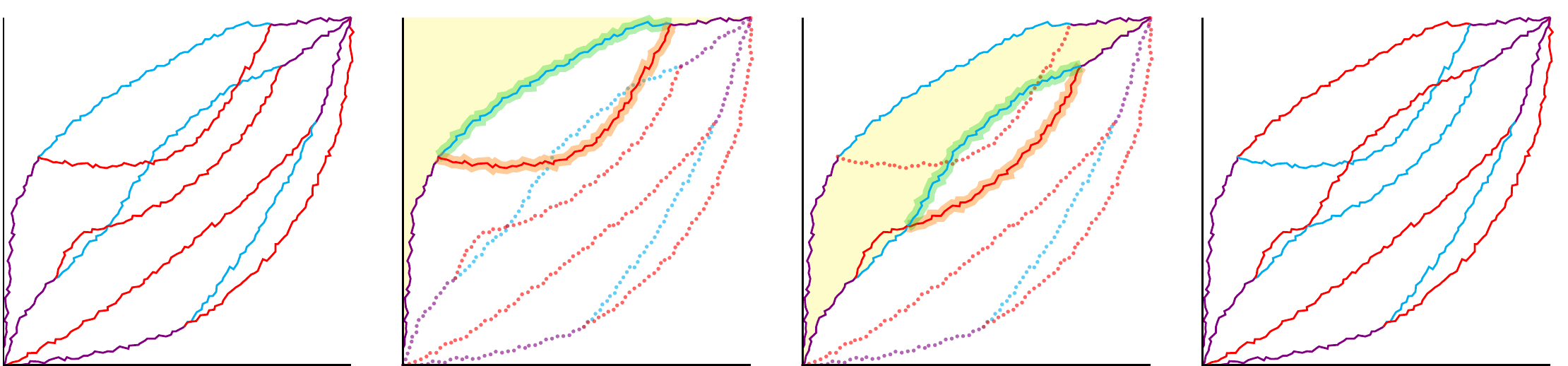}
 \caption{A depiction of the argument for Proposition~\ref{p.stronger interlacing}. The red curves comprise the supposed unique 4-geodesic watermelon $\phi_n^4$ and the blue curves comprise the unique 3-geodesic watermelon $\phi_n^3$; vertices which are covered by both types of curves are coloured violet. Note that the left two red curves are not in their home regions (the home regions of $\phi_{n,\shortrightarrow}^{4,1}$ and $\phi_{n,\shortrightarrow}^{4,2}$ are shaded in light yellow in the second and third panels respectively), and so $\phi_n^4$ is not interlacing with $\phi_n^3$; we look for a contradiction. In the middle two panels, the excursions of $\phi_n^4$ in their wilderness regions are highlighted in orange, and the border subsets of $\phi_n^3$ they will be swapped with are highlighted in green. In the final panel we have swapped each orange excursion with the corresponding green path to result in two new collections of three or four disjoint paths, $\gamma_n^3$ and $\gamma_n^4$ (which do not necessarily interlace). But $\smash{\sum_{j=3,4}\ell(\gamma_n^j) = \sum_{j=3,4}\ell(\phi_n^j)}$ since the same vertices are covered the same number of times; thus the weight maximizing property of $\phi_n^3$ and $\phi_n^4$ implies that $\ell(\gamma_n^j) = \ell(\phi_n^j)$ for $j=3,4$, contradicting the uniqueness of $\phi_n^3$ and $\phi_n^4$.}
 \label{f.interlacing proof}
\end{figure}

\begin{lemma}\label{l.existence of good path}
For every $\gamma\in \WPF{k}{j}$, there exists a path $C_j(\gamma)$, which we call a \emph{candidate} path, such that $C_j(\gamma)\subseteq \pordmel{k-1}{j-1}$ (if $\gamma$ is to the left of $\pordmel{k-1}{j-1}$) or $C_j(\gamma)\subseteq \pordmel{k-1}{j}$ (if $\gamma$ is to the right of $\pordmel{k-1}{j}$). Further, $(\pordmel{k}{j}\cup C_j(\gamma))\setminus \gamma$ and $(\pordmel{k-1}{i} \cup\gamma)\setminus C_j(\gamma)$ are upright paths, where $i\in \{j-1,j\}$ is chosen so that $C_j(\gamma)\subseteq \pordmel{k-1}{i}$; that is, the sets obtained by swapping the wilderness excursion with the candidate path are upright paths.
\end{lemma}

\begin{proof}
Let $\gamma\in\WPF{k}{j}$ be a given wilderness excursion, and let $u$ be its startpoint and $v$ be its endpoint. Note by planarity; the ordering of the paths $\{\smash{\pordmel{k-1}{j}}:j\in\intint{k-1}\}$; and the starting and ending points of the curves of $\phi_n^k$ and $\phi_n^{k-1}$ guaranteed by Proposition~\ref{p.starting and ending points} that both $u$ and $v$ are adjacent to  $\pordmel{k-1}{j}$ or both are adjacent to $\pordmel{k-1}{j-1}$.

Suppose both are adjacent to $\pordmel{k-1}{j}$; the other case is argued in the same way. Then $\gamma$ must be to the right of $\pordmel{k-1}{j}$. The points on $\pordmel{k-1}{j}$ adjacent to $u$ and $v$ must also lie on $\pordmel{k}{j}$; call these points $\underline u$ and $\overline v$. Then define $C_j(\gamma)\subseteq \pordmel{k-1}{j}$ to be the portion of $\pordmel{k-1}{j}$ lying between $\underline u$ and $\overline v$, excluding these two points. It is easy to see that this path satisfies the claims made in Lemma~\ref{l.existence of good path}. \end{proof}

Having identified which subsets of vertices to swap, we define the resulting new set of $k$ curves. For each $j\in\intint{k}$, define
$$\tordmel{k}{j} = \pordmel{k}{j}\cup \bigg(\bigcup_{\gamma\in \WPF{k}{j}} C_j(\gamma)\bigg)\setminus \bigg(\bigcup_{\gamma\in \WPF{k}{j}}\gamma\bigg).$$
This is an upright path by Lemma~\ref{l.existence of good path}. We will define an analogous collection of $k-1$ curves by considering the other half of the swap after the next lemma, which asserts that these just defined $k$ curves are in fact disjoint.

\begin{lemma}\label{l.new curves are disjoint}
$\{\tordmel{k}{j} : j\in\intint{k}\}$ is a disjoint collection of $k$ curves.
\end{lemma}

\begin{proof}
Suppose to the contrary that, for some $j_1, j_2 \in \intint{k}$ with $j_1\neq j_2$, there is a vertex $v$ such that $v\in\tordmel{k}{j_1}\cap\tordmel{k}{j_2}$. Since $\pordmel{k}{j_1}\cap\pordmel{k}{j_2}=\emptyset$, there are two remaining cases:
\begin{enumerate}
	\item there exist $\gamma_1\in\WPF{k}{j_1}$ and $\gamma_2\in\WPF{k}{j_2}$ such that $v\in C_{j_1}(\gamma_1)\cap C_{j_2}(\gamma_2)$; or

	\item there exists $\gamma\in\WPF{k}{j_1}$ such that $v\in C_{j_1}(\gamma)\cap \pordmel{k}{j_2}$ and $v\not\in \bigcup\{\gamma':\gamma'\in\WPF{k}{j_2}\}$ (without loss of generality).
\end{enumerate}
We note immediately that Case 1 cannot occur as $C_{j_1}(\gamma_1)$ and $C_{j_2}(\gamma_2)$ are subsets of distinct curves of $\gamma_n^{k-1}$, which are disjoint. We next show that Case 2 leads to a contradiction.

There are again two cases, depending on whether $C_{j_1}(\gamma)$ is a subset of $\pordmel{k-1}{j_1}$ or of $\pordmel{k-1}{j_1-1}$. We handle the first as the argument for the second is similar. In that case, by planarity, $\gamma\subseteq \pordmel{k}{j_1}$ is to the right of $\pordmel{k-1}{j_1}$. Since $v\in C_{j_1}(\gamma)\cap \pordmel{k}{j_2}$, it follows that $j_2\leq j_1-1$; for planarity and the ordering of $\{\pordmel{k}{j} : j\in\intint{k}\}$ precludes $j_2\geq j_1+1$. But this implies that $v\in W^k_{j_2} = \intint{n}^2\setminus R^{k-1}_{j_2-1}$ since $v\in \pordmel{k-1}{j_1}\subseteq R^{k-1}_{j_1}$ and $R^{k-1}_{j_1} \cap R^{k-1}_{j_2-1} = \emptyset$ when $j_2 \leq j_1-1$ by \eqref{e.disjoint region condition}.

We have arrived at a contradiction since this implies that $v\in\cup\{\gamma':\gamma'\in\WPF{k}{j_2}\}$.
\end{proof}

We now define the new set of $k-1$ disjoint curves by doing the swap in the other way: for each $j\in\intint{k-1}$, define
$$\tordmel{k-1}{j} = \pordmel{k-1}{j}\cup \bigg(\bigcup_{i=j}^{j+1}\bigcup_{\substack{\gamma\in \WPF{k}{i}:\\ C_i(\gamma)\subseteq \phi_{n,\shortrightarrow}^{k-1, j}}} \gamma\bigg)
\setminus \bigg(\bigcup_{i=j}^{j+1}\bigcup_{\substack{\gamma\in \WPF{k}{i}:\\ C_i(\gamma)\subseteq \phi_{n,\shortrightarrow}^{k-1, j}}} C_i(\gamma)\bigg).$$
From Lemma~\ref{l.existence of good path}, we see that each $\tordmel{k-1}{j}$ is an upright path. We also need the analogue of Lemma~\ref{l.new curves are disjoint}, that $\smash{\{\tordmel{k-1}{j} : j\in\intint{k-1}\}}$ is a collection of $k-1$ disjoint curves. The proof bears some similarities to that of Lemma~\ref{l.new curves are disjoint} but we provide it here for completeness. For convenience, we first define
$$
A_j \, = \, \bigcup_{i=j}^{j+1}\bigcup_{\substack{\gamma\in \WPF{k}{i}:\\ C_i(\gamma)\subseteq \phi_{n,\shortrightarrow}^{k-1, j}}} \gamma \, .
$$

\begin{lemma}\label{l.new k-1 curves are disjoint}
$\{\tordmel{k-1}{j} : j\in\intint{k-1}\}$ is a disjoint collection of $k-1$ curves.
\end{lemma}

\begin{proof}
Suppose to the contrary that, for some $j_1, j_2 \in \intint{k-1}$ with $j_1\neq j_2$, there is a vertex $v$ such that $v\in\tordmel{k-1}{j_1}\cap\tordmel{k-1}{j_2}$. Since $\pordmel{k-1}{j_1}\cap\pordmel{k-1}{j_2}=\emptyset$, there are two remaining cases:
\begin{enumerate}
	\item $v\in A_{j_1}\cap A_{j_2}$; or

	\item $v\in \pordmel{k-1}{j_1}\cap A_{j_2}$ and $v\not\in \bigcup\big\{C_i(\gamma) : i\in\llbracket j_1, j_1+1\rrbracket, \gamma\in \WPF{k}{i}, C_i(\gamma)\subseteq \pordmel{k-1}{j_1}\big\}$ (without loss of generality). 
\end{enumerate}
We immediately note that Case 1 cannot occur since any $\gamma\in \cup_{i=1}^j\WPF{k}{i}$ can belong to at most one of the $A_j$, and distinct elements of $\smash{\cup_{i=1}^j}\WPF{k}{i}$ are disjoint. We next show that Case 2 leads to a contradiction.

Since $v\in A_{j_2}$, there are two possibilities: there exists $\gamma\in\WPF{k}{i}$ such that $v\in\gamma$ and $C_i(\gamma)\subseteq \pordmel{k-1}{j_2}$ for $i=j_2$ or $j_2+1$. We deal with the case that $i=j_2$; the argument for the other is similar.

That $C_{j_2}(\gamma)\subseteq \pordmel{k-1}{j_2}$ implies that $\gamma$ is to the right of $\pordmel{k-1}{j_2}$ from Lemma~\ref{l.existence of good path}. Thus $v\in\pordmel{k-1}{j_1}$ is to the right of $\pordmel{k-1}{j_2}$, yielding that $j_1\geq j_2+1$ since $j_1\neq j_2$.

Now $\pordmel{k}{j_2+1}$ is to the right of $\pordmel{k}{j_2}$, and so also of $\gamma$. Since $v\in\gamma$ lies on $\pordmel{k-1}{j_1}$, which is the left boundary of $R^{k-1}_{j_1}$, we see that $\pordmel{k}{j_2+1}$ intersects $R^{k-1}_{j}$ for some $j\geq j_1 \geq j_2+1$, which for all such $j$ is to the right of the home region $R^{k-1}_{j_2}$ for $\pordmel{k}{j_2+1}$. This implies that $\pordmel{k}{j_1}$ also intersects its wilderness region, and that it is to the right of its home region, by the ordering of $\{\pordmel{k}{j} : j\in\intint{k}\}$ by considering the anti-diagonal line through $v$. But then by the definition of the candidate paths, there must exist $\tilde\gamma \in \WPF{k}{j_1}$ such that $v\in C_{j_1}(\tilde\gamma)$, which is a contradiction.
\end{proof}

Finally we record a sufficient condition for there to be no wilderness path fragments which we will use shortly; note that there being no wilderness path fragments is equivalent to $\phi_n^{k}$ and $\phi_n^{k-1}$ interlacing.

\begin{lemma}\label{l.condition for no wild fragments}
If $\cup_{j=1}^k\WPF{k}{j} \subseteq \phi_n^{k-1}$, then $\cup_{j=1}^k\WPF{k}{j} = \emptyset$.
\end{lemma}

The proof idea is to assume the contrary and to then consider the leftmost path $\pordmel{k}{j}$ of $\phi_n^k$ which has a wilderness excursion. This excursion must be a subset of some path of $\phi_n^{k-1}$. By considering an anti-diagonal line $x+y=\ell$ through a vertex of this wilderness excursion, we see that each path of $\phi_n^k$ to the right of $\pordmel{k}{j}$ must have a wilderness excursion which intersects this anti-diagonal line and which must be a subset of a distinct path of $\phi_n^{k-1}$. But this is a contradiction as there are not enough paths in $\phi_n^{k-1}$ to accommodate all these paths of $\phi_n^k$.

\begin{proof}[Proof of Lemma~\ref{l.condition for no wild fragments}]
Suppose $\cup_{j=1}^k\WPF{k}{j} \neq\emptyset$. Then there is a minimum $j_0$ such that $\pordmel{k}{j_0} \not\subseteq R^{k-1}_{j_0-1}$. Let $v\in \pordmel{k}{j_0}\setminus R^{k-1}_{j_0-1}$. That $\pordmel{k}{j_0}$ is a path implies that $\pordmel{k}{j_0}\subseteq \phi_n^{k-1}$, and therefore there is an $i_0\geq 1$ such that $\pordmel{k}{j_0}  = \pordmel{k-1}{j_0+i_0}$ on the unique anti-diagonal line $x+y= \ell$ that passes through $v$. The ordering of $\{\pordmel{k}{j} : j\in\intint{k}\}$ implies that, for each $i\geq 1$, $\pordmel{k}{j_0+i}$ equals a distinct curve of $\phi_n^{k-1}$, again on the line $x+y=\ell$, because the vertex of $\pordmel{k}{j_0+i}$ on this line cannot lie in $R^k_{j_0+i-1}$. This is a contradiction as there are $k-1-j_0-i_0$ curves of $\phi_n^{k-1}$ to the right of $\pordmel{k-1}{j+i_0}$ but $k-j_0-1$ to the right of $\pordmel{k}{j_0}$.
\end{proof}

With the two new collections of $k$ and $k-1$ disjoint curves, we may implement the last step of the outline of the proof of Proposition~\ref{p.stronger interlacing} by showing that $\gamma_n^k$ and $\gamma_n^{k-1}$ are new geodesic watermelons, contradicting the uniqueness of geodesic watermelons as implied by the hypothesis that distinct sets have distinct weights.

\begin{proof}[Proof of Proposition~\ref{p.stronger interlacing}]
We first note that $\phi_n^k$ and $\phi_n^{k-1}$ interlacing is equivalent to $\cup_{j=1}^k \WPF{k}{j} = \emptyset$, and we will show the latter for $k\in\intint{n}$.

Fix $k\in\intint{n}$. The weight maximality property of $\phi_n^k$ and $\phi_n^{k-1}$, combined with Lemmas~\ref{l.new curves are disjoint} and \ref{l.new k-1 curves are disjoint}, implies that
\begin{equation}\label{e.comparing weights}
\ell(\gamma_n^{k}) \leq \ell(\phi_n^k) \quad\text{and}\quad
\ell(\gamma_n^{k-1}) \leq \ell(\phi_n^{k-1}).
\end{equation}
But in the specification of $\gamma_n^k$ and $\gamma_n^{k-1}$ compared to $\phi_n^k$ and $\phi_n^{k-1}$, we have only swapped vertices. Thus we have
$$\ell(\gamma_n^k) + \ell(\gamma_n^{k-1}) = \ell(\phi_n^k) + \ell(\phi_n^{k-1}).$$
Combining the last display with \eqref{e.comparing weights} implies that
$$\ell(\gamma_n^{k-1}) = \ell(\phi_n^{k-1})\quad\text{and}\quad \ell(\gamma_n^{k}) = \ell(\phi_n^{k}).$$
Then by the hypothesis that distinct sets have distinct weights, $\gamma_n^k = \phi_n^{k}$ and $\gamma_n^{k-1} = \phi_n^{k-1}$, where the equalities refer to the vertex sets. This can only happen if $\cup_{j=1}^k\WPF{k}{j} = \emptyset$. This is because $\gamma_n^{k-1} = \phi_n^{k-1}$ implies that $\cup_{j=1}^k\WPF{k}{j} \subseteq \phi_n^{k-1}$, which allows us to reach the conclusion of Lemma~\ref{l.condition for no wild fragments}. This completes the proof of Proposition~\ref{p.stronger interlacing}.
\end{proof}

Next we turn to proving Proposition~\ref{p.specified melon interlacing}, which handles the case where $\nu$ may have atoms.
Its proof proceeds by introducing a deterministic perturbation which makes the given $k$-geodesic watermelon $\phi_{n}^k$ the unique such melon; adding a second random perturbation to make all distinct subsets have distinct weights; invoking Proposition~\ref{p.stronger interlacing} to get interlacing; and showing that the $k$-geodesic watermelons in the perturbed environment are also $k$-geodesic watermelons in the original environment.

\begin{proof}[Proof of Proposition~\ref{p.specified melon interlacing}]
Consider the perturbed environment obtained from $\intint{n}^2$ as follows: each point $v$ in $\phi^k_{n}$ is deterministically given weight $\xi_v+\epsilon/4n^2$, where $\xi_v$ is the original weight of $v$ and $\epsilon>0$ will be chosen later; the remaining points in $\intint{n}^2$ have weight unchanged. We now further randomly perturb weights, by increasing the weight of each $v\in\intint{n}^2$ by an independent random variable distributed uniformly on $[0,\varepsilon/8n^2]$. Then it is clear that, for every $\epsilon>0$, $\phi^k_{n}$ is the vertices of the unique $k$-geodesic watermelon in the perturbed environment and that every distinct subset of vertices in this environment almost surely has distinct weight.

Applying Proposition~\ref{p.stronger interlacing} gives a sequence of interlacing geodesic watermelons $\{\widetilde\gamma_{n}^j : j\in\intint{n}\}$ (which is $\epsilon$-dependent) in the perturbed environment, and the uniqueness at level $k$ implies that the set of vertices of $\widetilde\gamma_{n}^k$ is the same as that of $\phi^k_{n}$. 

We set $\epsilon$ to be the minimum gap, over $j\in \intint{n}$, between the $j$-geodesic melon's weight and the weight of the best collection of $j$ disjoint paths which is not a $j$-geodesic melon, both in the unperturbed environment. It is clear that $\epsilon>0$.

We now argue that \smash{$\{\widetilde\gamma_{n}^j : j\in\intint{n}\}$}, which is a sequence of geodesic watermelons in the perturbed environment, must also be a sequence of geodesic watermelons in the original environment. If we let $\widetilde\ell(\cdot)$ denote the perturbed length, then it is easy to see that, for any subset of vertices $A\subseteq \intint{n}^2$,
\begin{equation}\label{e.perturbed unperturbed relation}
\widetilde\ell(A) - \frac{3\epsilon}{8} \leq \ell(A) \leq \widetilde\ell(A).
\end{equation}
Let $\gamma_{n}^j$ be a $j$-geodesic melon in the unperturbed environment and suppose by way of contradiction that for some $j\in\intint{n}$, $\widetilde\gamma_{n}^j$ is not a $j$-geodesic melon in the unperturbed environment. Then observing $\widetilde\ell(\gamma_{n}^j)<\widetilde\ell(\widetilde\gamma_{n}^j)$  and applying \eqref{e.perturbed unperturbed relation} with $A=\gamma_{n}^j$ for the first inequality and $A=\widetilde\gamma_{n}^j$ for the final inequality yields
$$\ell(\gamma_{n}^j) \leq \widetilde \ell(\gamma_{n}^j) <\widetilde \ell(\widetilde\gamma_{n}^j) \leq \ell(\widetilde\gamma_{n}^j) + \frac{3\epsilon}{8}.$$
But the definition of $\epsilon$ implies that this inequality does not hold. So $\widetilde\gamma_{n}^j$ must be a $j$-geodesic melon in the original environment for every $j\in\intint{n}$.
\end{proof}

\subsection{Proving Proposition~\ref{p:ordering}}
We must prove that the sequence $Y_{n,k} = X^k_n - X_{n}^{k-1}$ is non-decreasing in $k > 1$, for any given set of non-negative weights.
Note that $Y_{n,k} \leq Y_{n,k-1}$ is equivalent to $X_{n}^k + X_{n}^{k-2} \leq 2X_{n}^{k-1}$. To prove the latter, we will decompose the $2(k-1)$ curves of $\melon{k}$ and $\melon{k-2}$ into two sets of $k-1$ disjoint curves and then invoke the maximality of the  weight $X_{n}^{k-1}$. This is easy when $k=3$, as Figure~\ref{f.decomposition} illustrates.

The case of general $k$ will be proved by induction. We will establish a stronger inductive hypothesis, Lemma~\ref{l.k-1 sets decomposition}. We remind the reader of the abuse of notation we have adopted till now where a path is sometimes identified with the set of vertices in its range. It serves the purpose of stating and proving Lemma~\ref{l.k-1 sets decomposition} to adopt a further abuse, where we will sometimes regard a path as a \emph{multi}set of vertices in its range. A multiset in this instance is an $\N$-valued map on $\intint{n}^2$, i.e, a subset of $\intint{n}^2$ where the elements may have multiplicities. We use the symbol $\equiv$ to denote equality of two multisets. So, for example, it will hold for two paths $\gamma_1$ and $\gamma_2$ that $(\gamma_1\cup\gamma_2) \setminus \gamma_1 \equiv \gamma_2$, even if the two paths have portions of overlap.

\begin{figure}[h]
   \centering
        \includegraphics[width=.97\textwidth]{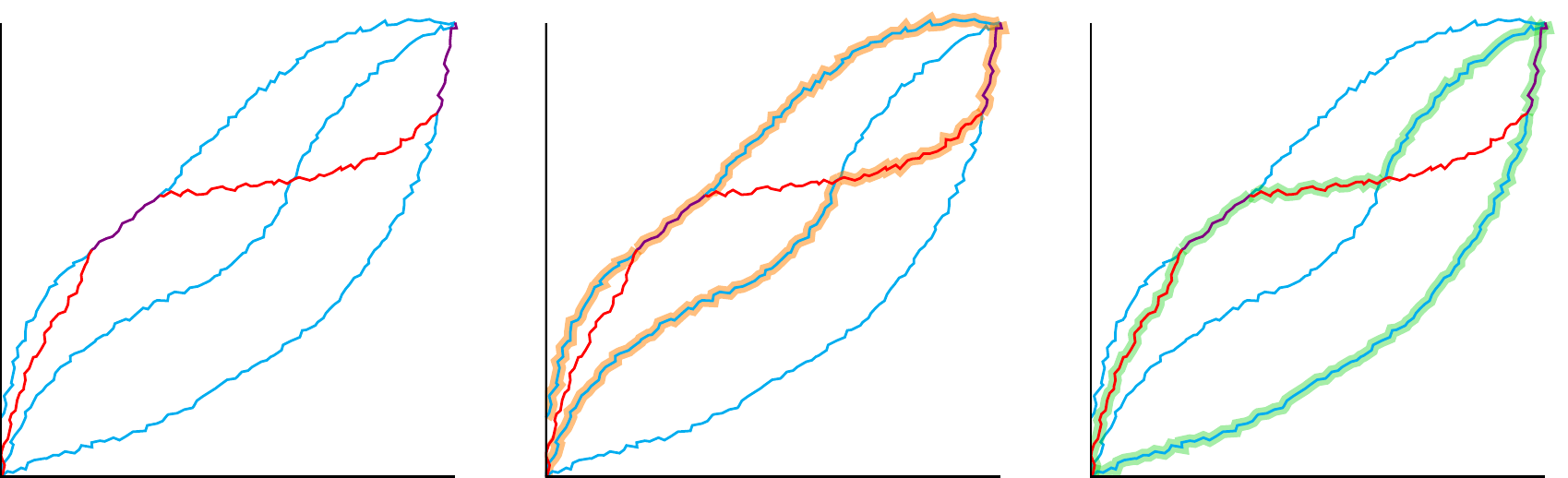}
 \caption{The base case of Lemma~\ref{l.k-1 sets decomposition}. In the first panel are the $\gamma$ curves in blue and $\gamma'$ curve in red (with overlap in purple), which in our application of Lemma~\ref{l.k-1 sets decomposition} will be the 3-melon and 1-melon. In the remaining two panels are highlighted how these four paths can be decomposed into a pair of two disjoint paths while respecting multiplicities. The rightmost orange path in the second panel is $\max(\gamma_1', \gamma_2)$.}
 \label{f.decomposition}
\end{figure}

\begin{lemma}\label{l.k-1 sets decomposition}
Suppose $\gamma_1 \preceq \ldots \preceq \gamma_k$ and $\gamma_1'\preceq \ldots \preceq \gamma_{k-2}'$ are such that all $2k-2$ curves have the same time-range (recall Definition~\ref{d.time range}) and no vertex is counted with multiplicity strictly greater than two in $\bigcup_{i=1}^k \gamma_i \cup \bigcup_{i=1}^{k-2}\gamma_i'$. Then there exist two collections $A_1$ and $A_2$ of $k-1$ disjoint curves each such that the curves have the same time range as before and
$$
A_1\cup A_2 \equiv \bigcup_{i=1}^k \gamma_i \cup \bigcup_{i=1}^{k-2}\gamma_i' \, ,
$$ %
with
$\gamma_1\in A_1$ and $\gamma_k\in A_2$.
\end{lemma}

Here is how we will prove Lemma~\ref{l.k-1 sets decomposition}: regard the union of the paths as a subset of $\intint{n}^2$, peel off the left boundary and label it as one path in $A_1$, and peel off the right boundary and label it as a path in $A_2$. Now peel off the \emph{right} boundary of the remaining set of vertices and label it as a path in $A_1$, and do likewise with the left boundary and $A_2$. Continue iteratively in this fashion, switching whether the left boundary is assigned to $A_1$ or $A_2$, till all the vertices have been depleted. This will yield the desired decomposition: see Figure~\ref{f.decomposition} to see that this is how the orange curves have been assigned to $A_1$ and the green curves to $A_2$.

To render a proof from this idea, we will need to define the max and min of two curves $\gamma$ and $\gamma'$ in order to specify precisely what we called peeling off the right and left boundaries. In each excursion between $\gamma$ and $\gamma'$, we will retain the right path fragment in the max, the left fragment in the min, and the overlapping portion in both. If the curves are disjoint, the right path is the max. In the notation introduced in Definition~\ref{d.time range}, for $t\in \mathcal R(\gamma)\cap \mathcal R(\gamma')$,
\begin{align*}
\max(\gamma, \gamma')(t) &= \max(\gamma(t), \gamma'(t))\\
\min(\gamma, \gamma')(t) &= \min(\gamma(t), \gamma'(t)),
\end{align*}
where the max and min on the right are with respect to the normal $\leq$ ordering. Note that $\min(\gamma,\gamma') \preceq \gamma,\gamma' \preceq \max(\gamma,\gamma')$.
We also record without proof the following simple statement.

\begin{lemma}\label{l.min max of paths}
Let $\gamma_1$, $\gamma_1'$, $\gamma_2$, and $\gamma_2'$ be upright paths with $\mathcal R(\gamma_1) = \mathcal R(\gamma_2)$, $\mathcal R(\gamma_1') = \mathcal R(\gamma_2')$, and $\gamma_i\preceq \gamma_i'$ for $i=1,2$. Then
\begin{enumerate}
	\item $\gamma_1\cup\gamma_2 \equiv \max(\gamma_1, \gamma_2)\cup \min(\gamma_1,\gamma_2)$;

	\item $\max(\gamma_1,\gamma_2) \preceq \max(\gamma_1',\gamma_2')$ and $\min(\gamma_1,\gamma_2) \preceq \min(\gamma_1',\gamma_2')$; and

	\item $\max(\gamma_1, \gamma_2)$ and $\min(\gamma_1, \gamma_2)$ are upright paths.
\end{enumerate}
\end{lemma}

\begin{proof}[Proof of Lemma~\ref{l.k-1 sets decomposition}]
The proof is by induction on $k$. We first verify the base case of $k=3$; see Figure~\ref{f.decomposition}. We take $A_1 = \{\gamma_1, \max(\gamma_1', \gamma_2)\}$ (highlighted orange in the second panel of Figure~\ref{f.decomposition}). %
 We take $A_2 = \{\min(\gamma_1', \gamma_2), \gamma_3\}$ (highlighted green in the third panel there). We now show that $A_1$ is a pair of disjoint curves; the argument of the same for $A_2$ is similar. We observe that
\begin{gather*}
\gamma_1 \preceq \gamma_2 \preceq \max(\gamma_1', \gamma_2) \quad\text{and}\quad
\gamma_1 \preceq \gamma_1' \preceq \max(\gamma_1', \gamma_2).
\end{gather*}
Thus any vertex in both $\gamma_1$ and $\max(\gamma_1', \gamma_2)$ must be a member of $\gamma_1,$ $\gamma_2$, and $\gamma_1'$, contradicting the multiplicity assumption.
This completes the case $k=3$. %

Now we move to the inductive step. The idea is similar to the base case's. Consider the set
$$
B \equiv \bigcup_{i=2}^{k-1} \gamma_i \cup \bigcup_{i=1}^{k-2}\gamma_i' \, ;
$$%
i.e., remove the left and right boundaries $\gamma_1$ and $\gamma_k$ from the original collection.
We claim that this set satisfies the inductive hypotheses for $k-1$. To see this, define $\widetilde \gamma_i = \min(\gamma_{i+1}, \gamma'_{i})$ for $i \in \intint{k-2}$ and $\widetilde \gamma_{k-1} = \max(\gamma_{k-1}, \gamma_{k-2}')$, and define $\widetilde \gamma_i' = \max(\gamma_{i+1},\gamma_i')$ for $i \in \intint{k-3}$.  Lemma~\ref{l.min max of paths} asserts that each of these is an upright path, and that their union is $B$.
It is also immediate from Lemma~\ref{l.min max of paths} that the ordering relations of $\widetilde\gamma_i$ and $\widetilde\gamma_i'$ are satisfied, and the multiplicity condition is inherited from the original $\gamma_i$ and $\gamma'_i$ curves. Thus $B$ satisfies the inductive hypotheses for $k-1$.

We therefore obtain two sets, $A_1'$ and $A_2'$, each with $k-2$ disjoint curves, such that
$$
\widetilde A_1\cup \widetilde A_2  \, \equiv \, \bigcup_{i=2}^{k-1} \gamma_i \cup \bigcup_{i=1}^{k-2}\gamma_i' \, ,
$$
$\min(\gamma_2,\gamma'_1) = \widetilde\gamma_1 \in \widetilde A_1$, and $\max(\gamma_{k-1}, \gamma_{k-2}') = \widetilde\gamma_{k-1} \in \widetilde A_2$. We now form $A_1$ and $A_2$ by defining
$$
A_1 = \widetilde A_2\cup\gamma_1 \quad\text{and}\quad A_2 = \widetilde A_1\cup\gamma_k \, .
$$
All we have to verify is that $\gamma_1$ is disjoint with the curves in $\widetilde A_2$ and the same for $\gamma_k$ and $\widetilde A_1$. We will show the first as the argument for the second is similar.

Suppose there is $\gamma\in \widetilde A_2$ such that there exists $v \in \gamma\cap \gamma_1$. Notice that
$$
\gamma_1 \preceq \min(\gamma_2,\gamma_1') \preceq \gamma \, .
$$
Thus if $v$ is an element of the left-hand side and right-hand side paths in this display, then it must also be an element of the middle path. But then $v$ must have multiplicity at least three in $\bigcup_{i=1}^k\gamma_i\cup\bigcup_{i=1}^{k-2}\gamma_i'$ as the three paths in the last display belong to three disjoint collections of paths, namely $\bigcup_{i=1}^k\gamma_i\cup\bigcup_{i=1}^{k-2}\gamma_i' \setminus (\widetilde A_1 \cup \widetilde A_2)$, $\widetilde A_1$, and $\widetilde A_2$ respectively. This is a contradiction, completing the induction step and the proof of Lemma~\ref{l.k-1 sets decomposition}.
\end{proof}

\begin{remark}
In fact, the hypotheses that $\{\gamma_i:i\in\intint{k}\}$ and $\{\gamma'_i : i\in\intint{k-2}\}$ are ordered; $\gamma_1\preceq \gamma_1'$; and $\gamma_{k-2}'\preceq \gamma_k$ are all not needed for the conclusion of Lemma~\ref{l.k-1 sets decomposition} to hold, with the modification that we require the left boundary of the collection to lie in $A_1$ and the right boundary in $A_2$, instead of $\gamma_1$ and $\gamma_k$. Without these hypotheses, the set $B$ in the proof above is taken to be the original collection with the left boundary and right boundary curves removed; that the remaining curves can be decomposed as needed to apply the induction hypothesis is proved by a separate induction argument. As the argument becomes more involved and we do not require this generality for our application, we have chosen to not make this more general statement.
\end{remark}

Now we use Lemma~\ref{l.k-1 sets decomposition} to complete the proof of Proposition~\ref{p:ordering}.

\begin{proof}[Proof of Proposition \ref{p:ordering}]
	That $Y_{n,1} \geq Y_{n,2}$ is equivalent to $2X_{n}^1 \geq X_n^2$, which is obvious from the definition of the two quantitites.

	To prove $Y_{n,k-1}\geq Y_{n,k}$ for $k\geq 3$, as we noted, it is sufficient to prove $X_n^{k-2} + X_n^k \leq 2X_n^{k-1}$. 
	Let $\gamma_n^k = \{\ordmel{k}{j}: j\in\intint{k}\}$ and $\gamma_n^{k-2} = \{\ordmel{k-2}{j}: j\in\intint{k-2}\}$ be $k$- and $(k-2)$-geodesic watermelons as given by Proposition~\ref{p.starting and ending points}. We extend them by horizontal line segments so that $\gamma_i$ is $\ordmel{k}{i}$ extended to begin at $(i-k+1, k-i+1)$ and end at $(n+i-1, n-i+1)$ for $i\in\intint{k}$, and $\gamma'_i$ is $\ordmel{k-2}{i}$ extended to begin at $(i-k+3, k-i-1)$ and end at $(n+i-1, n-i+1)$ for $i\in\intint{k-2}$; this is done so that all the paths have the same time-range.

	We apply Lemma~\ref{l.k-1 sets decomposition} to these curves;
	the hypothesis on the ordering of the curves and the multiplicity assumption is satisfied by the disjointness of the curves in geodesic watermelons and an invocation of interlacing (Proposition~\ref{p.stronger interlacing}). Then we obtain two collections of $k-1$ disjoint curves $A_1$ and $A_2$. We intersect these curves with $\intint{n}^2$ to get $A_1'$ and $A_2'$, which are still two collections of $k-1$ disjoint curves and whose union is the same as the union of $\gamma_n^k$ and $\gamma_n^{k-2}$, regarded as multisets. The maximum weight property of $(k-1)$-geodesic watermelons completes the proof by implying that
	\begin{equation*}
	X_{n}^{k-2} + X_{n}^{k} = \ell(A_1') + \ell(A_2') \leq 2X_{n}^{k-1}. \qedhere
	\end{equation*}
\end{proof}

\section{Connections to determinantal point processes and eigenvalue rigidity}
\label{s:special}

We begin this section by showing how the just proved deterministic monotonicity of $\{Y_{n,i}: i\in\intint{n}\}$ can be used, somewhat surprisingly, to quickly obtain a  concentration of measure result (which we expect to be optimal) for the same point process.

\begin{proposition}
\label{p:kthweight} 
Suppose that Assumptions \ref{a.passage time continuity}, \ref{a.limit shape assumption}, and \ref{a.one point assumption} (resp. \ref{a.one point assumption convex}) hold.
There exist positive constants $C_1$, $C_2$ and $c_1,$
such that, for some $c>0$, and for  $1\leq k\leq c_1n^{1/2}$ (resp. $1\leq k \leq c_1 n$),
$$\P\left(Y_{n,k} - \mu n \notin - k^{2/3}n^{2/3} \cdot (C_1 ,  C_2 )\right)\leq e^{-ck^2} \, . $$
\end{proposition}

\begin{proof}%
By Proposition \ref{p:ordering}, $\big\{ Y_{n,i}: i \in \intint{n}\big\}$ are ordered. Thus, taking $C_1$ as in Theorem~\ref{t.notwidenotthingeneral}(\ref{weight1'}), 
	\begin{align*}
	\P\left(Y_{n,k} < \mu n - 4C_1k^{2/3}n^{1/3}\right)
	&\leq \P\left(\sum_{i=k}^{2k-1} Y_{n,i} < \mu nk - 4C_1k^{5/3}n^{1/3}\right).
	\end{align*}
	The latter quantity is bounded by
	$$
	\P\left(\sum_{i=1}^{2k-1} Y_{n,i}  < \mu n(2k-1) - 2^{5/3}C_1k^{5/3}n^{1/2}\right) + \P\left(\sum_{i=1}^{k-1} Y_{n,i}  > \mu n(k-1) + (4-2^{5/3})C_1k^{5/3}n^{1/3}\right) \, .
	$$
	Note that $\sum_{i=1}^{2k-1}Y_{n,i} = X_n^{2k-1}$ and $\sum_{i=1}^{k-1}Y_{n,i} = X_n^{k-1}$. Thus the preceding display is equal to
	$$\P\left(X_{n}^{2k-1}  < \mu n(2k-1) - C_1(2k)^{5/3}n^{1/3}\right)
	+ \P\left( X_{n}^{k-1}  > \mu n(k-1) + (4-2^{5/3})C_1k^{5/3}n^{1/3}\right);$$
	the first term is bounded by $e^{-ck^2}$ by the construction in Theorem~\ref{t.flexible construction}, and the second term by the crude upper bound in Lemma~\ref{l.first upper bound on melon weight} (noting that $4-2^{5/3}>0$). Both these inferences hold for ranges of $k$ and $n$ depending on whether Assumption~\ref{a.one point assumption} or \ref{a.one point assumption convex} is in force, as in the statement of Proposition~\ref{p:kthweight}.

	Now we turn to the other bound. Since the sequence $\big\{ Y_{n,i}: i \in \intint{n}\big\}$ is ordered,
	\begin{align*}
	\P\left(Y_{n,k} > \mu n - C_2k^{2/3}n^{1/3}\right) \leq \P\left(\sum_{i=1}^{k} Y_{n,i} > \mu n k- C_2k^{5/3}n^{1/2}\right)
	&= \P\left(\melonweight > \mu n k- C_2k^{5/3}n^{1/2}\right)\\
	&\leq e^{-ck^2},
	\end{align*}
	where Theorem~\ref{t.notwidenotthingeneral}(\ref{weight1'}) was used in the last line, with $C_2$ as in the latter, and where  $k$ and~$n$ are again supposed to belong to certain ranges depending on whether Assumption~\ref{a.one point assumption} or \ref{a.one point assumption convex} is in force. 
\end{proof}

Integrable methods also yield a lot of information about the process $\{Y_{n,i}: i\in\intint{n}\}$. For instance, in exponential LPP,  this process has the distribution of $\lambda_1 \geq  \cdots \geq \lambda_{n},$ the eigenvalues of the LUE \cite{adler-eig-perc-connection}, i.e., the Hermitian matrix $X^*X$ where $X$ is an $n\times n$ matrix of i.i.d.\ standard complex Gaussian entries. This can be used to show that the edge, given by  $\big\{ Y_{n,i}: i \in \intint{k} \big\}$ with $k$ fixed,   converges weakly in the limit of high~$n$ to the $k$ uppermost elements in the Airy point process, after appropriate centering and scaling.  Determinantal techniques prove the following weaker counterpart to Proposition \ref{p:kthweight}.
\begin{submitted-version}
\begin{equation}\label{e.eigenvalue concentration}
\P\left (\lambda_{k}\notin (4n-C_5k^{2/3}n^{1/3}, 4n-C_6k^{2/3}n^{1/3})\right)\leq e^{-ck}.
\end{equation}
\end{submitted-version}
\begin{arxiv-version}
\begin{proposition}
\label{p:kthweight2}
There exist $C_5>C_6>0,c>0$ and $k_0\in \N$ such that for all $k\geq k_0$ and all $n\geq n_0(k)$ we have 
$$\P\left (\lambda_{k}\notin (4n-C_5k^{2/3}n^{1/3}, 4n-C_6k^{2/3}n^{1/3})\right)\leq e^{-ck}.$$
\end{proposition}
\end{arxiv-version}
\begin{arxiv-version} 
\end{arxiv-version}\begin{submitted-version} 
This is proved in the version of this article on arXiv.\end{submitted-version}
The basic argument for 
\begin{arxiv-version}%
 Proposition~\ref{p:kthweight2}%
\end{arxiv-version}\begin{submitted-version}
\eqref{e.eigenvalue concentration}
\end{submitted-version}
uses the representation of the number of points of a determinantal point process in a given interval as the sum of independent Bernoulli random variables \cite[Theorem~4.5.3]{manjunath} and then applies standard concentration inequalities for such sums. However, we also need a sharp estimate on the mean of this sum of Bernoulli random variables, which is available in the exponential LPP case from the literature on the LUE. The corresponding estimate does not seem to be available for geometric LPP, where the relevant determinantal process is the less-studied Meixner ensemble.

\begin{arxiv-version}
\begin{proof}[Proof of Proposition~\ref{p:kthweight2}]
	We first prove that
	$$\P(\lambda_{k}< 4n-C_5k^{2/3}n^{1/3}) \leq e^{-ck}.$$
	The argument for the other side is analogous; we shall point out the steps at the end.
	Let us denote the scaled eigenvalues $\lambda_i/n$ by $\tilde \lambda_i$. Thus it suffices to prove that
	\begin{equation}
	\P(\tilde \lambda_{k}< 4-C_5(k/n)^{2/3}) \leq e^{-ck}.
	\end{equation}
	Let $I_k = [4-C_5k^{2/3}n^{-2/3}, \infty)$. The event that $\tilde \lambda_k < 4-C_5k^{2/3}n^{-2/3}$ is the same as $N_{I_k} < k$, where $N_I$ is the number of scaled eigenvalues lying in the interval $I$. So it suffices to prove
	$$\P(N_{I_{k}}< k) \leq e^{-ck}.$$
	The eigenvalues of the LUE form a determinantal point process, and for such point processes, for any interval $I$, $N_I$ can be expressed as the sum of independent Bernoulli random variables \cite[Theorem 4.5.3]{manjunath}. Following an argument in \cite{dallaporta}, we combine this fact with a Bernstein-type inequality to arrive at
	$$\P\left(|N_{I_{k}}-\E N_{I_k}|> t\right) \leq 2\exp\left(-\frac{t^2}{2\sigma_k^2 + t}\right),$$
	where $\sigma_k^2 = \Var(N_{I_k}).$ Since $N_{I_k}$ is a sum of Bernoulli random variables, $\sigma_k^2 \leq \E[N_{I_k}]$, so we have
	\begin{equation}\label{e.bernstein}
	\P\left(|N_{I_{k}}-\E N_{I_k}|> t\right) \leq 2\exp\left(-\frac{t^2}{2\E[N_{I_k}] + t}\right).
	\end{equation}
	(We mention in passing that the empirical distribution of $\tilde\lambda_i$ converges to the Marchenko-Pastur distribution: see \cite{marvcenko1967distribution}.)
	What we require is a rate for this convergence of the empirical dsitribution of $\tilde\lambda_i$ to the Marchenko-Pastur distribution; this is given by \cite[Theorem 1.4]{gotze2005rate}, which states that that there exists an absolute constant $C_7 > 0$ independent of $n, k, C_5$ such that

	$$\left|\E[N_{I_k}] - n\int^{4}_{4-C_5k^{2/3}n^{-2/3}}\mu_{\mathrm{MP}}(x)\,dx\right| < C_7,$$
	where the Marchenko-Pastur density function $\mu_{\mathrm{MP}}(x)$ equals $\frac{1}{2\pi x}\sqrt{x(4-x)}$ on $[0,4]$. Straightforward calculus yields that for $0<y<1/5,$ and for $0<b_1<b_2$,%
	$${b_1}y^{3/2}\leq \int_{4-y}^4 \frac{1}{2\pi x}\sqrt{x(4-x)}\, dx \leq b_2 y^{3/2}.$$
	Now putting $y = C_5k^{2/3}n^{-2/3}$, $k>C_7$, we find that, for a large enough $C_5$ and $n\geq C_8k$ (as $y<1/5$ is needed for the last estimate),
	$$\E[N_{I_k}] \in  [R_1k, R_2k],$$
	for some numbers $R_1$ and $R_2$ (although one can precisely compute the number up to smaller order terms by exactly evaluating the integral). %
	Now, manipulating, and using \eqref{e.bernstein},
	\begin{align*}
	\P(N_{I_k} < k) = \P\big(N_{I_k} -\E[N_{I_k}] < k-\E[N_{I_k}]\big)
	&\leq \P\big(|N_{I_k} -\E[N_{I_k}]|  >k\big) \\
	&\leq 2\exp\left(-\frac{k^2}{2\E[N_{I_k}] + k}\right)
	\end{align*}
	for large enough $k$ and $n > k$. This last quantity is bounded by $2e^{-ck}$, as required.%

	For proving $\P(\lambda_{k}> 4n-C_6k^{2/3}n^{1/3}) \leq e^{-ck}$, we similarly define the interval $\widetilde I_k = [4-C_6k^{2/3}n^{-2/3}, \infty)$ %
	and note that this probability is the same as $\P(N_{\widetilde I_k}>k)$. We follow the same steps as before, writing $N_{\widetilde I_k}$ as a sum of independent Bernoulli random variables and using a Bernstein-type inequality. The only difference is that, for this side, we use $\E[N_{\widetilde I_k}] < k/2$ for some $C_6$, which follows using the same estimate of $\smash{\E[N_{\widetilde I_k}]}$ as before for $k>4C_7$. The upper bound with these choices is $\smash{2e^{-k/6}}$. We omit the remaining details.
\end{proof}

We wish to point out that while straightforward applications of tools from determinantal point processes, as above, lead to quantitatively weaker concentration estimates for $Y_{n,k}$ even in case of the integrable model of exponential LPP, it may be possible to obtain such estimates by using more refined techniques from random matrix theory and determinantal point processes. We have not explored this direction.

\end{arxiv-version}

\subsection{Another route to Theorems \ref{t.notwidenothin} and \ref{t.notwidenotthingeneral} using the lower bound of $Y_{n,k}$}
The proof of Theorem~\ref{t.notwidenotthingeneral} that we have presented entails obtaining a high probability lower bound on $X_n^{j,j}$ for a random $j$ by averaging \eqref{avecons}, and applying  interlacing to relate $j$ and $k$. Since $X_{n}^{k,k} \geq Y_{n,k}$ by \eqref{e.averaging inequality}, there is another route to Theorems \ref{t.notwidenothin} and\ref{t.notwidenotthingeneral}, which uses the lower bound of $Y_{n,k}$ provided by Propositions \ref{p:kthweight} and~\ref{p:kthweight2}, allowing the direct application of Theorems~\ref{t:disjoint} and \ref{t.tf} to prove Theorem~\ref{t.notwidenotthingeneral}(\ref{tf'}), bypassing averaging and interlacing. %
The upper bound on $X_{n}^k$ follows the argument given in Section~\ref{s:upper}, but again invokes the lower bound on $X_{n}^{k,k}$ directly. 
 
{Note that the lower bound on $Y_{n,k}$ in Proposition~\ref{p:kthweight} relies only on the construction in Theorem~\ref{t.flexible construction}, the upper bound in Lemma~\ref{l.first upper bound on melon weight}, and the key monotonicity property in Proposition \ref{p:ordering}.}

{
For exponential LPP, a more direct approach that yields weaker  bounds invokes Proposition \ref{p:kthweight2} to obtain the required lower bound on $Y_{n,k}$  and proceeds as in the preceding paragraph. The monotonicity result in this case is an immediate consequence of the correspondence with LUE \cite{adler-eig-perc-connection}. 
}

\section{Point-to-line LPP}
\label{s:p2l}

In this section we prove Theorem~\ref{t:p2lgeneral}. We start by noting that the proof of interlacing for $\Gamma_{n}^k$ applies verbatim to $\Upsilon_{n}^k$. \\

\noindent
{\bf{Proof of Theorem \ref{t:p2lgeneral}.}} We will first bound the weight fluctuations, proving the analogue of Theorem~\ref{t.notwidenotthingeneral}(\ref{weight1'}), and then provide the arguments to verify the versions of Theorem~\ref{t.notwidenotthingeneral}(\ref{tf'}) concerning transversal fluctuations.

\medskip
\noindent
\textbf{Weight fluctuations:}
The lower bound follows immediately from the weight lower bound of Theorem~\ref{t.notwidenotthingeneral}(\ref{weight1'}) as the point-to-point $k$-geodesic watermelon's weight is stochastically dominated by the point-to-line $k$-geodesic watermelon's weight. So we need to only prove the upper bound. 

In fact the proofs of all the ingredient lemmas and propositions that went into the proof of the weight upper bound of Theorem~\ref{t.notwidenotthingeneral}(\ref{weight1'}) apply verbatim to the point-to-line watermelon by replacing $X_{n}^k$ with $Z_{n}^k$ and $\underline X_{n}^{k,k}$ with $\underline Z_{n}^{k,j}$.
   The only ingredient for which this is not true is Theorem \ref{t.tf}, which must be replaced by its point-to-line version Theorem \ref{t.point to line weight loss}. %

We will now follow the steps of the proof of the weight upper bound of Theorem~\ref{t.notwidenotthingeneral}(\ref{weight1'}). We first observe the analogue of Lemma~\ref{l.first upper bound on melon weight} for the point-to-line watermelon weight:
\begin{equation}\label{e.p2l crude upper bound}
\P\left(Z_{n}^k > \mu nk + tk^{5/3}n^{1/3}\right) \leq e^{-ct^{3/2}k^2},
\end{equation}
for $k<t^{-3/4}n^{1/2}$.
This requires the one-point input for the point-to-line single geodesic weight, which is provided from Assumption~\ref{a.one point assumption} by Proposition \ref{p.p2l general upper tail} with $t=s=0$. With \eqref{e.p2l crude upper bound} and the weight lower bound of Theorem~\ref{t.notwidenotthingeneral}(\ref{weight1'}), we obtain the analogue of Lemma \ref{p.smallest curve weight lower bound}, which holds for large enough $n$, a $C<\infty$, and $k<C^{-3/4} n^{1/2}$:
\begin{equation}\label{e.p2l average smallest curve lower bound}
\P\left(\bigcup_{j =\lfloor \frac{k}{2} \rfloor}^{k} \left\{\underline Z_{n}^{j,j} > \mu n-Ck^{2/3}n^{1/3}\right\}\right) \geq 1-e^{-ck^2}.
\end{equation}
 Let $C_3>0$. Similarly to Lemma \ref{l.no interior packing}'s statement, let $\underline E_n^{j,*}(\delta)$ is the minimum number of curves of $\Upsilon_{n}^j$, over all $j$-geodesic point-to-line watermelons, which exit the strip $U_{n,\frac12\delta k^{1/3}n^{2/3}}$ with $\delta$ as given in Theorem \ref{t:disjoint}. (Note that Theorem~\ref{t:disjoint} applies equally well without any modification to its statement in this situation.)  Thus we obtain the analogue of Lemma~\ref{l.no interior packing}:
\begin{equation}\label{e.p2l no interior packing}
\P\left(\bigcup_{j=\lfloor k/2\rfloor}^k \left\{\underline E_n^{j,*}(\delta) > \frac{k}{4}\right\}\right) \geq 1- e^{-ck^2}.
\end{equation}
At this point the proof of the weight upper bound of Theorem~\ref{t.notwidenotthingeneral}(\ref{weight1'}) applies essentially verbatim, but we reproduce it here for the reader's benefit.

Let $B_1$ be the complement of the event whose probability is bounded below in \eqref{e.p2l no interior packing}, so
	$\P(B_1) \leq e^{-ck^2}.$
	On $B_1^c,$ by interlacing, at least $k/4$ of the curves of any $k$-geodesic point-to-line watermelon $\Upsilon_{n}^{k}$ must exit the strip of width $\frac12\delta k^{1/3}$.
	By Theorem \ref{t.point to line weight loss}, there is a $c'$ such that, if $B_2$ is defined as
	$$B_2 = \left\{\exists\, \Gamma : \ell(\Gamma)>\mu n-c'k^{2/3}n^{1/3}, \tf(\Gamma)>\frac12\delta k^{1/3}n^{2/3}\right\},$$
	where $\Gamma$ is a upright path from $(1,1)$ to the line $x+y =2n$, then
	\begin{equation} \label{e.p2l B_2 bound}
	\P(B_2)\leq e^{-ck}.
	\end{equation}
	Consider the events $A=\left\{Z_{n}^k > \mu nk - \frac{1}{16}c'k^{5/3}n^{1/3}\right\}$ and $B_3 = \{Z_{n}^{\lfloor7k/8\rfloor} > \mu n\lfloor7k/8\rfloor + \frac{1}{16}c' k^{5/3}n^{1/3}\}$. By \eqref{e.p2l crude upper bound},
	$\P(B_3)\leq e^{-ck^2}.$
	Fix some $k$-geodesic point-to-line watermelon $\Upsilon_n^k$, and let $Z_n^{k,j}$ be the weight of its $j$\textsuperscript{th} heaviest curve for each $j\in\intint{k}$. Now on $A\cap B_3^c$ we have 
	\begin{gather*}
	Z_{n}^{\lfloor7k/8\rfloor} + \left(Z_{n}^{k,\lfloor7k/8\rfloor+1} + \ldots + Z_{n}^{k,k}\right) \geq Z_{n}^{k} > \mu nk -\frac{1}{16}c'k^{5/3}n^{1/3}\\
	\implies Z_{n}^{k,\lfloor7k/8\rfloor + 1} + \ldots + Z_{n}^{k,k} > \frac18 \mu nk - \frac{1}{8}c' k^{5/3}n^{1/3} \, .
	\end{gather*}
	By definition, the $Z_{n}^{k,i}$ are ordered; thus, we learn that  $Z_{n}^{k,\lfloor7k/8\rfloor + 1}> \mu n - c'k^{2/3}n^{1/3}$. Again by the ordering, this bound applies to $Z_{n}^{k,1},\ldots, Z_{n}^{k,\lfloor7k/8\rfloor}$ as well. This means we have $\lfloor7k/8\rfloor$ disjoint curves $\Gamma$, each with $\ell(\Gamma) > \mu n - c'k^{2/3}n^{1/3}$. Thus, on $A\cap B_1^c\cap B_3^c$, we must have at least $\lfloor k/8\rfloor$ disjoint curves $\Gamma$, each satisfying $\tf(\Gamma)>\frac12\delta k^{1/3}n^{2/3}$ and $\weight(\Gamma)>\mu n - c'k^{2/3}n^{1/3}$.
	By the BK inequality Proposition~\ref{p.bk} and the bound on $\P(B_2)$ in \eqref{e.p2l B_2 bound}, the probability of this occurring is bounded by 
	$\exp\left(-ck\cdot k\right) = \exp\left(-ck^{2}\right).$
	Noting the bounds on $\P(B_1)$ and $\P(B_3)$ completes the proof.

\medskip
\noindent
\textbf{Transversal fluctuations:}
We first prove the exponent lower  bound, analogous of Theorem~\ref{t.notwidenotthingeneral}(\ref{tf2'}); namely, that there exist $C<\infty$, $c>0$ and $\delta>0$ such that, for $k>k_0$ and $n>Ck$,
\begin{equation}\label{e.p2l not thin}
\P\left(\underline\tf^*(n,k) < \delta k^{1/3}n^{2/3}\right) \leq e^{-ck^2}.
\end{equation}
The proof proceeds in exactly the same way as that of Theorem \ref{t.notwidenotthingeneral}(\ref{tf2'}). Let $C$ be as in \eqref{e.p2l average smallest curve lower bound}. Let $\delta=\delta(2^{2/3}C)$ as in Theorem \ref{t:disjoint} and let $\delta '= 2^{-1/3}\delta$. Let $A_{k}$ denote the event that there exist $\lfloor k/2\rfloor$ disjoint paths contained in $U_{n,\delta' k^{1/3}n^{2/3}}$, each of which has weight at least $\mu n- Ck^{2/3}n^{1/3}$. Further, let $B_{k}$ denote the event from \eqref{e.p2l average smallest curve lower bound}: 
$$B_{k}:=\bigcup_{j =\lfloor \frac{k}{2} \rfloor}^{k} \left\{\underline Z_{n}^{j,j} > \mu n-Ck^{2/3}n^{1/3}\right\}.$$
Clearly, by Theorem \ref{t:disjoint} and \eqref{e.p2l average smallest curve lower bound}, we have $\P(A_{k}^c \cap B_{k})\geq 1-e^{-ck^2}$ for some $c>0$. 

Observe next that, on $A_{k}^{c}\cap B_k$, there exists $j\in \{\frac{k}{2}, \ldots , k\}$ such that each of the curves of any $j$-geodesic point-to-line watermelon $\Upsilon_{n}^j$ has weight at least $\mu n- Ck^{2/3}n^{1/3}$, and hence some of them must exit $U_{n,\delta' k^{1/3}n^{2/3}}$. By the interlacing result for point-to-line watermelons, the same must be true for any $k$-geodesic point-to-line watermelon. This completes the proof of \eqref{e.p2l not thin} with $\delta$ in the statement replaced by $\delta'$.

Now we turn to the exponent upper bound, analogous to Theorem~\ref{t.notwidenotthingeneral}(\ref{tf1'})); namely, that there exist $C''<\infty, c>0$ and $k_0$ such that, for $k>k_0$ and $n>k$,
\begin{equation}\label{e.p2l not wide}
\P\left(\overline\tf^*(n,k) > C''k^{1/3}n^{2/3}\right) \leq e^{-ck},
\end{equation}
which immediately implies the upper exponent bound. This proof proceeds exactly as does Theorem~\ref{t.notwidenotthingeneral}(\ref{tf1'})'s.

Let $C'$ be as in \eqref{e.p2l average smallest curve lower bound} with $2k$ in place of $k$, and $B'_{k}$ denote the large probability event from there:
$$
B'_{k}:=\bigcup_{j =k}^{2k} \left\{\underline X_{n}^{j,j} > \mu n-C'k^{2/3}n^{1/3}\right\} \, .
$$
Let $A'_{k}=A'_{k}(C'')$ denote the event that there exists a path from $(1,1)$ to the line $x+y=2n$ that exits $U_{n,C''k^{1/3}n^{2/3}}$ and has weight at least $\mu n - C'k^{2/3}n^{1/3}$. Now choose $C''>0$ (possible by Theorem \ref{t.point to line weight loss}) such that 
$\P(A'_{k})\leq e^{-c''k},$ 
for some $c''>0$ for all $k$ and all $n$ sufficiently large. Clearly now it suffices to show that, on $B'_k \cap (A'_{k})^c$, no $k$-geodesic point-to-line watermelon exits $U_{n,C''k^{1/3}n^{2/3}}$. By the definition of $B'_{k}$, there exists $j\in \llbracket k,2k\rrbracket$ such that all paths of all $j$-geodesic point-to-line watermelons have weight at least $\mu n-C'k^{2/3}n^{1/3}$, and $(A'_{k})^c$ ensures that all these paths are contained in $U_{n,C''k^{1/3}n^{2/3}}$. The proof of \eqref{e.p2l not wide} is completed by invoking the interlacing of geodesic point-to-line watermelons. \qed

\bibliographystyle{alpha}
\bibliography{watermelon-combined}
\appendix
\section{Exponential \& geometric LPP satisfy the assumptions}\label{app.exp and geo satisfy assumptions}
In this appendix we cite the results which show that  exponential and geometric LPP satisfy Assumptions~\ref{a.passage time continuity}, \ref{a.limit shape assumption}, and \ref{a.one point assumption combined}. 

We start with the foundational result on the Tracy-Widom fluctuations of the maximal path weight, due to Johansson \cite{johansson2000shape}, and continue with moderate deviations estimates and expectation asymptotics for the maximal path weight. 

Recall that $X_{n,\lfloor hn\rfloor}$ is the last passage value from $(1,1)$ to $(n,\lfloor hn\rfloor)$.

\begin{theorem}\label{t.tracy-widom}
Suppose that the vertex weight law in LPP is exponential of mean one, and let $h>0$.  As $n\to\infty$,
$$
\frac{X_{n, \lfloor hn\rfloor} - (1+\sqrt h)^2n}{h^{-1/6}(1+\sqrt h)^{4/3}n^{1/3}}\stackrel{d}{\to} F_{\mathrm{TW}} \, .
$$
Suppose instead that the weight law is geometric with parameter $p\in (0,1)$. Then, as $n\to\infty$,
$$\frac{X_{n, \lfloor hn\rfloor} - \omega(h,p)n}{\sigma(h,p)n^{1/3}}\stackrel{d}{\to} F_{\mathrm{TW}},$$
where
$$\omega(h,p) = \frac{(1+\sqrt{hp})^2}{1-p}\quad\text{and}\quad \sigma(h,p)=\frac{p^{1/6}h^{-1/6}}{1-p}(\sqrt h + \sqrt p)^{2/3}(1+\sqrt{hp})^{2/3}.$$
Here $\stackrel{d}{\to}$ denotes convergence in distribution and $F_{\mathrm{TW}}$ is the GUE Tracy-Widom distribution.
\end{theorem}

\begin{proof}
These results appear in \cite{johansson2000shape}, the first as Theorem 1.6 there and the second as Theorem~1.2.
\end{proof}

The next three statements establish that exponential and geometric LPP satisfy Assumptions~\ref{a.limit shape assumption} and \ref{a.one point assumption}.

\begin{theorem}[Moderate deviation estimate]\label{t.mod-dev}
Consider exponential LPP. Fix $\psi>1$ and let $h\in[\psi^{-1},\psi]$. There exist $t_0=t_0(\psi), n_0 = n_0(\psi)$ and $c = c(\psi)>0$ such that, for $n>n_0$ and $t > t_0$,
\begin{align*}
\P\left(X_{n, \lfloor hn\rfloor} - (1+\sqrt h)^2 n > tn^{1/3}\right) \leq \exp\left(-c\min(t^{3/2}, tn^{1/3})\right) \, .
\end{align*}
For the lower tail, for $t>t_0$,
\begin{align*}
\P\left(X_{n, \lfloor hn\rfloor} - (1+\sqrt h)^2 < -tn^{1/3}\right) \leq \exp\left(-ct^3\right) \, . 
\end{align*}
Similarly, in geometric LPP of parameter $p\in(0,1)$, the above two displays hold with $\omega(h,p)$ (as in Theorem~\ref{t.tracy-widom}) in place of $(1+\sqrt h)^2$.
\end{theorem}

\begin{proof}
	For exponential LPP, the lower tail for $t>t_0$ and the upper tail for $t_0 < t < n^{2/3}$ are provided by Theorem 2 of \cite{ledoux2010}. For the remaining case of $t\geq n^{2/3}$ in the upper tail, see \cite{johansson2000shape}.

	For geometric LPP, the upper tail bound is proved in \cite{johansson2000shape} (combining Corollary~2.4 and equation (2.22) there), while the lower tail bound is implied by \cite[Theorem 1.1]{baik2001optimal}.
\end{proof}

\begin{theorem}[Expected point-to-point weight]\label{l.polymer expected value}
Fix $\psi>1$ and let $h \in [\psi^{-1},\psi]$. There exist $c_1 = c_1(\psi)>0$, $c_2=c_2(\psi)>0$, and  $n_0=n_0(\psi)$ such that, for $n>n_0$, in exponential LPP,
$$\frac{\E[X_{n, \lfloor hn\rfloor}] - (1+\sqrt{h})^2n}{h^{-1/6}(1+\sqrt{h})^{4/3}n^{1/3}} \in (-c_1,-c_2),$$
while for geometric LPP of parameter $p\in(0,1)$,
$$\frac{\E[X_{n, \lfloor hn\rfloor}] - \omega(h,p) n}{	\sigma(h,p)n^{1/3}} \in (-c_1,-c_2),$$%
with $\omega(h,p)$ and $\sigma(h,p)$ as in Theorem~\ref{t.tracy-widom}.
\end{theorem}

\begin{proof}
	This follows from Theorem \ref{t.tracy-widom} and Theorem \ref{t.mod-dev}, and that the GUE Tracy-Widom distribution has negative mean.
\end{proof}

Assumption~\ref{a.limit shape assumption} can be verified in exponential and geometric LPP by replacing $n$ by $n-xn^{2/3}$ and setting $h$ to be $(n+xn^{2/3})/(n-xn^{2/3})$ in Theorem~\ref{l.polymer expected value}. 

We now note why the GUE Tracy-Widom distribution has negative mean, by pulling together known results 
to which Ivan Corwin has drawn our attention.

\begin{lemma}[Negative mean of GUE Tracy-Widom]\label{l.negative second order of mean}
Let $X_{\mathrm{TW}}$ be distributed according to the GUE Tracy-Widom distribution. Then we have $\E[X_{\mathrm{TW}}] < 0$.
\end{lemma}

\begin{proof}
	Remark 1.15 and equation 1.29 of \cite{quastel2019flat} show that  $X_{\mathrm{BR}}$  strictly stochastically dominates $4^{1/3} X_{\mathrm{GUE}}$, where the law of $X_{\mathrm{BR}}$  is the Baik-Rains distribution (the cumulative distribution function of $X_{\mathrm{BR}}$ is labeled $F_{\mathrm{stat}}^0$ in \cite[Equation 1.29]{quastel2019flat}).
	Now \cite[Proposition~2.1]{baik2000limiting} asserts that $\E[X_{\mathrm{BR}}] = 0$, completing the proof.
\end{proof}

\section{Proofs of basic tools}\label{app.tail bounds}

This appendix contains some consequences of Assumptions~\ref{a.passage time continuity}, \ref{a.limit shape assumption}, and \ref{a.one point assumption} that were used in the main text.

\begin{itemize}
	\item Section~\ref{app.interval to interval} proves Proposition~\ref{l.sup tail}, the upper tail bound on interval-to-interval weights; 

\item Section~\ref{app.point-to-line} proves Proposition~\ref{p.p2l general upper tail}, concerning upper tails for point-to-line weights; and 

\item Section~\ref{app.constrained lower tail} proves Proposition~\ref{p.constrained lower tail}, concerning the lower tail and mean of the constrained point-to-point weights.

\end{itemize}

\subsection{Interval-to-interval estimates}\label{app.interval to interval}

\newcommand{\Elow}{E_{\mathrm{low}}}
\newcommand{\Eup}{E_{\mathrm{up}}}

\begin{proof}[Proof of Proposition~\ref{l.sup tail}]
	We start with the bound \eqref{e.sup tail away from axes} on $\widetilde Z$, and we treat only the case that $|z|\leq \rho r^{1/3}$; when $|z|>\rho r^{1/3}$, our argument works by setting $z=\rho r^{1/3}$ and using the concavity of the limit shape posited in Assumption~\ref{a.limit shape assumption}.

	By considering the event that $\sup_{(u,v)\in\mathcal S(U)} X_{u \ar v}$ is large and two events defined in terms of the environment outside of $U$, we find a point-to-point path which has large length. To define these events, first define points $\philower$ and $\phiupper$ on either side of $A$ and $B$:
	\begin{align*}
	\philower &:= \left(-\ell^{3/2}r, -\ell^{3/2}r\right)\\
	\phiupper &:= \left((1+\ell^{3/2})r-zr^{2/3}, (1+\ell^{3/2})r +zr^{2/3}\right).
	\end{align*}
	Let $u^*$ and $v^*$ be points on $A$ and $B$ where the suprema in the definition of $\widetilde Z$ are attained, and set
	\begin{align*}
	\Elow &= \left\{X_{\philower \ar  u^*-(0,1)}  > \mu \ell^{3/2}r - \frac{\theta\ell^{1/2}}{3}r^{1/3}\right\}\quad \text{and}\quad
	\Eup = \left\{X_{v^*+(0,1) \ar  \phiupper}  > \mu \ell^{3/2}r - \frac{\theta\ell^{1/2}}{3} r^{1/3}\right\}.
	\end{align*}
	Using Assumption~\ref{a.limit shape assumption} to bound the expectation in going from the second to the third line of the following, we find that
	\begin{align}\label{e.side-to-side decomposition}
	\P\bigg(\sup_{(u,v)\in\mathcal S(U)} &X_{u\ar v} > \mu r - G_2\frac{z^2r^{1/3}}{1+2\ell^{3/2}}+\theta\ell^{1/2}r^{1/3}, \Elow, \Eup\bigg)\\
	&\leq \P\left(X_{\philower\ar \phiupper} \geq \mu(1+2\ell^{3/2})r - G_2\frac{z^2r^{1/3}}{1+2\ell^{3/2}}+ \frac{\theta\ell^{1/2}}{3} r^{1/3}\right)\nonumber\\
	&\leq \P\left(X_{\philower\ar \phiupper} \geq \E\left[X_{\philower \ar \phiupper}\right] + \frac{\theta\ell^{1/2}}{3} r^{1/3}\right)\nonumber
	\leq \exp\left(-c\min(\theta^{3/2}, \theta r^{1/3})\right)\nonumber,
	\end{align}
	for $c>0$ independent of $\ell$. We used Assumption~\ref{a.limit shape assumption} for the former inequality of the final line, and Assumption~\ref{a.one point assumption} and the stationarity of the random environment
	for the latter.

	Let us denote conditioning on the environment $U$ by the notation $\P(\,\cdot \mid U)$. By this we mean we condition on the weights of vertices interior to $U$ as well as those on the lower and upper sides $A$ and $B$. Then we see
	\begin{align*}
	\MoveEqLeft[3.5]
	\P\bigg(\sup_{(u,v)\in\mathcal S(U)}  X_{u \ar v} > \mu r - G_2\frac{z^2r^{1/3}}{1+2\ell^{3/2}}+\theta\ell^{1/2}r^{1/3}, \Elow, \Eup \, \bigg\vert \, U\bigg)\\
	&= \P\left(\sup_{(u,v)\in\mathcal S(U)} X_{u \ar v} > \mu r - G_2\frac{z^2r^{1/3}}{1+2\ell^{3/2}}+\theta\ell^{1/2}r^{1/3} \, \bigg\vert \,  U\right)\cdot \P\left(\Elow\mid U\right)\cdot \P\left(\Eup \mid U\right).
	\end{align*}
	So with \eqref{e.side-to-side decomposition}, all we need is a lower bound on $\P\left(\Elow\mid U\right)$ and $\P\left(\Eup\mid U\right)$. This is straightforward using independence of the environment  between $U$ and the regions above and below it: 	%
	\begin{align*}
	\P\left(E^c_{\mathrm{lower}} \mid U\right) &\leq \sup_{u\in A} \,\P\left(X_{\philower \ar u-(0,1)} \leq \mu \ell^{3/2}r - \frac{\theta\ell^{1/2}}{3}r^{1/3}\right)
	\leq \frac12
	\end{align*}
	for large enough $\theta$ and $r$ (depending on $\ell$), using Assumption~\ref{a.one point assumption}. 
	A similar upper bound holds for $\P\left(E^c_{\mathrm{upper}}\mid U\right)$. Together this gives
	\begin{align*}
	\MoveEqLeft[10]
	\P\bigg(\sup_{(u,v)\in\mathcal S(U)}X_{u \ar v} > \mu r - G_2\frac{z^2r^{1/3}}{1+2\ell^{3/2}}+\theta\ell^{1/2}r^{1/3}, \Elow, \Eup  \, \bigg\vert \, U\bigg)\\
	&\geq \frac14\cdot \P\bigg(\sup_{(u,v)\in\mathcal S(U)} X_{u \ar v} >\mu r - G_2\frac{z^2r^{1/3}}{1+2\ell^{3/2}}+ \theta\ell^{1/2}r^{1/3}\bigg) \, ,
	\end{align*}
	and taking expectation on both sides, combined with \eqref{e.side-to-side decomposition}, gives the bound \eqref{e.sup tail away from axes} of Proposition~\ref{l.sup tail}.

	We now treat the bound \eqref{e.sup tail extreme} on $Z^{\mathrm{ext}, \delta}$. Observe that
	$$
	\left\{Z^{\mathrm{ext},\delta} > \theta\right\} \subseteq \bigcup_{u\in \linelower}\left(\left\{X_{u \ar (\delta r, 2r)} > \mu r + \theta r^{1/3}\right\} \cup \left\{X_{u \ar (2r, \delta r)} > \mu r + \theta r^{1/3}\right\}\right) \, .
	$$
	We bound the probability of $\bigcup_{u\in \linelower}\left\{X_{u \ar (\delta r, 2r)} > \mu r + \theta r^{1/3}\right\}$; the full bound then follows by a symmetric argument and a union bound.

	The point $(\delta r, 2r)$ can be written as $(\tilde r -\tilde z, \tilde r +\tilde z)$ with $\tilde r = (1+\delta/2)r$ and $\tilde z =(1-\delta/2)r$. By the definitions, we have that $\tilde z/\tilde r \geq \rho$ if $\delta$ is sufficiently small. So by the concavity guaranteed by Assumption~\ref{a.limit shape assumption} and by making $\delta$ small enough, we see that
	\begin{align*}
	\E\left[X_{u \ar (\delta r,2r)}\right] \leq \mu\tilde r - G_2 \rho^2\tilde r
	&\leq \mu r +\frac{\mu\delta}{2}r- G_2\rho^2\left(1+\frac{\delta}{2}\right)r
	\leq (\mu -\eta)r
	\end{align*}
	for some $\eta>0$. With such a value of $\delta$ fixed, note that $(\delta r, 2r)$ is such that we may apply Assumption~\ref{a.one point assumption}; we find by so doing that
	\begin{align*}
	\P\left(X_{u \ar (\delta r,2r)} \geq  \mu r+\theta r^{1/3}\right) &\leq \P\left(X_{u \ar (\delta r, 2r)} - \E[X_{u \ar (\delta r, 2r)}] \geq \eta r+\theta r^{1/3}\right)\\
	&\leq \exp\left(-cr-c\min(\theta^{3/2}, \theta r^{1/3})\right).
	\end{align*}
	Taking a union bound over the $O(r^{2/3})$ many $u$ in $\linelower$ and reducing the value of $c$ to absorb this factor into the exponent complete the proof of Proposition~\ref{l.sup tail}.
\end{proof}

\subsection{Point-to-line estimates} \label{app.point-to-line}

\begin{proof}[Proof of Proposition~\ref{p.p2l general upper tail}]
	Note that we may assume $s\leq r^{1/3}$. Let $\rho$ and $G_1$ be as in Assumption~\ref{a.limit shape assumption}. Let $\Gamma$ be the leftmost path with endpoint not on the line segment connecting $(r-(s+t)r^{2/3}, r+(s+t)r^{2/3})$ and $(r+(s+t)r^{2/3}, r-(s+t)r^{2/3})$ whose weight is $X$; this is well-defined by the weight maximization property of $\Gamma$.
	Let $\pleft$ be the coordinates of the starting point of $\Gamma$ on $\lineleft$, and $\pright$ be the coordinates of the endpoint of $\Gamma$ on the line $x+y=2r$. Let $y_j = (s+t+j)r^{2/3}$, $\mathbb L_j$ be the line segment joining the points
	$$\left(r - y_j, r + y_j\right) \quad \text{and} \quad  \left(r - y_{j+1}, r + y_{j+1}\right),$$
	and let $A_j$ be the event that $\pright\in \L_j$ and $(\pleft, \pright)$ are such that the slope of the line connecting them is not extreme enough to apply the second part of Proposition~\ref{l.sup tail}, i.e., it holds that $|(\pleft -\pright)_y|/(\pleft+\pright)_x \leq 1-\delta/2$, for $j=0,\ldots, r^{1/3}$, with $\delta$ as in that proposition's second part. Finally, let $A$ be the event that $(\pleft,\pright)$ satisfy $|(\pleft -\pright)_y|/(\pleft+\pright)_x>1-\delta/2$. Then clearly the whole probability space equals
	$$
	\bigcup_{j=0}^{r^{1/3}}A_j \cup A \, .
	$$
	Thus we have
	\begin{equation}\label{e.p2l tf breakup}
	\begin{split}
	\P\Big(X -\mu r >  \theta r^{1/3}-0.5c_1s^2r^{1/3}\Big) \, \leq \, \sum_{j=0}^{r^{1/3}}\P\Big(&X > \mu r + \theta r^{1/3}-0.5c_1s^2r^{1/3}, A_j\Big)\\
		&+ \P\left(X > \mu r + \theta r^{1/3}-0.5c_1s^2r^{1/3}, A\right).
	\end{split}
	\end{equation}
	We will bound the two terms using the next two lemmas.

\begin{lemma}\label{l.bound on A_j}
In the notation of Proposition \ref{p.p2l general upper tail} and under Assumptions \ref{a.limit shape assumption} and \ref{a.one point assumption}, there exist $c>0$, $c_1>0$, and $r_0$ such that, for $r>r_0$ and $j=0,1,\ldots, r^{1/3}$,
$$\P\left(X > \mu r + \theta r^{1/3} - 0.5c_1s^2r^{1/3}, A_j\right) \leq \exp\left(-c(\min(\theta^{3/2}, \theta r^{1/3})+s^3+j^3)\right).$$
\end{lemma}

\begin{lemma}\label{l.bound on A}
In the notation of Proposition \ref{p.p2l general upper tail} and under Assumptions \ref{a.limit shape assumption} and \ref{a.one point assumption}, there exist  constants $c>0$, $c_1>0$, $\theta_0$, $s_0$, and $r_0$ such that, for $r>r_0$, $s>s_0$, and $\theta_0<\theta < r^{2/3}$,
$$\P\left(X > \mu r +\theta r^{1/3} +\theta r^{1/3} - 0.5c_1s^2r^{1/3}, A\right) \leq \exp\left(-c(\min(\theta^{3/2}, \theta r^{1/3})+s^3)\right).$$
\end{lemma}

	Before proving these lemmas, we show how the proof of Proposition~\ref{p.p2l general upper tail} is completed using them. Lemma \ref{l.bound on A_j} says that each individual summand in the first term of \eqref{e.p2l tf breakup} is bounded by $\exp\left(-c(\min(\theta^{3/2}, \theta r^{1/3})+s^3+j^3)\right)$, while Lemma \ref{l.bound on A} says that the second term is bounded by $\exp\left(-c(\min(\theta^{3/2}, \theta r^{1/3})+s^3)\right)$. So we have, by summing over $j$,
	$$\P\left(X > \mu r + \theta r^{1/3}-0.5c_1s^2r^{1/3}\right) \leq C\exp\left(-c(\min(\theta^{3/2}, \theta r^{1/3})+s^3)\right)$$
	for some $C<\infty$. Reducing the value of $c$ completes the proof.
\end{proof}

In seeking to prove Lemmas \ref{l.bound on A_j} and \ref{l.bound on A}, we wish to show that when the endpoint of a particular geodesic is sufficiently extreme, it suffers a weight loss with high probability. Lemma \ref{l.bound on A_j} addresses the case that the slope between the points is bounded away from $0$ and $\infty$, while Lemma \ref{l.bound on A} addresses when the slope between the points is extreme.%

\begin{proof}[Proof of Lemma~\ref{l.bound on A_j}]
	We fix $j$. We divide $\lineleft$ into segments of size at most $r^{2/3}$ each, indexed by~$i$ as $\lineleft^i$. Thus there are at most $t$ segments.

	In the notation of Proposition~\ref{l.sup tail}, we have that $z$ is bounded uniformly away from $r$, and so we may bound $X$ on $A_j$ by using Assumption~\ref{a.limit shape assumption} and Proposition~\ref{l.sup tail}. We set $c_1 = G_2/3$ with $G_2$ as in Assumption~\ref{a.limit shape assumption}. Then,
	\begin{align*}
	\P\Big(X > \mu r +\theta r^{1/3}-0.5c_1s^2r^{1/3}, A_j\Big)
	&\leq  \P\bigg(\sup_{\substack{y\in\lineleft,\\ w\in\mathbb L_j}} X_{y \ar w} > \mu r +\theta r^{1/3}-0.5c_1s^2r^{1/3}\bigg)\\
	&\leq \sum_{i=1}^t\P\bigg(\sup_{\substack{y\in\lineleft^i,\\ w\in\mathbb L_j}} X_{y \ar w} > \mu r +\theta r^{1/3}-0.5c_1s^2r^{1/3}\bigg).
	\end{align*}
	Each summand in the last quantity is in turn bounded by
	\begin{align*}
	\MoveEqLeft[22]
	\P\Bigg(\smash{\sup_{\substack{y\in\lineleft^i,\\ w\in\mathbb L_j}}} \bigg(X_{y \ar w}-\mu r - \frac{1}{3}G_2(s+j)^2r^{1/3}\bigg) >  \theta r^{1/3}-0.5c_1s^2r^{1/3}+ \frac{1}{3}G_2(s+j)^2r^{1/3}\Bigg)\\
	&\leq  \exp\left(-c(\min(\theta^{3/2}, \theta r^{1/3})+s^3+j^3)\right).
	\end{align*}
	The last inequality is via the first part of Proposition~\ref{l.sup tail}, and holds for both (i) $s=0$ and $\theta$ sufficiently large, as well as for (ii) $\theta=0$ and $s>s_0$ (with the earlier mentioned assumption that $s\leq r^{1/3}$). Using that $t\leq s^2$ and reducing the value of $c$ complete the proof.
\end{proof}

\begin{proof}[Proof of Lemma~\ref{l.bound on A}]
	Here the endpoints are not bounded uniformly away from the coordinate axes, and so we will make use of the second part of Proposition~\ref{l.sup tail}, which yields
	\begin{align*}
	\P\left(X > \mu r  +\theta r^{1/3} - 0.5c_1s^2r^{1/3}, A\right)
	&\leq \exp\left(-c (r+ \min(\theta^{3/2}, \theta r^{1/3}))\right)\\
	& \leq \exp\left(-0.5c(r + \min(\theta^{3/2}, \theta r^{1/3})+s^3)\right).
	\end{align*}
 	The last inequality is from the fact that $s\leq r^{1/3}$ (as otherwise the statement is trivial) and $\theta < r^{2/3}$, and again holds for both choices of parameters (i) and (ii) in Proposition~\ref{p.p2l general upper tail}. %
\end{proof}

\subsection{Lower tail and mean of constrained point-to-point}\label{app.constrained lower tail}

\begin{proof}[Proof of Proposition~\ref{p.constrained lower tail}]
	We first derive the lower tail bound \eqref{e.constrained lower tail}. Fix $J = \theta^{1/2}/\ell$ and define $u_j = \left(J^{-1}\cdot j(r-z-hr^{2/3}), J^{-1}\cdot j(r+z+hr^{2/3})\right)$ for $j=0,\ldots, J$. By the stationarity of the random field and the union bound, we have
	\begin{align*}
	\P\left(X_{u \ar u^{\prime}}^{U} \leq \mu r-\theta r^{1 / 3}\right)
	&\leq J\cdot\P\left(X_{u \ar u_1}^{U} \leq \frac{1}{J}\mu r -\frac{\theta}{J} r^{1 / 3}\right).
	\end{align*}
	We also have
	\begin{align*}
	\P\left(X_{u \ar u_1}^{U} \leq \frac{1}{J}\mu r -\frac{\theta}{J} r^{1 / 3}\right) &\leq \P\left(X_{u \ar u_1} \leq \mu \frac{r}{J}-\frac{\theta}{J^{2/3}}\cdot\left(\frac{r}{J}\right)^{1 / 3}\right)\\
	&\qquad + \P\left(X_{u \ar u_1} > \mu \frac{r}{J}-\frac{\theta}{J^{2/3}}\cdot \left(\frac{r}{J}\right)^{1 / 3}, \tf(\Gamma_{u,u_1})> \ell J^{2/3}\left(\frac{r}{J}\right)^{2/3}\right)\\
	&\leq \exp\left(-c\theta^{3/2}/J\right) + \exp\left(-c\ell^3 J^{2}\right)\\
	&\leq 2\exp\left(-c\ell\theta\right),
	\end{align*}
	for sufficiently large $\theta$ depending on $K$. For the second-to-last inequality we have used Assumption~\ref{a.one point assumption} for the first term and Theorem \ref{t.tf} for the second. The parameters (marked here with tildes to avoid confusion) for the application of Theorem~\ref{t.tf} are as follows: $\tilde r = r/J$, $\tilde t = KJ^{2/3}/J = K/J^{1/3}$, and $\tilde s = \theta^{1/2}/J^{1/3}$. It is easy to check that the conditions of Theorem~\ref{t.tf} are met with these parameter choices for sufficiently large $\theta$. For Theorem \ref{t.tf} to apply we also need $\ell J^{2/3} \geq \tilde s = \theta^{1/2}/ J^{1/3},$ which is satisfied with equality by our choice of $J$. This completes the proof of the lower tail estimate.

	For the lower bound \eqref{e.constrained mean} on $\E[X_{u \ar u'}^U]$, we have
	\begin{align*}
	\E[X_{u \ar u'}^{U}] &= \E[X_{u \ar u'}] - \E[X_{u \ar u'} - X_{u \ar u'}^U]\\
	&\geq \mu r - G_1z^2r^{1/3} - g_1r^{1/3}  - \E[X_{u \ar u'} - X_{u \ar u'}^U],
	\end{align*}
	using the lower bound for the first term from Assumption \ref{a.limit shape assumption}. Note that the second term is the expectation of a positive random variable. We claim that the expectation is bounded above by $Cr^{1/3}$ for some $C<\infty$. This follows from the tail probability formula for expectation:
	\begin{align*}
	r^{-1/3}\E[X_{u \ar u'} - X_{u \ar u'}^U] = \int_0^{\infty}\P\left(X_{u \ar u'} - X_{u \ar u'}^U > tr^{1/3}\right) \, dt.
	\end{align*}
	Bounding the integrand by 
	$
	\P(X_{u \ar u'} > \mu r + 0.5 tr^{1/3}) + \P(X_{u \ar u'}^U < \mu r - 0.5 tr^{1/3})
	$
	and using the bounds from Assumption~\ref{a.one point assumption} and the tail bound from above to show that this expression is integrable, with the bound on the integral being independent of $r$, complete the proof.
\end{proof}

\section{Proof of a crude upper bound under a convexity assumption}\label{app.proof of lemmas}

In this appendix we complete the proof of Lemma~\ref{l.first upper bound on melon weight} under the upper tail convexity hypothesis, Assumption~\ref{a.one point assumption convex}.

\begin{proof}[Proof of Lemma~\ref{l.first upper bound on melon weight} under Assumption~\ref{a.one point assumption convex}]
As in the earlier proof of Lemma~\ref{l.first upper bound on melon weight}, for any fixed $k$-geodesic watermelon $\mc W$, let $\smash{X_{n,\mc W}^{k,j}}$ be the weight of its $j$\textsuperscript{th} heaviest curve. For $J\subseteq \intint{k}$ and $\overline \delta = (\delta_1,\ldots,\delta_k)$ such that $\delta_j \geq 0$ for each $j\in J$ and $\delta_j = 0$ for $j\in \intint{k}\setminus J$, define
$$A_{\overline \delta, J} = \left\{\exists \mc W: X_{n, \mc W}^{k,j} > \mu n + \delta_j n^{1/3} \text{ for all }j\in J\right\}.$$
We claim that if $\sum_{j\in J}\delta_j = \lceil tk^{5/3}\rceil$, then, for an absolute constant $c>0$ and $c_1 = c_1(t)>0$, and for $k\leq t^{-3/2}n$,
\begin{equation}\label{e.P(A_delta) bound}
\P(A_{\overline \delta, J}) \leq \exp\left(-ct^{3/2}k^2\right).
\end{equation}
This follows from the BK inequality and Assumption~\ref{a.one point assumption convex}, which gives that
\begin{align*}
\P\left(A_{\overline \delta, J}\right)
&\leq \P\Big(\exists \text{ disjoint paths } \{\gamma_j: j\in J\} \text{ with } \ell(\gamma_j) \geq \mu n + \delta_j n^{1/3} \ \ \forall i=j\in J\Big)\\
&\leq \exp\bigg(-\sum_{j\in J} I_n(\delta_j)\bigg).
\end{align*}
Since $I_n$ is convex and $\sum_{j\in J}\delta_j = \lceil tk^{5/3}\rceil$, we get
\begin{align*}
\sum_{j\in J} I_n(\delta_j) \geq k I_n\bigg(\frac{1}{k}\sum_{j\in J} \delta_j\bigg)
= k I_n\left(k^{-1}\lceil tk^{5/3}\rceil\right)
&\geq c'\min(t^{3/2} k^2, tk^{5/3}n^{1/3})\\
&\geq c't^{3/2}k^2 
\end{align*}
as $I_n(x) \geq c'\min(x,x^{3/2})$ for all $x\geq 0$ from Assumption~\ref{a.one point assumption convex}, and $k\in\intint{\lfloor \min(1,t^{-3/2})n\rfloor}$. Let $D$ be the set of $\overline \delta$ defined by
$$D = \left\{ \overline\delta: \sum_{j=1}^k\delta_j = \lceil tk^{5/3}\rceil, \delta_j \in \N\cup\{0\} \ \ \forall j\in\intint{k}\right\}.$$
Now we observe
$\P\left(\melonweight > \mu nk + tk^{5/3}n^{1/3}\right) = \P\big(\exists\mc W: \sum_{j=1}^k (X_{n,\mc W}^{k,j} -\mu n)n^{-1/3} > tk^{5/3}\big).$
By considering the ceiling of each summand in the event of this last probability, we obtain that the last expression is bounded by
\begin{align*}
\MoveEqLeft[8]
\P\left(\exists\mc W: \sum_{j=1}^k \lceil(X_{n, \mc W}^{k,j} -\mu n)n^{-1/3}\rceil > tk^{5/3}\right)\\
&\leq\P\left(\bigcup_{\overline\delta\in D}\left\{\exists\mc W\text{ and } J\subseteq \intint{k}: \lceil(X_{n, \mc W}^{k,j} -\mu n)n^{-1/3}\rceil > \delta_j  \ \ \forall j\in J\right\}\right)\\
&\leq 2^k\cdot|D|\cdot\exp\left(-ct^{3/2}k^2\right),
\end{align*}
where the first inequality is seen by considering $J$ to be the set of indices $j$ such that $\lceil(X_{n, \mc W}^{k,j} -\mu n)n^{-1/3}\rceil > 0$ and taking $\delta_j$ to be an integer at most the latter quantity such that $\sum_{j\in J} \delta_j = \lceil tk^{5/3}\rceil$ for $j\in J$, and $\delta_j = 0$ for $j\in\intint{k}\setminus J$. The second inequality follows by the union bound; the bound \eqref{e.P(A_delta) bound} on $\P\left(A_{\bar\delta}\right)$; and the cardinality of the number of subsets of $\intint{k}$. The cardinality of $D$ is the number of non-negative solutions to $x_1+\ldots+x_k = tk^{5/3}$, which is $\binom{tk^{5/3}+k-1}{k-1} \leq \binom{2(tk^{5/3}\vee k)}{k}$. Using the trivial bound $\binom{n}{k}\leq n^k$ and simplifying, we obtain that $|D| \leq \exp(\frac12 ct^{3/2}k^2)$ for $k$ larger than some absolute $k_0$ depending on $t$, so the proof is complete by reducing the value of $c$.
\end{proof}

\end{document}